\documentclass[10pt]{article}

\usepackage{amsmath,amsfonts,amsthm,amssymb,mathrsfs}
\usepackage[T1]{fontenc}
\usepackage[utf8]{inputenc}
\usepackage{natbib} 
\usepackage[colorlinks,citecolor=blue,urlcolor=blue,backref]{hyperref}
\usepackage{xcolor}
\usepackage{fullpage}
\usepackage{mathtools}
\usepackage{authblk}

\RequirePackage{graphicx,microtype}
\RequirePackage[capitalize]{cleveref}

\usepackage{caption}
\usepackage{subcaption}

\theoremstyle{plain}

\newtheorem{theorem}{Theorem}[section]
\newtheorem{lemma}[theorem]{Lemma}
\newtheorem{corollary}[theorem]{Corollary}
\newtheorem{proposition}[theorem]{Proposition}
\theoremstyle{remark}   
\newtheorem{definition}[theorem]{Definition}
\newtheorem{remark}[theorem]{Remark}
\numberwithin{equation}{section}

\newtheorem{assumption}{Assumption}
\newtheorem{subassumption}{Assumption\normalfont}[assumption]

\newenvironment{assumption*}
 {\ifnum\value{subassumption}=0 \stepcounter{assumption}\fi\subassumption}
 {\endsubassumption}
\newenvironment{assumption+}[1]
 {\renewcommand{\thesubassumption}{#1}\subassumption}
 {\endsubassumption}
\usepackage{comment}
\DeclareMathOperator*{\argmin}{arg\,min}

\newcommand{\eps}{\ensuremath{\varepsilon}}
   \newcommand{\R}{\mathbb{R}}

\renewcommand{\t}{^{\top}}
\newcommand{\inv}{^{-1}}
\newcommand{\E}{\mathbb{E}}
\newcommand{\p}{\mathbb{P}}
\newcommand{\ipq}{\mathbf{I}_{p,q}}
\usepackage{tikz-cd}
\usepackage{mathtools}
\usepackage{algorithmic}
\usepackage{algorithm}
\newcommand{\wx}{\mathbf{W_*^X}}
\newcommand{\wy}{\mathbf{W_*^Y}}
\newcommand{\tr}{\textnormal{Tr}}
\newcommand{\supp}{\textnormal{supp}}
\newcommand{\aone}{\mathbf{A}^{(1)}}
\newcommand{\atwo}{\mathbf{A}^{(2)}}
\newcommand{\xhat}{\mathbf{\hat X}}
\newcommand{\yhat}{\mathbf{\hat Y}}
\newcommand{\xtilde}{\mathbf{\tilde X}}

\newcommand{\ytilde}{\mathbf{\tilde Y}}
\newcommand{\one}{^{(1)}}
\newcommand{\two}{^{(2)}}

\title{Nonparametric two-sample hypothesis testing for low-rank
random graphs of differing sizes}
\author{Joshua Agterberg, Minh Tang, and Carey Priebe}
\date{\today}

\begin{document}

\maketitle

\begin{abstract}
Given two networks of differing sizes, it is of interest to test whether the two networks belong to the same distribution.  We formalize the notion of ``equality of distribution'' under the framework of the generalized random dot product graph, which considers as special cases a number of popular network  models with low-rank expectations.  We then propose a nonparametric two-sample test statistic to conduct this test, assuming only that the networks have independent edges generated from low-rank probability matrices.  
Our proposed test statistic involves using the maximum mean discrepancy applied to suitably rotated rows of a graph embedding, 
where the rotation is estimated using optimal transport.  We show that our test statistic, appropriately scaled, is consistent for sufficiently dense graphs, and we study its convergence under different sparsity regimes, and our results are demonstrated in numerical simulations.  
\end{abstract}





\section{Introduction}
\label{sec:introduction}


Network data arises naturally in several fields, including neuroscience \citep{bullmore_complex_2009,bullmore_brain_2011,vogelstein_graph_2013,finn_functional_2015,priebe_two-truths_2019,arroyo-relion_network_2019} and social networks \citep{newman_random_2002,newman_modularity_2006,carrington_models_2005} among others. With the introduction of network data in the various sciences, there is a need for developing corresponding statistical methodology and theory.    
Often one wishes to determine whether or not two graphs exhibit similar distributional properties for some notion of similarity between distributions on networks.  Furthermore, as in classical statistics, one may wish to analyze graph data with only a few assumptions on the probability distributions. For example, for Euclidean data, given i.i.d.  observations $\{X_i\}_{i=1}^n $ and  $\{Y_j\}_{j=1}^m\in \R^d$ with cumulative distribution functions denoted $F_X$ and $F_Y$ respectively, a model-agnostic way to test whether $F_X = F_Y$ is use nonparametric methods, such as \citet{mood_asymptotic_1954,anderson_two-sample_1994,gretton_kernel_2012,szekely_energy_2013}, and \citet{chen_new_2017}.

For the two-sample test we consider, we take the perspective that a single network constitutes an observation. 
A number of works have studied hypothesis testing when the graphs are on the same vertex set (assuming the vertex matching is known \emph{a priori}) \citep{tang_semiparametric_2017,li_two-sample_2018,levin_central_2019,ghoshdastidar_two-sample_2020,draves_bias-variance_2020}, analogous to the ``matched pairs'' paradigm in Euclidean data.  However, in many settings, there is not necessarily an \emph{a priori} matching between vertices; for example, one could introduce and remove vertices without fundamentally altering the network structure. 

Explicitly, we consider observing two adjacency matrices $\aone \in \{0,1\}^{n\times n}$ and $\atwo \in \{0,1\}^{m\times m}$, and we consider testing the null hypothesis:
\begin{align*}
    H_0: \text{The distribution of $\aone$ is equal to the distribution of $\atwo$},
\end{align*}
where we only assume that $\mathbb{E} \aone$ (and $\E \atwo$) are low rank (see \cref{sec:preliminaries} for a formal description) and that the edges are drawn independently. By considering the latent-space model \citep{hoff_latent_2002} introduced in \citet{rubindelanchy_statistical_2022}, wherein each vertex has a latent vector in low-dimensional Euclidean space associated to it, we are able to consider testing whether two networks are ``equal in distribution,'' where our definition of equality of distribution is specific enough to allow for a meaningful notion of similarity between graphs on different vertex sets but general enough to allow for arbitrary low-rank graphs.  Under this framework, the model we consider herein includes stochastic blockmodels \citep{holland_stochastic_1983}, degree-corrected stochastic blockmodels, mixed-membership stochastic blockmodels, random dot product graphs \citep{athreya_statistical_2018}, and finite-rank graphons \citep{lovasz_large_2012} as special cases.  
Previous tests have primarily focused on specific submodels, such as the stochastic blockmodel or its variants \citep{lei_goodness--fit_2016,bickel_hypothesis_2016,fan_simple_2022}, or operate on general edge probability matrices \citep{ghoshdastidar_two-sample_2017}.  A more thorough discussion of related literature is in Section \ref{sec:previousresults}.

To discuss our main results, consider observing two networks on $n$ and $m$ vertices respectively, each with average expected degree of order $n\alpha_n$ and $m\beta_m$ respectively, where $\alpha_n,\beta_m \to 0$ as $n,m\to\infty$ or $\alpha_n,\beta_m \equiv  1$.  Our main contributions can be described as follows. 
\begin{itemize}
\item We consider testing whether two networks have equal distribution, and we formalize the notion of ``equality of distribution'' for networks with low-rank edge probability matrices.  
    \item We demonstrate that a modification of the test statistic considered in \citet{tang_nonparametric_2017} is consistent under the null and alternative for sparse networks (with sparsity $n\alpha_n \gg \textcolor{black}{\log^4(n)}, 
    m\beta_m \gg \textcolor{black}{\log^4(m)}$) 
    whose edge probability matrices may have negative and/or close eigenvalues (\cref{thm:consistency_repeated} and \cref{thm:centering}) in an asymptotic regime where each network is of comparable sparsity ($\frac{m\beta_m}{n\alpha_n + m\beta_m} \to \rho \in (0,1)$). In contrast to \citet{tang_nonparametric_2017}, our results hold for \emph{any} low-rank random graph models (satisfying suitable regularity conditions).
    \item We then show that the modified test statistic is consistent in an asymptotic regime where each network is sufficiently dense (with average expected degree growing faster than $\sqrt{n}\log(n)$), but the number of vertices are comparable $(\frac{m}{m+n} \to \rho \in (0,1)).$ 
        \item To compute the modified test statistic we provide an optimal transport algorithm, and we show that our algorithm converges when initialized sufficiently close to the global optimum.
    \item We further demonstrate a sort of ``uniform consistency'' of our test statistic under various fixed alternative hypotheses (see \cref{prop3} for details).
    \item We demonstrate the utility of our algorithm with simulated data.
\end{itemize}

The most similar to our work is \citet{tang_nonparametric_2017}, in which the
authors consider testing if two random dot product graph networks
have the same distribution. However, the theory developed in \citet{tang_nonparametric_2017} does not address the sparse network setting, nor does
it allow for negative eigenvalues or repeated eigenvalues of the edge
probability matrix and thus precludes testing under several commonly
used random graph models such as the balanced K-blocks SBM. In
this work we removed these limitations of \citet{tang_nonparametric_2017}. In particular we substantially extend their framework to address the indefinite
geometry induced by the negative eigenvalues, and our proposed test
procedure and its theoretical properties depend on a careful analysis of
the interplay between indefinite orthogonal transformations, optimal
transport, and convergence of degenerate U -statistics.  See \cref{sec:algorithm} and \cref{sec:previousresults} for more detailed discussions.

The paper is organized as follows. In 
Section \ref{sec:preliminaries}, we give the relevant definitions and describe our setting.  In Section \ref{sec:theory} we state our main  theoretical results for sparse, indefinite random graphs with negative and repeated eigenvalues and describe a modification to handle repeated eigenvalues.  In Section \ref{sec:sims} we show our results on simulated data, and in Section \ref{sec:discussion} we discuss our results.  Section \ref{sec:proofs} contains the proofs of our main results.  
\\ \ \\
\noindent
\textbf{Notation:} We use capital letters to denote random vectors $X \in \R^d$, bold  lowercase letters to denote fixed vectors, and bold capital letters for fixed or random matrices (which will be clear from context).  The distribution of a random vector $X$ will be denoted by $F_X$, and for $X_1, \dots, X_n$ i.i.d. some distribution $F_X$, we use $\mathbf{X}$ to denote the $n \times d$ matrix with its rows the vectors $X_1, \dots, X_n$.   In many occasions, given $X_1, \dots, X_n$ i.i.d. $F_X$, we let $X$ denote a realization from $F$ that is independent from $\{X_i\}_{i=1}^{n}$. We write $\|\cdot\|$ for the usual Euclidean norm on vectors and the spectral norm on matrices and $\|\cdot \|_F$ for the Frobenius norm.  For a matrix $\mathbf{M}$ \textcolor{black}{we write its trace as $\mathrm{Tr}(\mathbf{M})$} and its $\ell_2$ to $\ell_\infty$ operator norm via $\|\mathbf{M}\|_{2,\infty} \equiv \max_{i} \|\mathbf{M}_{i \cdot}\|$, where $\mathbf{M}_{i \cdot}$ are the rows of $\mathbf{M}$.  We use $\mathbf{I}_k$ to denote the $k\times k$ identity matrix.  For a matrix $\mathbf{M}$, the operator $\text{diag}(\mathbf{M})$ extracts its diagonal as a matrix, and for two matrices $\mathbf{M}$ and $\mathbf{N}$, the operator $\text{bdiag}(\mathbf{M,N})$ constructs the block-diagonal matrix $\begin{pmatrix}\mathbf{M} &0 \\0 & \mathbf{N} \end{pmatrix}.$ \textcolor{black}{For two nonnegative functions $f(n)$ and $g(n)$, } we write $f(n) = O(g(n))$ if there exists a constant $C >0$ such that $f(n) \leq C g(n)$ for all $n$ sufficiently large, and $f(n) = \omega(g(n))$ if there exists  a constant $c$ such that $cg(n) \leq f(n)$ for all $n$ sufficiently large.  We also write $f(n) \gg g(n)$ if $g(n)/ f(n) \to 0$ as $n\to\infty$.

\section{Preliminaries}
\label{sec:preliminaries}
In order to develop our hypothesis test we use the latent-space framework motivated by \citet{rubindelanchy_statistical_2022} to compare distributions between networks of different sizes.  
First, we discuss the notion of a $(p,q)$ admissible distribution, \textcolor{black}{which is crucial for our formulation of the hypothesis test}.  In what follows, the matrix $\ipq$ is defined as $\ipq:= \text{bdiag}(\mathbf{I}_p, -\mathbf{I}_q)$.
\begin{definition}
We say that $F_X$ with compact support $\Omega \subset \R^d$ is a $(p,q)$ admissible distribution if for all $\mathbf{x,y} \in \Omega $, $\mathbf{x\t I}_{p,q}\mathbf{ y} \in [0,1]$.
\end{definition}

\noindent
For a fixed $(p,q)$ admissible distribution, we consider the generalized random dot product graph as follows.

\begin{definition}[\citet{rubindelanchy_statistical_2022}] \label{def:grdpg} We say a graph $\mathbf{A} \in \{0,1\}^{n\times n}$ is a \emph{generalized random dot product graph on $n$ vertices} with $(p,q)$-admissible distribution $F_X$, sparsity factor $\alpha_n \textcolor{black}{\in (0,1]}$, and latent positions $\mathbf{X}$ if the matrix $\mathbf{A}$ is symmetric, and the entries $\mathbf{A}_{ij}$ are conditionally independent given $\mathbf{X}$ and Bernoulli random variables with $$\p(\mathbf{A}_{ij} = 1 | \mathbf{X}) = \alpha_n X_i\t \mathbf{I}_{p,q} X_j$$ with $\mathbf{A}_{ij} = \mathbf{A}_{ji}$, and $X_1, \dots, X_n \sim F_X$ are i.i.d.  We write $(\mathbf{A},\mathbf{X}) \sim GRDPG( F_X,n,\alpha_n)$.  
\end{definition}

The introduction of the matrix $\ipq$ is to model large in magnitude negative eigenvalues of the adjacency matrix.  We denote $d := p + q$ as the dimension of the latent Euclidean space associated to $F_X$.  Observe that
\begin{align*}
    \mathbb{E} \mathbf{A} | \mathbf{X} = \alpha_n \mathbf{X} \ipq \mathbf{X}\t,
\end{align*}
and hence that $\mathbb{E}\mathbf{A}| \mathbf{X}$ is rank at most $d$.  The GRDPG model allows for arbitrary low-rank models, so is sufficiently agnostic to provide a meaningful setting for inference. In a nonparametric setting, the parameters for the GRDPG model are simply the signature $(p,q)$, the sparsity parameter $\alpha_n$, and the distribution $F_X$ (which may or may not be parametric).  For this work, we assume the signature $(p,q)$ is known.  

There are a number of low-rank network models, all of which are GRDPGs:
\begin{itemize}
    \item \textbf{Stochastic Blockmodels}: The stochastic blockmodel \citep{holland_stochastic_1983} posits that each vertex $i$ belongs to one of $K$ communities, with $\p(\tau(i)=k) = \pi_k$, where $\tau(i): [n] \to [K]$ is the \emph{membership function} specifying the membership of each vertex.  Once communities are determined, edges are formed according to $\p(\mathbf{A}_{ij} = 1| \tau ) = \mathbf{B}_{\tau(i),\tau)j)}$, where $\mathbf{B} \in [0,1]^{K\times K}$ is a symmetric edge connectivity matrix.  
        To reformulate this model as a GRDPG, let $\mathbf{B} = \mathbf{VDV}\t$ denote its eigendecomposition, and let $F_X$ denote the distribution supported on the rows of $\mathbf{V |D|}^{1/2}$, with $\p(X_i = (\mathbf{V|D|}^{1/2}\big)_{k\cdot}) = \pi_k$. Let $p$ and $q$ denote the number of positive and negative eigenvalues of $\mathbf{B}$.  It is straightforward to check that this results in the equivalent distribution.
    \item \textbf{Degree-Corrected Stochastic Blockmodels}: Similar to the stochastic blockmodel, the Degree-Corrected Stochastic Blockmodel (DCSBM)\citet{karrer_stochastic_2011} posits that there are $K$ unique communities in addition to \emph{degree correction parameters} for each vertex, denoted $\theta_i$. Given  the $\theta_i$'s, the edge probabilities are determined by $\p(\mathbf{A}_{ij} = 1 | \{\theta_i\},\tau) = \theta_i \theta_j \mathbf{B}_{\tau(i)\tau(j)}$, where $\mathbf{B},\tau$ are as above.  
    Similar to the stochastic blockmodel, $F_X$ can be formed by providing a distribution on scalars to generate the $\theta_i$'s, and then assigning probability to latent positions defined via the eigendecomposition of $\mathbf{B}$ as above.  Alternatively, any distribution on the \emph{rays} with endpoints   
    defined by the rows of the eigendecomposition of $\mathbf{B}$ as in the previous example can yield an appropriate $F_X$.
    \item \textbf{Mixed-Membership Stochastic Blockmodels}:  This model relaxes the assumption that latent vectors belong to discrete communities, and instead requires only that total community memberships sum to unity \citet{airoldi_mixed_2008}.  Explicitly, given community membership vectors $\eta_i, \eta_j \in [0,1]^{K}$ such that $\sum_{k} (\eta_i)_k = 1$, edge probabilities are given by $\mathbb{P}(\mathbf{A}_{ij} = 1| \eta) = \eta_i\t \mathbf{B} \eta_j$.  Similar to the previous examples, a distribution $F_X$ can be obtained by allowing any distribution on the convex hull of the rows defined via the eigendecomposition of $\mathbf{B}$.  
    \item \textbf{Popularity-Adjusted Stochastic Blockmodels}: The popularity-adjusted stochastic blockmodel (PABM) \citep{sengupta_block_2018}  generalizes the degree-corrected stochastic blockmodel to allow for \emph{popularity parameters} for a given vertex within each community.  The $K$-block PABM is a special case of a GRDPG where $F_X$ is a mixture distribution of $K$-dimensional linear subspaces of $\mathbb{R}^{K^2}$ see \citet{koo_popularity_2023} for more details.   
    \item \textbf{Random Dot Product Graphs}: The random dot product graph \citet{athreya_statistical_2018} is simply a special case of the GRDPG with $p = d$ and $q = 0$.  
    \item \textbf{Finite-Rank Graphons}: A graphon is a symmetric function $f: \mathcal{X} \times \mathcal{X}\to [0,1]$, where $\mathcal{X}$ is some latent Euclidean space.  We say $f$ is finite-rank if its corresponding integral operator on $\mathcal{X}$ is of finite rank.  On can straightforwardly associate a finite-rank graphon to a GRDPG by considering the (functional) eigendecomposition of $f$; see \citet{lei_network_2020} for details.  
\end{itemize}
    Given two networks from any of the models above, the framework considered herein can test whether each network has the same distribution, even if one of the networks is a submodel of the other (e.g. the first network is an instantiation of a stochastic blockmodel and the second is an instantiation of a degree-corrected stochastic blockmodel). Therefore, a key feature of our proposed test is that it is \emph{universally consistent} for the class of rank $d$ network models.  

One potential issue with working with the GRDPG model is that it exhibits nonidentifiability.  To be explicit, suppose $\mathbf{Q}$ is a matrix such that $\mathbf{Q}\ipq \mathbf{Q}\t = \ipq$ (this is known as the \emph{indefinite orthogonal group} $\mathbb{O}(p,q)$).  Define the distribution $\tilde F_X := F_X \circ \mathbf{Q}$, where $F_X \circ \mathbf{Q}$ means that one generates $X_i \sim F_X$ and then left multiplies the vectors $X_i$ by $\mathbf{Q\t}$.  Then the probabilities of each edge are fixed since
\begin{align*}
    \p_{\tilde F_X} (\mathbf{A}_{ij} =1 | \mathbf{X} ) &= \alpha_n (\mathbf{Q\t} X_i)\t \ipq (\mathbf{Q}\t X_j) \\
    &= \alpha_n X_i\t (\mathbf{Q}\ipq \mathbf{Q\t}) X_j = \alpha_n X_i\t \ipq X_j = \p_{F_X} (\mathbf{A}_{ij} =1 | \mathbf{X} ).
\end{align*}
Hence, the distribution of the graph remains unchanged if one transforms the support of $F_X$ by any indefinite orthogonal transformation $\mathbf{Q}$. Therefore, any nonparametric test of equality of distribution must allow equality up to indefinite orthogonal transformations. This motivates the following definition.

\begin{definition}
Let $F_X$ and $F_Y$ be two $(p,q)$ admissible distributions.  We say $F_X$ and $F_Y$ are \emph{equivalent up to indefinite orthogonal transformation} if there exists a matrix $\mathbf{Q} \in \mathbb{O}(p,q)$ such that
\begin{align*}
    F_X = F_Y \circ \mathbf{Q}.  
\end{align*}
In this case, we write $F_X \simeq F_Y$.
\end{definition}


We are now ready to formally describe our hypothesis test under the generalized random dot product graph framework.
Suppose we observe two graph adjacency matrices $\aone$ and $\atwo$ such that $(\aone,\mathbf{X}) \sim GRDPG( F_X,n,\alpha_n)$ and $(\atwo,\mathbf{Y}) \sim GRDPG(F_Y,m,\beta_m)$ are mutually independent and have the same signature $(p,q)$.  We consider the hypothesis test
\begin{align*}
   &H_0: F_Y \simeq F_X  \\
   &H_A: F_Y \not \simeq F_X.
\end{align*}
Again, we assume throughout that $(p,q)$ is known and fixed in $n$ and $m$. A subtle point is that we assume that $(p,q)$ is the same for each graph -- if not, we simply reject the null hypothesis.  In addition, in general we do not assume $(\alpha_n,\beta_m)$ are known, but, for ease of exposition we shall first assume that they are known.  They can be estimated consistently, so we will revisit these issues later (see Corollary \ref{cor:scaling_estimate}).

In light of the generality of the hypothesis $F_Y \simeq F_X$, the additional nonidentifiability associated to the GRDPG requires carefully accounting for indefinite orthogonal matrices in deriving theoretical results, but still maintains a meaningful notion of similarity between networks in practical settings.  In \cref{thm:consistency_repeated}, we will show that this nonidentifiability can be eliminated in the sense that the results will hold up to \emph{orthogonal} (as opposed to \emph{indefinite orthogonal}) transformation.  
However, we emphasize that the proofs of our theoretical results require careful analyses of indefinite orthogonal matrices and their convergence, and are not just simple extensions of previous results.



In practice one observes only the graph adjacency matrix, and therefore must estimate the latent position matrix $\mathbf{X}$.  The statistical properties of the scaled eigendecomposition, referred to as the \emph{adjacency spectral embedding} (ASE), are investigated in \citet{rubindelanchy_statistical_2022}.  The definition is given below.

\begin{definition}[Adjacency Spectral Embedding] \label{def:ase} Suppose $(\mathbf{A},\mathbf{X})\sim GRDPG$ $( F_X,n,\alpha_n)$, and write the eigendecomposition of $\mathbf{A}$ as $\sum_{i=1}^{n} \lambda_i \mathbf{u_i u_i\t}$, where the $\lambda_i$ are ordered by magnitude; $|\lambda_1| \geq |\lambda_2| \geq \dots \geq |\lambda_n|$, and the $\mathbf{u_i}$ are orthonormal.  Form the $d\times d$ matrix $\mathbf{\Lambda_A}$ by taking the $d$ largest (in magnitude) eigenvalues of $\mathbf{A}$ sorted by positive and then negative components, and the $n \times d$ matrix $\mathbf{U_A}$ with columns consisting of the eigenvectors associated to the eigenvalues in $\mathbf{\Lambda_A}$.  The \emph{adjacency spectral embedding} of $\mathbf{A}$ is the $n \times d$ matrix $$\mathbf{\hat X : = U_A |\Lambda_A|^{1/2}},$$ where the operator $|\cdot|$ takes the absolute values of the entries.  
\end{definition}


\subsection{A Kernel Estimator}
\label{sec:gretton}
To describe our test statistic, we must first define \emph{mean embedding} of a distribution. Consider a symmetric positive-definite kernel $\kappa( \cdot, \cdot): \R^d \times \R^d \to \R$ with associated reproducing kernel Hilbert space $\mathcal{H}$.  Let $\mathcal{P}$ denote the set of probability distributions on $\R^d$.  The \textit{mean embedding} of a distribution function $F$ is defined via
\begin{align*}
\mu: \mathcal{P} &\to \mathcal{H} \\
    \mu[F](\cdot) &:= \mathbb{E} \kappa(\cdot, X), 
\end{align*}
\textcolor{black}{where the expectation is taken for the random variable $X$ with distribution function $F$.} Hence $\mu[F](\cdot)$ defines a function in $\mathcal{H}$, so that $\mu[\cdot]$ is a map from probability distributions to $\mathcal{H}$.
%
A kernel $\kappa$ is called \emph{characteristic} \citep{sriperumbudur_universality_2011}
if $\mu[\cdot]$ is injective, so that $F = G$ if and only if $\mu[F] = \mu[G]$. Examples of characteristic kernels include the Gaussian kernel $\kappa(x,y) = \exp\left(- \frac{1}{\sigma^2} \| x-y\|^2 \right)$ and the Laplace kernel $\kappa(x,y) = \exp\left(- \frac{1}{\sigma} \| x - y \|_1\right)$, \textcolor{black}{where $\|\cdot\|_1$ denotes the $\ell_1$ norm on vectors}.  We say $\kappa$ is \emph{radial} if $\kappa(x,y) = \kappa( \mathbf{W} x, \mathbf{W} y)$ for any orthogonal matrix $\mathbf{W}$.

Since $\kappa$ is a function of two variables, given independent samples $\mathbf{X} = \{X_i\}_{i=1}^{n}$ and $\mathbf{Y} = \{Y_j\}_{j=1}^{m}$, we define the (two-sample) $U$-statistic
\begin{align*}
   U_{n, m}(\mathbf{X}, \mathbf{Y}):=\frac{1}{n(n-1)} \sum_{j \neq i} \kappa\left(X_{i}, X_{j}\right)-\frac{2}{m n} \sum_{i=1}^{n} \sum_{k=1}^{m} \kappa\left(X_{i}, Y_{k}\right) + \frac{1}{m(m-1)} \sum_{l \neq k} \kappa\left(Y_{k}, Y_{l}\right).
\end{align*}  
In the asymptotic regime that $n$ and $m$ tend to infinity, under the assumption $\frac{m}{n+m} \to \rho \in (0,1)$,  \citet{gretton_kernel_2012} showed that 
\begin{align*}
    U_{n,m}(\mathbf{X,Y}) \to \| \mu[F_X] - \mu[F_Y]\|_{\mathcal{H}}^2 
\end{align*}
almost surely, and, when $\kappa$ is characteristic, then $\| \mu[F_X] - \mu[F_Y]\|_{\mathcal{H}}^2 = 0$ if and only if $F_X = F_Y$.
  Moreover, they showed that $(n+m)U_{n,m}(\mathbf{X,Y})$ has a nondegenerate limiting distribution under the null hypothesis $F_X = F_Y$, though the limiting distribution is not independent of $F_X$ and $F_Y$ in general.  The scaling by $(n+m)$ is due to the fact that the $U$-statistic is degenerate, where degeneracy of a $U$-statistic with kernel $h$ of two variables means that $\E_{F_X}(h(X,\cdot))$ is constant. 
  See, for example, Chapter 5 of \citet{serfling_approximation_1980} for more details on the theory of degenerate $U$-statistics.    

\section{Consistent Hypothesis Testing} \label{sec:theory}

We now present a detailed asymptotic analysis of our two-sample test statistic. 
Given two graphs $\aone$ and $\atwo$ on $n$ and $m$ vertices respectively our test statistic is defined via
\begin{align*}
    U_{n,m}(\xhat,\yhat) := \frac{1}{n(n-1)} \sum_{j \neq i} \kappa\left(\hat X_{i}, \hat X_{j}\right)-\frac{2}{m n} \sum_{i=1}^{n} \sum_{k=1}^{m} \kappa\left(\hat X_{i}, \hat Y_{k}\right)+\frac{1}{m(m-1)} \sum_{l \neq k} \kappa\left(\hat Y_{k}, \hat Y_{l}\right),
\end{align*}
where $\xhat$ and $\yhat$ are the adjacency spectral embeddings of $\aone$ and $\atwo$ respectively. In what follows, all of our asymptotic results are stated as $n$ and  $m$ tend to infinity. 
We will require some assumptions on the kernel $\kappa$.

\begin{assumption+}{1}\label{assumption:kappa}
The kernel $\kappa$ is characteristic, radial, and twice continuously differentiable on $\R^d$.
\end{assumption+}

The assumption that $\kappa$ is radial is so that our results can be expressed in terms of individual orthogonal matrices that may themselves be products of several orthogonal matrices.  For example, we have the identity $U(\xhat \mathbf{W}_1, \yhat \mathbf{W}_2) = U(\xhat \mathbf{\tilde W}, \yhat)$, where $\mathbf{\tilde W} = \mathbf{W}_1 \mathbf{W}_2\t$.  
Our main results can be modified slightly to hold without this assumption at the penalty of introducing more orthogonal matrices.  Differentiability is a relatively mild requirement, \textcolor{black}{and we note that it implies $\kappa$ is Lipschitz on compact sets, a fact we use several times in our proof.} The assumption of $\kappa$ being characteristic is so that we achieve consistency against all fixed alternatives. 

Since real-world graphs are sparse, we conduct a more thorough study of our test statistic under sparsity.  First, we make assumptions on the sparsity for which our more general results hold.  We implicitly assume that either $\alpha_n,\beta_m \to 0$ or that $\alpha_n \equiv \beta_m \equiv 1$, since if $\alpha_n$ or $\beta_m$ are converging \textcolor{black}{to some constant in $(0,1)$}, one can just rescale the distribution $F_X$ or $F_Y$.

\begin{assumption+}{2a}\label{sparsity1}
The sparsity parameters for the graphs satisfy
\begin{align*}
    \min(n\alpha_n,m\beta_m) \gg  \log^4(\max(n,m)) ,
\end{align*}
 and
 \begin{align*}
     \frac{m\beta_m }{m\beta_m+ n\alpha_n} \to \rho \in (0,1); \qquad m \asymp n.
 \end{align*}
\end{assumption+}
\noindent 
 In \cref{sparsity1} the technical assumption is that $\min(n\alpha_n,m\beta_m)$ grows polylogarithmically in $n$.  Recent work \citep{lei_unified_2019,abbe_entrywise_2020,su_strong_2020,xie_entrywise_2022} has established that $n\alpha_n \geq C \log(n)$ is required for perfect recovery in stochastic blockmodels and random dot product graphs.  While \cref{sparsity1} is slightly suboptimal, our analysis is based on the theory in \citet{rubindelanchy_statistical_2022}, who require that $n\alpha_n \gg \log^4(n)$.  \textcolor{black}{It is likely possible to extend our results down to the optimal regime $n\alpha_n \geq C \log(n)$, but this analysis requires significantly more bookkeeping and is unlikely to yield any useful conceptual or technical ideas.  We leave extending our results to this regime for future work.}
 

If instead we have a slightly denser graph, we make the following assumption.
\begin{assumption+}{2b}\label{sparsity2}
   The sparsity parameters for the graphs satisfy
\begin{align*}
    \min(n\alpha_n,m\beta_m) \gg n^{1/2} \log(n),
\end{align*}
and
 \begin{align*}
     \frac{m }{m+ n} \to \rho \in (0,1).
 \end{align*}
\end{assumption+}

In both asymptotic regimes, there are two competing factors: the first is in the approximation of the \emph{unperturbed} $U$-statistic to the population value; that is, the $U$-statistic obtained given access to the latent vectors $X_1, \dots ,X_n$, and the second is in the approximation of the \emph{estimated} $U$-statistic to the unperturbed $U$-statistic.  In the first asymptotic regime, the primary difficulty stems from the approximation of $\xhat$ and $\yhat$ to $\mathbf{X}$ and $\mathbf{Y}$ (up to indefinite orthogonal transformation) because the graph sparsity makes estimation difficult.  In the second asymptotic regime, the primary difficulty comes from the approximation of the $U$-statistic to the maximum mean discrepancy between two appropriately defined distributions.  The asymptotic regime in Assumption \ref{sparsity2} up to the logarithmic term has been assumed in the literature previously \citep{tang_asymptotically_2017,jones_multilayer_2020}. 



Under the GRDPG framework, one necessarily has to contend with indefinite orthogonal transformations. 
From the equation $\mathbf{Q}\ipq \mathbf{Q\t} = \ipq$ 
we have that $| \text{det}(\mathbf{Q})| = 1$, and hence $\mathbf{Q}$ is invertible and $\mathbf{Q\inv} \in \mathbb{O}(p,q)$ as well.  We also note that the set $\mathbb{O}(p,q)$ includes block-diagonal orthogonal matrices; i.e. if we have $\mathbf{W}_p$ and $\mathbf{W}_q$ for $p \times p$ and $q\times q$ orthogonal matrices, then
\begin{align*}
    \begin{pmatrix} \mathbf{W}_p &0 \\0 & \mathbf{W}_q \end{pmatrix} \ipq  \begin{pmatrix} \mathbf{W}_p\t &0 \\0 & \mathbf{W}_q\t \end{pmatrix} &=  \begin{pmatrix} \mathbf{W}_p\mathbf{W}_p\t &0 \\0 & -\mathbf{W}_q\mathbf{W}_q\t \end{pmatrix} = \ipq.
\end{align*}
We refer to the subgroup $\mathbb{O}(p, q) \cap \mathbb{O}(d)$ as the subgroup of block-orthogonal
matrices.  Note that $ \| \mathbf{Q} \| = 1$ for any block-orthogonal $\mathbf{Q}$, whereas 
for any finite $M > 0$, there exists $\mathbf{Q} \in \mathbb{O}(p,q)\setminus \mathbb{O}(d)$ with $\|\mathbf{Q} \|  > M$. 

Therefore, allowing for negative eigenvalues involves studying matrices $\mathbf{Q} \in \mathbb{O}(p,q)$ that could be badly behaved (in a spectral norm sense). Nevertheless, using the limiting results in \citet{agterberg_two_2020}, by subtly passing between indefinite and Euclidean geometry, we can show that when using the adjacency spectral embeddings, one does not even have to consider indefinite orthogonal matrices.  

In order to present our main results, we require some additional notation.  Define
\begin{align*}
    \mathbf{P}\one := \alpha_n\mathbf{X} \ipq \mathbf{X\t}; \qquad \mathbf{P}\two := \beta_m \mathbf{Y} \ipq \mathbf{Y\t},
\end{align*}
and let $\mathbf{U_X \Lambda_X U_X\t}$ and $\mathbf{U_Y \Lambda_Y U_Y\t}$ be their respective eigendecompositions, with $\mathbf{\Lambda_X}$ and $\mathbf{\Lambda_Y}$ arranged with the $p$ positive eigenvalues first and $q$ negative eigenvalues second.  Define
\begin{align*}
    \mathbf{\tilde X := U_X |\Lambda_X|}^{1/2}; \qquad  \mathbf{\tilde Y := U_Y |\Lambda_Y|}^{1/2}. 
\end{align*}
%
The matrices $\mathbf{\tilde X}$ and $\mathbf{\tilde Y}$ can be viewed as surrogates for the matrices $\alpha_n^{1/2}\mathbf{X}$ and $\beta_m^{1/2}\mathbf{Y}$ up to indefinite orthogonal transformation. 

We have the following result. All proofs are deferred to Section \ref{sec:proofs}.  

\begin{theorem}\label{thm:consistency_repeated}
Suppose Assumptions \ref{assumption:kappa} and \ref{sparsity1} hold.  Then under the null hypothesis $F_X \simeq F_Y$ there exist two sequences of block-orthogonal matrices $\mathbf{\hat W}_n$ and $\mathbf{W}_n$ such that  
\begin{align*}
(m\beta_m + n\alpha_n) \left| U_{n,m}\left( \frac{ \mathbf{\hat X \hat W}_n}{\sqrt{\alpha_n}}, \frac{\mathbf{\hat Y}}{\sqrt{\beta_m}} \right) -  U_{n,m}\left( \frac{\mathbf{\tilde X} \mathbf{W}_n}{\sqrt{\alpha_n}}, \frac{\mathbf{\tilde Y}}{\sqrt{\beta_m}} \right)\right| \to 0 
\end{align*}
almost surely.  
If instead $F_X \not \simeq F_Y$, then almost surely,
\begin{align*}
    \frac{(m\beta_m + n\alpha_n)}{\log(n)} \left| U_{n,m}\left( \frac{ \mathbf{\hat X \hat W}_n}{\sqrt{\alpha_n}}, \frac{\mathbf{\hat Y}}{\sqrt{\beta_m}} \right) -  U_{n,m}\left( \frac{\mathbf{\tilde X} \mathbf{W}_n}{\sqrt{\alpha_n}}, \frac{\mathbf{\tilde Y}}{\sqrt{\beta_m}} \right)\right|   \to 0. 
\end{align*}
\end{theorem}

We emphasize that these results are desirable since the matrices $\mathbf{\hat W}_n$ are block orthogonal matrices and not indefinite orthogonal matrices.  The proofs in Section \ref{sec:proofs} demonstrate that we can effectively bypass the indefinite orthogonal transformations through careful analyses of their convergence, which requires careful tabulation of several block-orthogonal matrices.  Consequently, the only estimation required is the matrix $\mathbf{\hat W}_n$ -- we will return to estimating this orthogonal matrix in Section \ref{sec:algorithm}. In essence, the  orthogonal matrices $\mathbf{\hat W}_n$ and $\mathbf{W}_n$ appear because without distinct eigenvalues, the embeddings $\xhat$ and $\yhat$ need to simultaneously aligned to each other as well as $\xtilde$ and $\ytilde$.   

While the convergence in Theorem \ref{thm:consistency_repeated} shows that the \textcolor{black}{difference between} the $U$-statistic evaluated at $\alpha_n^{-1/2}\xhat\mathbf{\hat W}_n$ and $\beta_m^{-1/2}\yhat$ and the $U$-statistic evaluated at $\alpha_n^{-1/2}\xtilde \mathbf{W}_n$ and $\beta_m^{-1/2}\ytilde$ \textcolor{black}{tends to zero quickly}, this does not immediately imply consistent testing.  The following result quantifies the convergence of the centering term.

\begin{theorem} \label{thm:centering}
For the sequence of block-orthogonal matrices $\mathbf{W}_n$ appearing in Theorem \ref{thm:consistency_repeated}, it holds that
\begin{align*}
  \bigg|  U_{n,m}(\xtilde \mathbf{W}_n/\alpha_n^{1/2}, \ytilde/\beta_m^{1/2}) &- \| \mu[F_X \circ \mathbf{\tilde Q_X\inv}] - \mu[F_Y \circ \mathbf{\tilde Q_Y\inv}] \|_{\mathcal{H}}^2 \bigg| \\&= O\bigg( \sqrt{\frac{\log(n+m)}{n}} + \sqrt{\frac{\log(n+m)}{m}} \bigg),
\end{align*}
with probability at least $1 - O(n^{-2} + m^{-2})$.  Here $\mathbf{\tilde Q_X}$ and $\mathbf{\tilde Q_Y}$ are indefinite orthogonal matrices depending only on $F_X$ and $F_Y$ and the signature $(p,q)$ and are such that
\begin{align*}
    \| \mu[F_X \circ \mathbf{\tilde Q_X\inv}] - \mu[F_Y \circ \mathbf{\tilde Q_Y\inv}] \|_{\mathcal{H}}^2 &= \begin{cases} 0 & F_X \simeq F_Y \\ C > 0 & F_X \not \simeq F_Y.\end{cases}
\end{align*}
\end{theorem}
Consequently, both Theorems \ref{thm:consistency_repeated} and \ref{thm:centering} imply that if one can estimate the matrices $\mathbf{\hat W}_n$ consistently, then we can devise a consistent test procedure through a permutation test and bootstrapping the test statistic distribution. Again, we postpone estimating $\mathbf{\hat W}_n$ to Section \ref{sec:algorithm}.

The convergence in Theorem \ref{thm:consistency_repeated} requires scaling by $m\beta_m + n\alpha_n$.  This scaling is not the same as the scaling in \citet{gretton_kernel_2012} to yield a nondegerate limiting distribution.  However, for sufficiently dense graphs, we have the following result.

\begin{corollary}\label{cor:scaling}
Consider the setting of Theorem \ref{thm:consistency_repeated}, but suppose instead that Assumption \ref{sparsity2} is satisfied.
Under the null hypothesis, we have that
\begin{align*}
(m + n) \left( U_{n,m}\left( \frac{\mathbf{\hat X  \hat W}_n}{\sqrt{\alpha_n}}, \frac{\mathbf{\hat Y}}{\sqrt{\beta_m}} \right) -  U_{n,m}\left( \frac{\mathbf{\tilde XW}_n}{\sqrt{\alpha_n}}, \frac{\mathbf{\tilde Y}}{\sqrt{\beta_m}} \right)\right) \to 0
\end{align*}
almost surely, for the same sequences of orthogonal matrices $\mathbf{\hat W}_n$ and $\mathbf{W}_n$ as in Theorem \ref{thm:consistency_repeated}.  
\end{corollary}

Finally, we note that in general the sparsity factors $\alpha_n$ and $\beta_m$ are not known.  If instead we wish to use the estimated sparsity factors, we have the following result.  

\begin{corollary} \label{cor:scaling_estimate}
Assume that $\E( X_1 \t \ipq X_2) = \frac{1}{2}$ and that $\alpha_n,\beta_m \to 0$. 
Define
\begin{align*}
    \hat \alpha_n := \frac{1}{{{n }\choose 2}} \sum_{i < j} \aone_{ij}; \qquad \hat \beta_m := \frac{1}{{{m}\choose 2}} \sum_{i < j} \atwo_{ij}.
\end{align*}

Then the limiting results in Theorem \ref{thm:consistency_repeated} and Corollary \ref{cor:scaling} all hold under their respective conditions with $\mathbf{\hat X}/\alpha_n^{1/2}$ and $\mathbf{\hat Y}/\beta_m^{1/2}$ replaced with $\mathbf{\hat X}/\hat\alpha_n^{1/2}$ and $\mathbf{\hat Y}/\hat\beta_m^{1/2}$ respectively and the almost sure convergence replaced with  convergence in probability.
\end{corollary}

The condition $\E(X_1\ipq X_2) = \frac{1}{2}$ is 
used only
for identifiability of $\alpha_n$ and $\beta_m$ when they need to be estimated. See
e.g., \citet{lunde_subsampling_2023} for an identical condition in the setting of graphons.  

\subsection{Interpretation}
There are several different alignment matrices that appear in order to show the convergence in Theorem \ref{thm:consistency_repeated}. However, in our analysis we are able to show that only the indefinite orthogonal matrices that are simultaneously orthogonal have any effect on the limiting values.  Given Theorem \ref{thm:centering}, the main results in Theorem \ref{thm:consistency_repeated} further details that under the null hypothesis $F_X \simeq F_Y$ one can perform consistent testing given access to only the graphs $\aone$ and $\atwo$. The results of \citet{gretton_kernel_2012} imply that $(m+n)U_{n,m}(\mathbf{X,Y})$ has a nondegenerate limiting distribution under the null hypothesis. For graphs with average degree growing faster than $n^{1/2}\text{polylog}(n)$, Corollary \ref{cor:scaling} says that the same scaling occurs under the null hypothesis with $\xtilde$ and $\ytilde$ as replacements for $\mathbf{X}$ and $\mathbf{Y}$.  For almost surely dense graphs; i.e. graphs with $\alpha_n = \beta_m = 1$, Theorem \ref{thm:consistency_repeated} also provides a result under the alternative.

Lemma \ref{lem:qtilde} shows that the rate of the approximation of $\xtilde \mathbf{W}_n/\sqrt{\alpha_n}$ and $\ytilde/\sqrt{\beta_m}$ to $\mathbf{X\tilde Q_X\inv}$ and $\mathbf{Y\tilde Q_Y\inv}$ is of order $\sqrt{\log(n)/n}$, which, in general, is not fast enough to guarantee the convergence of $$(n+m)\left(U_{n,m}(\xtilde \mathbf{W}_n/\alpha_n^{1/2},\ytilde/\beta_m^{1/2}) -U_{n,m}(\mathbf{X\tilde Q_X\inv},\mathbf{Y\tilde Q_Y\inv}) \right)$$ to zero. However, Theorem \ref{thm:consistency_repeated} together with Theorem \ref{thm:centering} shows that the $U$-statistic evaluated at $\xhat \mathbf{\hat W}_n$ and $\yhat$ still tends to zero under the null hypothesis and to a constant under the alternative, which guarantees consistent testing.


For testing purposes, the lack of distributional results is of no consequence, since if one can reliably estimate the orthogonal transformation $\mathbf{\hat W}_n$ appearing in Theorem \ref{thm:consistency_repeated} then one can perform consistent testing through a bootstrapped permutation test; see the following section.  Furthermore, the limiting distribution for the maximum mean discrepancy between two distributions $F_X$ and $F_Y$ will not be independent of $F_X$ and $F_Y$ in general, so one may have to use a permutation test to approximate the null distribution anyways.



For sparser graphs, the convergence in Theorem \ref{thm:consistency_repeated} under the null hypothesis requires the scaling $m\beta_m + n\alpha_n$, which, if $n \asymp m$ and $\alpha_n \asymp \beta_m $ is slower than the convergence in \citet{gretton_kernel_2012} by a factor of $\alpha_n$.  The reason for this stems primarily from the fact that for sparse graphs it is much more difficult to estimate $\xtilde$ and $\ytilde$.  The sparsity factor here brings down the effective sample size; one observes only $\Theta(n\alpha_n)$ edges per vertex on average for sparse graphs instead of $\Theta(n)$ edges for dense graphs.  Therefore, though the scaling is slower than that for \cref{cor:scaling}, Theorem \ref{thm:consistency_repeated} still guarantees consistency under both fixed and local alternatives as long as 
\begin{align*}
    \| \mu[F_X \circ \mathbf{\tilde Q_X}\inv] - \mu[F_Y \circ \mathbf{\tilde Q_Y}\inv] \|_{\mathcal{H}}^2 \gg \max\bigg\{ \sqrt{\frac{\log(n)}{n}}, \frac{\log(n)}{n\alpha_n}\bigg\}
\end{align*}
whenever \cref{sparsity1} holds, estimating $\mathbf{\hat W}_n$ not withstanding.

\subsection{Estimating $\mathbf{\hat W}_n$} \label{sec:algorithm}

We note that thus far, we have demonstrated that negative eigenvalues do not affect limiting results despite \emph{a priori} having to consider indefinite orthogonal transformations.  Such a result is desirable, as one does not have to resort to numerical algorithms optimizing over $\mathbb{O}(p,q)$, which could be unstable due to the ill-conditioning inherent in indefinite orthogonal transformations.  Furthermore, we have shown that any modification to our test need  estimate only the matrix $\mathbf{\hat W}_n$ from Theorem \ref{thm:consistency_repeated}.  We now turn our attention to estimating $\mathbf{\hat W}_n$.

First, observe that by Theorem \ref{thm:centering} under the null hypothesis there exist indefinite orthogonal matrices $\mathbf{\tilde Q_X}$ and $\mathbf{\tilde Q_Y}$ depending only on $F_X$, $F_Y$, and the signature $(p,q)$ such that
\begin{align*}
    F_X \circ \mathbf{\tilde Q_X}\inv =  F_Y\circ \mathbf{\tilde Q_Y}\inv,
\end{align*}
which follows from the fact that $\kappa$ is assumed to be a characteristic kernel.  Since the two distributions are equal under the null, we consider estimating $\mathbf{\hat W}_n$ via the machinery of \emph{optimal transport}.  

Let $\mathbf{\hat X}$ and $\mathbf{\hat Y}$ be the adjacency spectral embeddings of $\aone$ and $\atwo$.  One can view a collection of points as a distribution by assigning equal point mass to each point.  Define $\hat F_{\hat X}$ as the empirical distribution for $\mathbf{\hat X}$ and define $\hat F_{\hat Y}$ as the empirical distribution for $\mathbf{\hat Y}$; that is $\hat F_{\hat X}$ places point mass of $\frac{1}{n}$ at each $\hat X_i$, and $\hat F_{\hat Y}$ places point mass of $\frac{1}{m}$ at each $\hat Y_j$. Let $d_2(\cdot,\cdot)$ denote the Wasserstein $\ell^2$ distance between two distributions; that is, given two distributions $F$ and $G$, we define
\begin{align*}
    d_2( F,G) := \inf_{\Gamma_{F,G}}( \E_{(X,Y)\sim\Gamma_{F,G}} \| X - Y\|_2^2 )^{1/2}, 
\end{align*}
where $\Gamma_{F,G}$ is the set of 
distributions whose marginals are $F$ and $G$. The set $\Gamma_{F,G}$ is called the set of \emph{couplings} of $F$ and $G$.  If $F$ and $G$ are empirical distributions on $n$ and $m$ points respectively, the couplings can be represented by matrices whose rows and columns sum to $\frac{1}{m}$ and $\frac{1}{n}$; these are the \emph{assignment matrices}.

In light of Theorems \ref{thm:consistency_repeated} and \ref{thm:centering}, we propose finding the orthogonal matrix $\mathbf{\hat W}_n$ that solves the problem
\begin{align}
    \inf_{\mathbf{W} \in \mathbb{O}(d)\cap \mathbb{O}(p,q)} d_2( \hat F_{\hat X/\hat\alpha_n^{1/2}}, \hat F_{\hat Y/\hat \beta_m^{1/2}} \circ \mathbf{ W}), \label{wasserstein}
\end{align}
where $\hat F_{\hat Y/\hat \beta_m^{1/2}} \circ \mathbf{W}$ is the empirical distribution $\hat F_{\hat Y/\hat \beta_m^{1/2}}$ transformed by an orthogonal matrix $\mathbf{W}$.  The above distance is considered in both \citet{lei_network_2020} and \citet{levin_bootstrapping_2025} as the \emph{orthogonal Wasserstein distance}.  


In order to account for the positive and negative parts separately, we consider optimizing over the top $p$ part and the bottom $q$ part separately.  To be concrete, we consider minimizing the two (separate) problems
\begin{align}
\mathbf{W}_1 :&= \argmin_{\mathbf{W} \in \mathbb{O}(p)} d_2( \hat F_{\hat X/\hat\alpha_n^{1/2}}^{(p)},  \hat F_{\hat Y/\hat\beta_m^{1/2}}^{(p)} \textcolor{black}{\circ \mathbf{W}}); \label{eq:otp} \\
\mathbf{W}_2 :&= \argmin_{\mathbf{W} \in \mathbb{O}(q)} d_2( \hat F_{\hat X/\hat\alpha_n^{1/2}}^{(q)},  \hat F_{\hat Y/\hat\beta_m^{1/2}}^{(q)} \textcolor{black}{\circ \mathbf{W}}), \nonumber
\end{align}
where $\hat F_{\hat X/\hat\alpha_n^{1/2}}^{(p)}$ denotes the marginal distribution of the first $p$ components (with the other terms defined similarly).  We then set 
\begin{align*}
    \mathbf{\hat W} :&= \begin{pmatrix}
    \mathbf{W}_1 & 0 \\
     0 & \mathbf{W}_2
    \end{pmatrix}.
\end{align*}
The problem in expression \eqref{eq:otp} is simultaneously  an optimal transport problem in finding the minimum over couplings and a Procrustes problem in finding the minimum over orthogonal matrices. Define the matrix $\mathbf{C_W} \in \R^{n \times m}$ as the \emph{cost matrix with respect to $\mathbf{W}$} by setting $$\mathbf{(C_W)}_{ij} := \|\mathbf{W} \hat X_i^{(p)} -  \hat Y_j^{(p)}\|^2.$$  Then expression \ref{wasserstein} can be written as 
\begin{align}
    \min_{ \mathbf{W, \Pi}} \langle \mathbf{\Pi,C_W} \rangle \label{mainprob}
\end{align}
where the inner product is the Frobenius (matrix) inner product, $\mathbf{W}$ is a block-orthogonal matrix, and $\mathbf{\Pi}$ satisfies $\mathbf{\Pi}\mathbf{1} = \frac{1}{m} \mathbf{1}$ and $\mathbf{\Pi\t}\mathbf{1} = \frac{1}{n} \mathbf{1}$; that is, $\mathbf{\Pi}$ is an assignment matrix. 

A natural solution to solve the problem \eqref{mainprob} is to alternate between 1) solving for $\mathbf{\Pi}$ (given fixed $\mathbf{W})$, and 2) solving for $\mathbf{W}$.  
The following result demonstrates the properties of the global minimizers of \eqref{mainprob}.
\begin{proposition}
 \label{prop3}
   Suppose that $F_X \simeq F_Y$.  Then with probability at least $1 - O(n^{-2} + m^{-2}),$
   \begin{align*}
       \inf_{\mathbf{W} \in \mathbb{O}(d) \cap\mathbb{O}(p,q)} d_2(\mathbf{\hat X} \mathbf{W}, \mathbf{\hat Y} ) &= O\left( \frac{\log^{1/d}(n)}{n^{1/d}} +\frac{\log^{1/d}(m)}{m^{1/d}} + \frac{\log(n)}{(n\alpha_n)^{1/2}} + \frac{\log(m)}{(m\beta_m)^{1/2}}\right).
   \end{align*}
\end{proposition}
Unfortunately, the procedure of alternating optimization suffers from a computational bottleneck -- the optimal transport problem is an \emph{assignment problem}, and while there exist efficient algorithms to solve any one assignment problem, this procedure requires solving an assignment problem at each iteration, which takes time $O(\max\{n,m\} \times \min\{n,m\}^2)$.  To ameliorate this bottleneck, we propose solving a regularized version of the optimal transport problem, first proposed by \citet{cuturi_sinkhorn_2013}.  Define the auxiliary expression
\begin{align}
\inf_{\mathbf{\Pi, W}} \langle \mathbf{\Pi}, C_\mathbf{W} \rangle + \eps H(\mathbf{\Pi}) \label{prob:sinkhorn}
\end{align}
where $H(\mathbf{\Pi})$ is the entropy of the distribution given by $\mathbf{\Pi}$.  For a fixed $\eps$, Equation \ref{prob:sinkhorn} can be computed efficiently via the Sinkhorn algorithm \citep{cuturi_sinkhorn_2013}. 
We then alternately minimize over $\mathbf{W}$ and $\mathbf{\Pi}$ to find the solution.

The following result shows that provided there is a pair of fixed points $(\Pi^{(\infty)},\mathbf{\hat W}^{(\infty)})$, 
alternate minimization yields convergence provided we initialize within a constant radius of the global minimizer. For convenience we focus on the case $q = 0$ since we optimize both separately.
%
\begin{theorem}\label{thm:iterativeconvergence}
Suppose that $(\mathbf{\hat W}^{(\infty)}, \mathbf{\hat \Pi}^{(\infty)})$ are fixed points of \eqref{prob:sinkhorn}.  
Suppose that $\mathbf{\hat X}\t \mathbf{\hat \Pi}^{(\infty)} \mathbf{\hat Y}$ is rank $d$, and has smallest singular value $\sigma_d$ satisfying $\sigma_d \geq C_0 > 0$, \textcolor{black}{where $C_0$ is some fixed  constant depending on the supports of $\mathbf{X}$ and $\mathbf{Y}$}. Then the iterates of \cref{algo2} satisfy
\begin{align*}
    \| \mathbf{W}_T - \mathbf{\hat W}^{(\infty)} \|_F \leq \frac{1}{2^T} \| \mathbf{W}_{\mathrm{init}} - \mathbf{\hat W}^{(\infty)} \|_F
\end{align*}
for all $T \in \mathbb{N}$.  
\end{theorem}

\cref{thm:iterativeconvergence} demonstrates that \textcolor{black}{\cref{algo2}} yields convergence to fixed points, assuming there are fixed points. 
However, consider the case that $\mathbf{\hat X} = \mathbf{\hat Y \mathbf{W}}$ for some orthogonal $\mathbf{W}$.  Then $\mathbf{\hat X}\t \mathbf{\hat Y} = \mathbf{W}\t$, and hence $\mathbf{W}$ would be a fixed point provided $\eps$ is taken sufficiently small, as the optimal (non-regularized) coupling between $\mathbf{X}$ and $\mathbf{YW}$ is given by the identity.    In practice, we have found that starting at the sign matrices $\text{diag}(\pm 1)$  and finding various local minimums yields good performance.

In addition, \cref{thm:iterativeconvergence} requires that the smallest singular value of $\xhat\t \mathbf{\hat \Pi}^{(\infty)} \yhat$ is bounded below by some absolute constant, \textcolor{black}{which itself depends on our application of Theorem 3 of \citet{deligiannidis_quantitative_2024}, which depends on the unknown distribution function.  Therefore,  while this constant is unknown in practice, it does not grow with $n$.}  If $\mathbf{\hat X} = \mathbf{\hat Y \mathbf{W}}$ for some orthogonal $\mathbf{W}$ and $\eps$ is taken sufficiently small, then $\mathbf{\hat \Pi}^{(\infty)}$ is the identity, and the singular values of $\mathbf{\hat X}\t \mathbf{\hat Y} \mathbf{W}$ can be lower bounded by each of the nonzero singular values of $\mathbf{\hat X}$ and $\mathbf{\hat Y}$, which, under our assumptions, can be shown to be growing in $n$.  Consequently, the assumption that the singular values of $\xhat\t \mathbf{\hat \Pi}^{(\infty)} \yhat$ are lower bounded by an absolute constant, while unverifiable in practice, is not particularly stringent.

Finally, we demonstrate that even when minimizing over all block-orthogonal transformations, under a fixed alternative we still have consistency.  The following result holds regardless of how we choose to estimate $\mathbf{\hat W}_n$ (e.g. via regularized or unregularized optimal transport).
\begin{proposition} 
 \label{prop2} Let $(\aone,\mathbf{ X}) \sim GRDPG(F_X,n,\alpha_n)$ and $(\atwo, \mathbf{Y}) \sim GRDPG(F_Y,m,\beta_m)$ be independent. Suppose Assumption \ref{sparsity1} or Assumption \ref{sparsity2} is satisfied, and suppose further that $\kappa$ satisfies Assumption \ref{assumption:kappa}. 
 If $F_X \not \simeq F_Y$, then for any sequence of orthogonal matrices $\mathbf{ W}_n 
 \in \mathbb{O}(p,q) \cap \mathbb{O}(d)$, there exists a constant $C > 0$ depending only on $F_X$ and $F_Y$ such that almost surely $$\liminf_{n,m} U_{n,m}(\xhat \mathbf{W}_n/ \alpha_n^{1/2},\yhat / \beta_m^{1/2}) \geq  C. $$
 \end{proposition}


Given two adjacency matrices $\aone$ and $\atwo$, calcuating the adjacency spectral embeddings, and aligning these embeddings with an orthogonal matrix yields a consistent test statistic.  From there, one can bootstrap the null distribution of $U_{n,m}$ to get an approximate $p$-value.  In summary, our full procedure for two-sample testing is described in \cref{alg:mainprocedure}, wherein the step of estimating $\mathbf{\hat W}_n$ is given in \cref{algo2} which, as described above, uses  Sinkhorn regularization within each step.  This regularization scheme only suffers from a penalty in the objective function of $O(\eps \log(1/\eps))$, so by taking $\eps$ sufficiently small, the error from the Sinkhorn regularization will be smaller than the error obtained via optimal transport.

\begin{algorithm}[t]
\caption{Optimal Transport-Procrustes}
 \label{algo2}
\begin{algorithmic}[t]
\REQUIRE $\mathbf{\hat X, \hat Y}$ initial guesses $\mathbf{W}\in \mathbb{O}(p)$, $\mathbf{\Pi }=  \frac{1}{mn}\mathbf{1}\mathbf{1}\t$, dimension $p$.
\REPEAT
\STATE Calculate $C_{\mathbf{W}'}$ via $(C_{\mathbf{W}})_{ij} := \|\hat X_i - \mathbf{W}\hat Y_j\|_2^2$
\STATE Set $\eps > 0$ as some positive number and solve for $\mathbf{\Pi}$ aligning $\mathbf{\hat X}\mathbf{W}$ and $\mathbf{\hat Y}$ via the Sinkhorn algorithm
\STATE Set $\mathbf{M := \Pi \hat X}\mathbf{W (\hat Y} )\t$ with singular value decomposition $\mathbf{U\Sigma V\t}$, set $\mathbf{W}' := \mathbf{UV\t}$
\STATE Set $\mathbf{W} := \mathbf{W}'$
\UNTIL max number of iterations
\RETURN $\mathbf{ W}$
\end{algorithmic}

\end{algorithm}

\begin{algorithm}[h]

\begin{algorithmic}[1]

\caption{Nonparametric Two-Graph Hypothesis Testing}
\label{alg:mainprocedure}
\REQUIRE $\aone \in \R^{n\times n}, \atwo\in\R^{m \times m}$
\STATE Embed $\aone$ and $\atwo$ into $\R^d$ using the adjacency spectral embeddings, obtaining $\mathbf{\hat X}$ and $\mathbf{\hat Y}$ and sparsity estimates $\hat\alpha_n^{1/2}$ and $\hat \beta_m^{1/2}$.
\STATE Find $\mathbf{\hat W}_n^1$ using Algorithm \ref{algo2} on the first $p$ coordinates of $\mathbf{\hat X}$ and $\mathbf{\hat Y}$ 
\STATE Find $\mathbf{\hat W}_n^2$ using Algorithm \ref{algo2} on the last $q$ coordinates of $\mathbf{\hat X}$ and $\mathbf{\hat Y}$ 
\STATE Set $\mathbf{\hat W}_n = \mathrm{bdiag}(\mathbf{\hat W}_n^1, \mathbf{\hat W}_n^2)$.  
\STATE Calculate the value of the $U$-statistic $U_{n,m}(\mathbf{\hat X}/\hat\alpha_n^{1/2}, \mathbf{\hat Y \hat W}_n/\hat\beta_m^{1/2})$;
\STATE Bootstrap the $U$-statistic distribution by repeating: 
\begin{enumerate}
\item Draw a sample of size $n$ from the empirical distributions $\hat F_{\hat X/\hat \alpha_n^{1/2}}$ and $\hat F_{\hat Y/\hat \beta_m^{1/2} } \circ \mathbf{\hat W}_n$ and assigning it to $\mathbf{X}^*$;
    \item Assign $\mathbf{Y}^{*}$ as the remaining samples;
    \item Calculate $U_{n,m}(\mathbf{X}^*, \mathbf{Y}^{*})$.
\end{enumerate}
\STATE Calculate the empirical probability of observing $U_{n,m}(\mathbf{\hat X}/\hat\alpha_n^{1/2}, \mathbf{\hat Y \hat W}_n/\hat\beta_m^{1/2})$ under the bootstrapped null distribution.
\RETURN Estimated p-value.
\end{algorithmic}

\end{algorithm}



\begin{remark}[Close, but not repeated eigenvalues] \label{sec:close_eigs_discussion}
    Before moving on, we provide some intuition as to why estimating a rotation can be beneficial even when one does not  have exactly repeated eigenvalues.  We focus on the positive semidefinite case for convenience, though the analysis for the indefinite case is similar. 

Suppose $\E(XX\t)$ has $d$ distinct eigenvalues, and let $\mathbf{U_A}$ and $\mathbf{U_P}$ be the leading $d$ eigenvectors of $\mathbf{A}$ and $\mathbf{P} = \alpha_n \mathbf{XX\t}$. Let $\mathbf{W}_*$ be defined via
\begin{align*}
    \mathbf{W}_* &= \inf_{\mathbf{W} \in \mathbb{O}(d)} \| \mathbf{U_A - U_PW}\|_F,
\end{align*}
which has a closed-form solution in terms of the left and right singular vectors of the matrix $\mathbf{U_A\t U_P}$.  Since $\E(XX\t)$ has distinct eigenvalues, without loss of generality assume that the columns of $\mathbf{U_A}$ are chosen so that the inner product between the columns of $\mathbf{U_A}$ and $\mathbf{U_P}$ are positive.  Then the sequence of matrices $\mathbf{W}_*$ is converging to the identity, which also provides the uniqueness (up to sign) of the matrices $\xhat$ and $\yhat$. 

Define $$\delta := \min_{1\leq i \leq d}\bigg( \lambda_{i}(\E(XX\t)) - \lambda_{i+1}(\E(XX\t))\bigg),$$ where $\lambda_{d+1} := -\infty$ by convention.  It can be shown (see Appendix \ref{sec:close_eigs}) that
\begin{align*}
    \|\mathbf{W}_* - \mathbf{I}\|_{F} &= O\bigg( \frac{\log(n)}{n\alpha_n \delta} \bigg),
\end{align*}
where the big $O(\cdot)$ notation hides dependence on the dimension $d$. 
Hence, even though the right hand side tends to zero as $n\alpha_n \to \infty$, so the eigenvectors of $\mathbf{A}$ and $\mathbf{P}$ are well-aligned (up to sign), the rate of convergence of the orthogonal matrix depends on $n,\alpha_n$, and the corresponding eigengap.  

In practice, one  observes only the two graphs, and the eigenvalues must be estimated from the eigenvalues of $\mathbf{A}$.  Therefore, even though the orthogonal matrix is converging to the identity, for any finite $n$, it may not be close if the eigengap is small relative to $n$.  So if one observes two graphs from the same model, but both $n$ and $m$ are small relative to $\delta$, then one may still need to estimate a rotation to align $\xhat$ and $\yhat$, despite asymptotically having distinct eigenvalues. 

\end{remark}

\subsection{Relation to Previous Results}
\label{sec:previousresults}
Both \citet{ghoshdastidar_two-sample_2017} and \citet{tang_nonparametric_2017} consider a similar test as in this paper.  In \citet{ghoshdastidar_two-sample_2017}, the authors introduce a formalism for two-sample testing under the assumption one observes only the adjacency matrices.  Although our broad setting is similar to theirs, our formulation is in terms of the underlying distribution for the graphs, while \citet{ghoshdastidar_two-sample_2017} is in terms of some network statistic $f$, and thus, given two networks $\aone \sim Q$ and $\atwo \sim Q'$ on $n$ and $m$ vertices with $n,m \to \infty$, our procedure is consistent against all fixed alternatives $Q\neq Q'$ (provided both $Q$ and $Q'$ are instances of the GRDPG model), while the test procedure in \citet{ghoshdastidar_two-sample_2017} is only consistent if the distribution of the network statistic $f$ is different under $Q$ and $Q'$.


Our results are perhaps most similar to \citet{tang_nonparametric_2017}. In \citet{tang_nonparametric_2017} the authors consider a similar test statistic (albeit without additionally computing $\mathbf{\hat W}_n$) under the assumption that $\aone$ and $\atwo$ are both from the random dot product graph model, the sparsity parameters are constant, and the matrices $\E(XX\t)$ and $\E(YY\t)$ have distinct eigenvalues. Leveraging previous results for random dot product graphs, 
the authors show that if $\kappa$ is a radial kernel, and $\frac{m}{n+m} \to \rho \in (0,1)$, then there exists a sequence of orthogonal matrices $\mathbf{W}_n$ such that
\begin{align*}
    (n+m) \left(U_{n,m}(\mathbf{\hat X, \hat Y}) - U_{n,m} (\mathbf{X, YW}_n) \right)\to 0
\end{align*}
almost surely under the null hypothesis, where $ U_{n,m}(\mathbf{\hat X, \hat Y})$ is the $U$-statistic defined using the estimates $\mathbf{\hat X}$ and $\mathbf{\hat Y}$ from the adjacency spectral embeddings of $\aone$ and $\atwo$. 
Moreover, for the sequence of orthogonal matrices $\mathbf{W}_n$ we have
\begin{align*}
    U_{n,m}(\mathbf{X, YW}_n) \to \begin{cases} 0 & F_X \simeq F_Y \\ c > 0 &F_X \not \simeq F_Y, \end{cases}
\end{align*}
where recall for RDPGs $F_X \simeq F_Y$ means that $F_X$ and $F_Y$ are the equivalent up to orthogonal transformation.

While at first glance the test statistic proposed in \citet{tang_nonparametric_2017} is similar to our test statistic, analyzing the test statistic in a general low-rank setting involves substantial theoretical and methodological considerations with respect to indefinite orthogonal transformations, optimal transport, and sparsity.  
%
Consider  for illustration the $K$-block balanced homogeneous stochastic blockmodel, where each vertex $i$ is assigned to community $k$ with probability $\frac{1}{K}$, and $\mathbf{A}_{ij}$ is drawn Bernoulli$(a)$ if vertices $i$ and $j$ are the same community and Bernoulli$(b)$ if vertices $i$ and $j$ are in different communities.  Letting $\sigma(\mathbf{M})$ denote the spectrum of a symmetric matrix $\mathbf{M}$, observe that it holds that
\begin{align*}
    \sigma\bigg( \frac{1}{n} \mathbf{X} \ipq \mathbf{X}\t \bigg) = \sigma\bigg( \frac{1}{n} \mathbf{X}\t \mathbf{X} \ipq \bigg) \to \mathbb{E}(X X\t) \ipq,
\end{align*}
and hence the assumption that $\E(XX\t)$ has distinct eigenvalues is equivalent to the assumption that $\E \aone | \mathbf{X}$ has gaps between nonzero eigenvalues growing at order $n\alpha_n$.  However, immediate inspection shows that this condition is violated (with high probability) in the $K$-block balanced homogenous stochastic blockmodel if $K \geq 3$.  Moreover, if $b > a$, it is straightforward to check that the $K$-block balanced homogeneous blockmodel translates into a GRDPG model with $p = 1$ and $q = K-1$, which renders the results of \citet{tang_nonparametric_2017} invalid (who require $q = 0$).  Moreover, our results shed light on an interesting phenomenon with respect to sparsity: for dense networks with average expected degree growing faster than $\sqrt{n}\log(n)$, one attains the scaling $(m+n)$ to zero under the null hypothesis without any additional modification, something that cannot be ascertained from the results of \citet{tang_nonparametric_2017}, which only hold under dense networks with average expected degree of order $n$.  


Beyond \citet{tang_nonparametric_2017}, there have been several tests   proposed assuming that the graphs have the same set of vertices, such as
%
%
\citet{tang_semiparametric_2017,ghoshdastidar_two-sample_2017,li_two-sample_2018,levin_central_2019} and \citet{draves_bias-variance_2020}.  In \citet{tang_semiparametric_2017,levin_bootstrapping_2025} and \citet{draves_bias-variance_2020}, the authors work under the random dot product graph model, though they require that the expected degree grows as $\Theta(n)$ (that is, the sparsity parameter is constant). In \citet{li_two-sample_2018}, working under the stochastic blockmodel, the authors are able to derive more explicit limiting results for their test statistic, under the condition that the expected degree grows as $\Theta(\sqrt{n})$. 
In \citet{ghoshdastidar_two-sample_2020}, the authors allow for arbitrary distributions on two graphs, but again require that the graphs be on the same set of nodes.  In contrast to all of these works, we do not assume that the two graphs are on the same set of vertices.  

Our test statistic is based on a two-sample $U$-statistic using the rows of $\xhat$ and $\yhat$.  In \citet{levin_bootstrapping_2025}, the authors consider bootstrapping nondegenerate $U$-statistics for random dot product graphs by estimating the latent positions. 
In addition, there have been a number of works on $U$-statistics for graphs in the more general graphon model \citep{lunde_subsampling_2023,zhang_edgeworth_2022,lin_higher-order_2020,lin_theoretical_2020}, but these involve bootstrapping moments of the underlying graphon, which can be computationally infeasible in practice.  In this paper, we study a \emph{degenerate} two-sample test statistic, which is not considered in any of these works.   

We remark that Proposition \ref{prop3} provides a similar bound to Theorem 5 of \citet{levin_bootstrapping_2025} and Theorem 4.4 of \citet{lei_network_2020}, both of which consider convergence of empirical distributions to the corresponding latent position distribution under the single-graph setting, though in both of these works they require that the orthogonal matrix  in the eigendecomposition of $\E(XX\t)$ is block diagonal. 
As a counterexample, consider the following GRDPG model.  Let $\mathbf{B} \in [0,1]^{K\times K}$ be a symmetric connectivity matrix of rank $K$, and let $\mathbf{VDV\t}$ be its eigendecomposition.  Let $Z_i \sim $Dirichlet$(\alpha)$ for some $\alpha \in \R^K$, and define $X_i = \mathbf{V |D|}^{1/2} Z_i$, which is a valid GRDPG distribution.  Then the matrix $\E(XX\t)$ has a block-orthogonal eigendecomposition if and only if $\mathbf{V |D|}^{1/2} \E(ZZ\t) \mathbf{|D|}^{1/2} \mathbf{V}\t$ does, where $\E(ZZ\t)$ is the second moment matrix for a Dirichlet random variable.  
If $\alpha$ is the all ones vector, then
\begin{align*}
     \mathbf{V |D|}^{1/2} \E(ZZ\t) \mathbf{|D|}^{1/2} \mathbf{V}\t &= \frac{K}{K+1}\mathbf{V |D| V\t} + \frac{1}{K(K+1)} \mathbf{V|D|}^{1/2} \mathbf{1 1\t} \mathbf{|D|}^{1/2} \mathbf{V\t}.
\end{align*}
The example $\mathbf{B} = -.1\mathbf{I} + .2 \mathbf{11\t}$ yields an orthogonal matrix that is not block-diagonal.  Hence, even though assuming the eigendecomposition of $\E(XX\t)$ has a block diagonal structure is an attractive assumption amenable to theoretical analysis, this assumption can be violated by many different models.

Because of the prevalence of spectral methods in the literature, estimation of $\mathbf{\hat W}_n$ arises often in related inference tasks.  For example, \citet{zhang_unseeded_2018} proposes solving a smooth function of the Laplace distance between distributions to estimate $\mathbf{\hat W}_n$, and \citet{li_two-sample_2018}, operating under the stochastic blockmodel, consider estimating $\mathbf{\hat W}_n$ by minimizing over the community memberships. 
Indeed, both methods are practically similar to ours, and may provide comparable results in practice, though we believe we are the first to apply it to nonparametric hypothesis testing and to provide asymptotic statistical guarantees under both the null and alternative hypotheses.  
Furthermore, Optimal Transport-Procrustes has been used to some success in the literature on natural language processing.  Though our methodology is similar to \citet{alvarez-melis_towards_2019}, other methods have been proposed for numerically solving the problem (e.g. \citet{grave_unsupervised_2019}).

Finally, we mention that though our algorithm is based on entropy-regularized Wasserstein distance, our results are stated in terms of the unregularized Wasserstein distance.  
While it may be possible to extend results on regularized optimal transport (e.g. \citet{gangrade_efficient_2019,bigot_central_2019}) to the mixed continuous and discrete setting implicitly required for our purposes, such an extension would require nontrivial analysis of the regularized Sinkhorn distance between $\xhat$ and $\yhat$. 



\section{Simulations}
\label{sec:sims}


\textcolor{black}{For all simulations we fix $n = m$ and $\alpha _n = \beta_m$.} We first consider the balanced homogeneous stochastic blockmodel defined as follows.  First, we generate $\aone$ as a stochastic blockmodel with $\mathbf{B}\one$ fixed as
\begin{align*}
    \mathbf{B}^{(1)}&=\begin{pmatrix}.4 & .8 & .8 \\ .8 & .4 & .8 \\ .8 & .8 & .4\end{pmatrix},
\end{align*}
with equal probability of community membership.  We then generate $\atwo$ independently from $\aone$ with probability matrix $\mathbf{B}\one + \textcolor{black}{\nu} \mathbf{I}$ for $\textcolor{black}{\nu}  \in \{0,.05,.1,.15,.2\}$.  Note that all these choices of $\textcolor{black}{\nu} $ still yield a probability generating matrix with $p =1$ and $q = 2$.  We run 200 simulated permutation tests with 200 permutations and choose the critical value for $\alpha = .05$.  We run this for different choices of sparsity with $\alpha_n = \beta_m \in \{1,.8.,.6\},$ corresponding to sparser and sparser networks respectively.  

\begin{figure}[t]
 \subfloat{%
            \includegraphics[width=.3\linewidth,height=60mm]{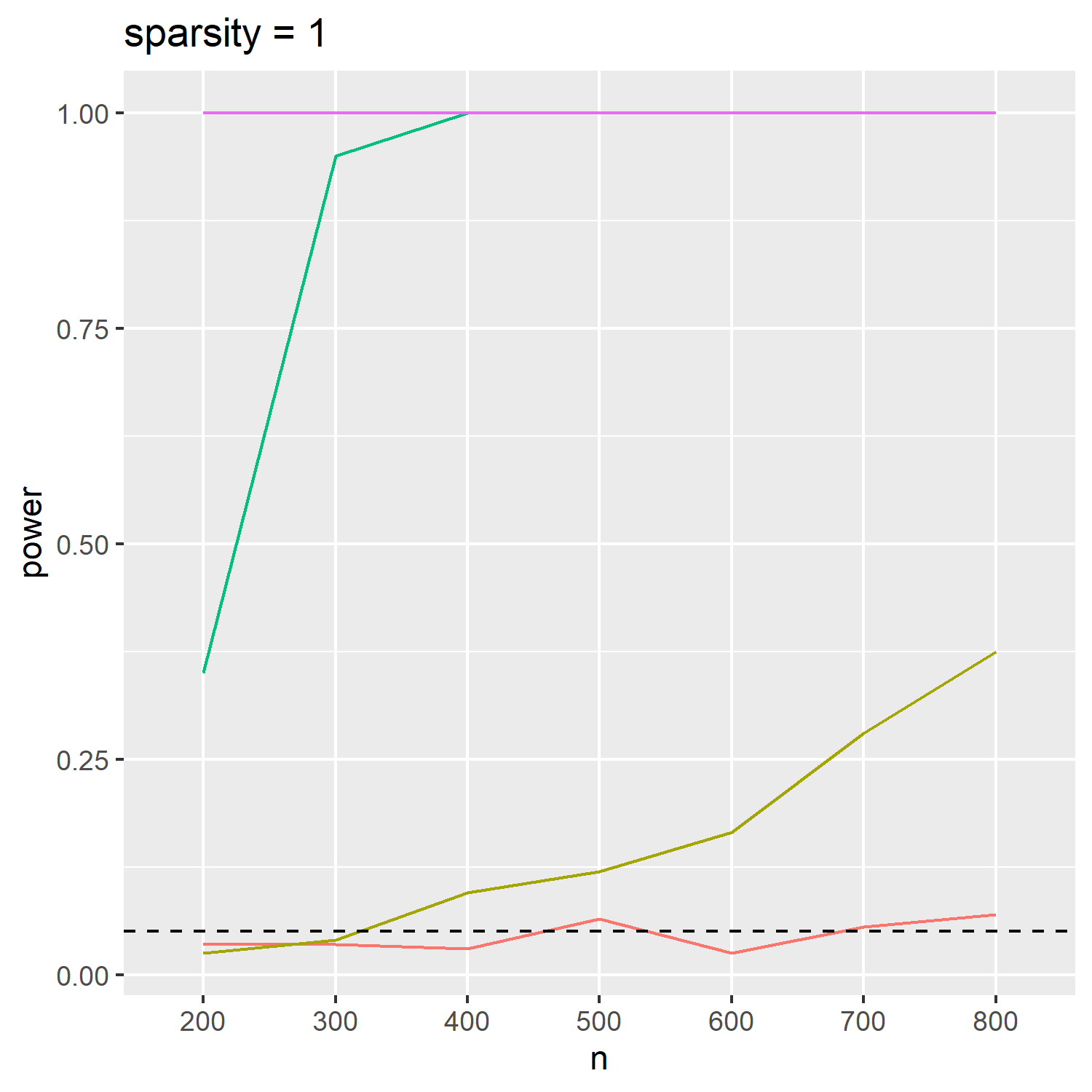}%
            \label{subfig:a}%
        }
        \subfloat{%
            \includegraphics[width=.3\linewidth,height=60mm]{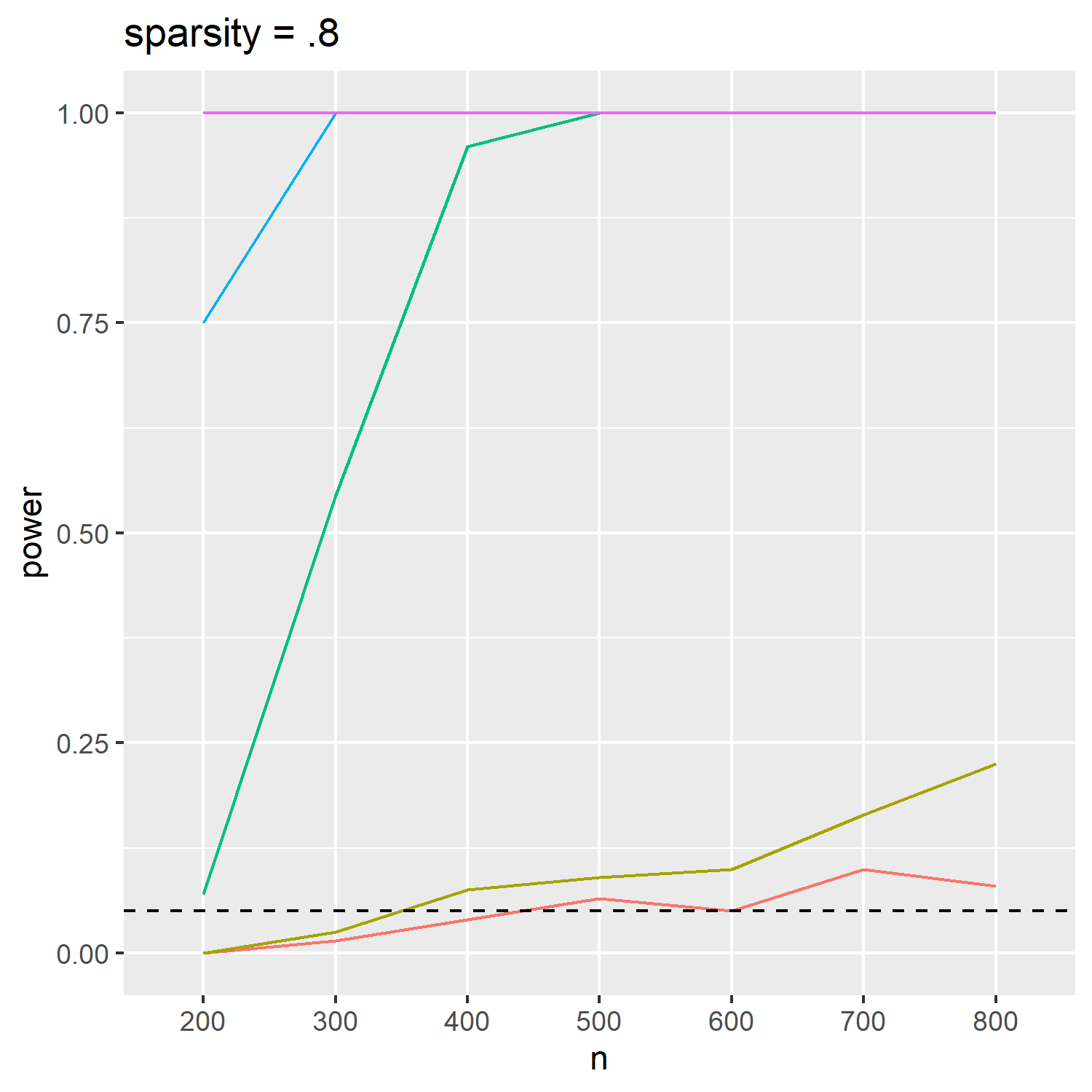}%
            \label{subfig:b}%
        }
    \subfloat{%
            \includegraphics[width=.36\linewidth,height=60mm]{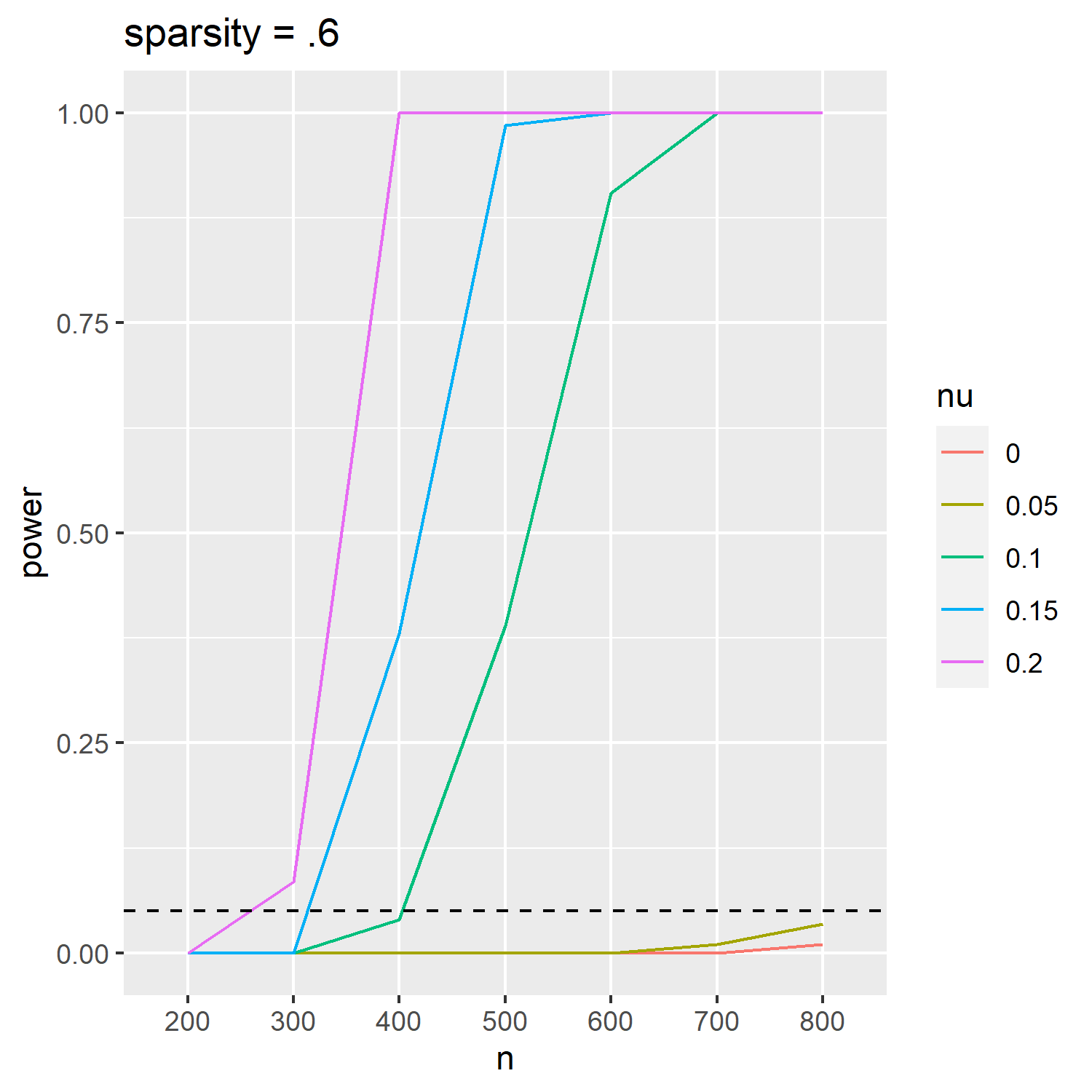}%
            \label{subfig:b}%
        }
    \caption{Hypothesis testing results with different values of $\nu$, where $\nu$ represents a local departure from the null hypothesis described in \cref{sec:sims}. As $n$ increases, it is  clear that power increases, with power tending to one slower for sparser networks, which reflects our theory. The dotted line represents level $.05$. \textcolor{black}{When curves overlap, the larger value of $\nu$ takes precedence.}   \textcolor{black}{For the value of $\nu = .05$ representing the smallest local departure from the null, the test is only able to detect when $n$ is large and the networks are sufficiently dense. Indeed, in the rightmost figure, the power curve nearly overlaps with that of the null hypothesis. }
    }   \label{fig:sbm}
\end{figure}

 In Figure \ref{fig:sbm} the results of this simulation are plotted, with power tending to increase as a function of $n$.  In addition, sparser networks result in slower power improvement, with $\alpha_n = .6$ (right) corresponding to the sparsest network.  While the tests are not exactly valid at level .05, this may be because the test is naturally a little conservative, as we minimize over all orthogonal matrices, meaning that the observations are not precisely exchangeable under the null, as required for a bootstrap test.  However, the observations are \emph{nearly} exchangeable, which is what allows the test to function properly under alternatives.  \textcolor{black}{For $\nu = .05$ (the gold line), the test struggles slightly to detect this departure, though for much larger values of $n$ and sparser graphs its power increases.  However, for the rightmost figure, this local departure is nearly undetectable.  }
Next, we consider local alternatives from the stochastic blockmodel as quantified through additional degree-corrections.  Explicitly, we keep $\mathbf{B}^{(1)} = \mathbf{B}^{(2)}$, only now varying the distributions on the degree-correction parameters.  In the first graph the degree-correction parameters are fixed at the constant $.5$, which corresponds to a fairly sparse network.  In the second graph we vary the degree-corrections by drawing them uniformly $(.5-\gamma,.5+\gamma)$, for $\gamma \in \{0,.1,.2,.3,.4\}$, resulting in higher degree heterogeneity as $\gamma$ increases. This model also results in $p =1$ and $q = 2$. In addition, we consider $\alpha_n = \beta_m \in \{1,.8,.6\}$. This null departure can be understood as a DCSBM departure from an SBM null hypothesis, which is permitted in our hypothesis testing framework. 


\begin{figure}[t]
 \subfloat{%
            \includegraphics[width=.3\linewidth,height=60mm]{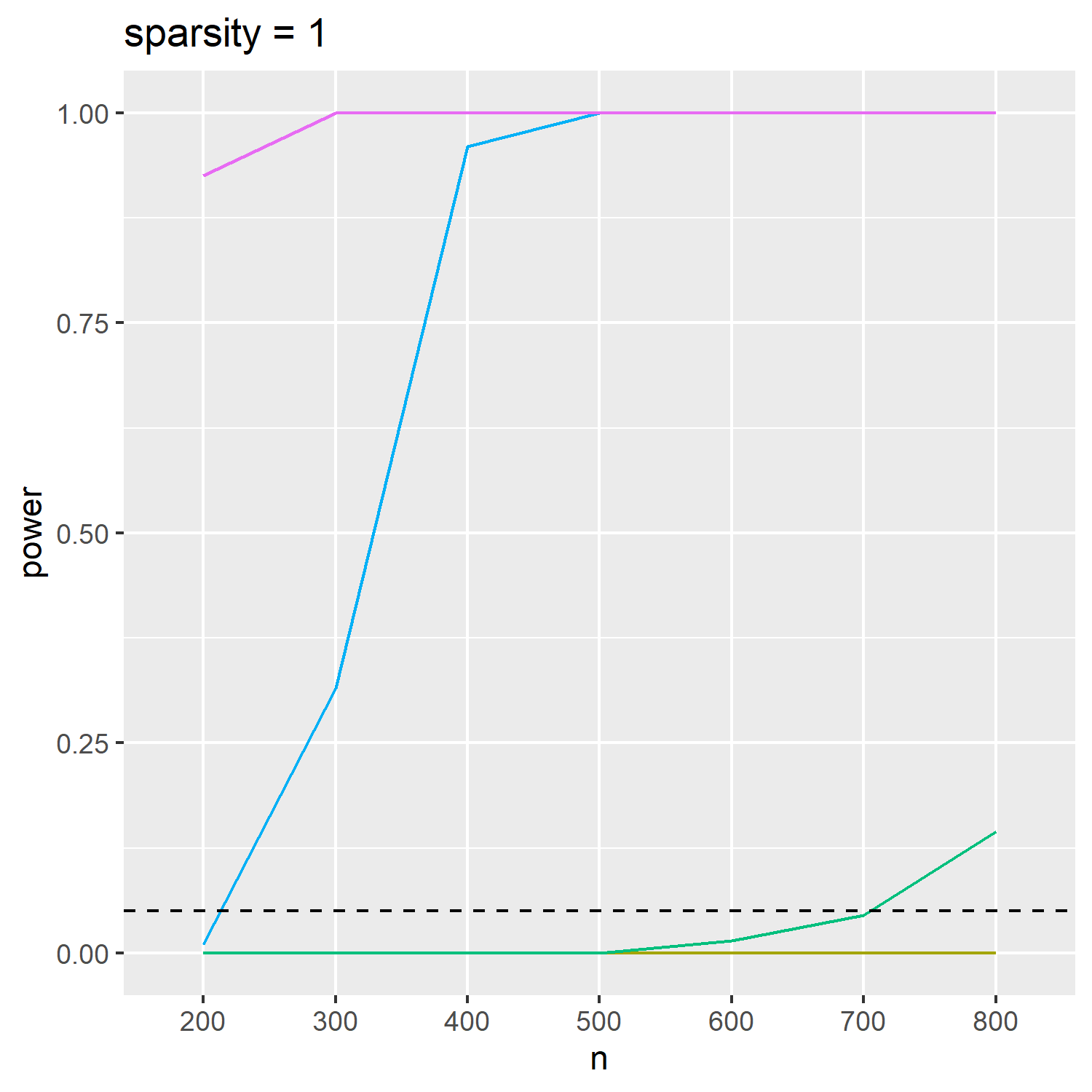}%
            \label{subfig:a}%
        }
        \subfloat{%
            \includegraphics[width=.3\linewidth,height=60mm]{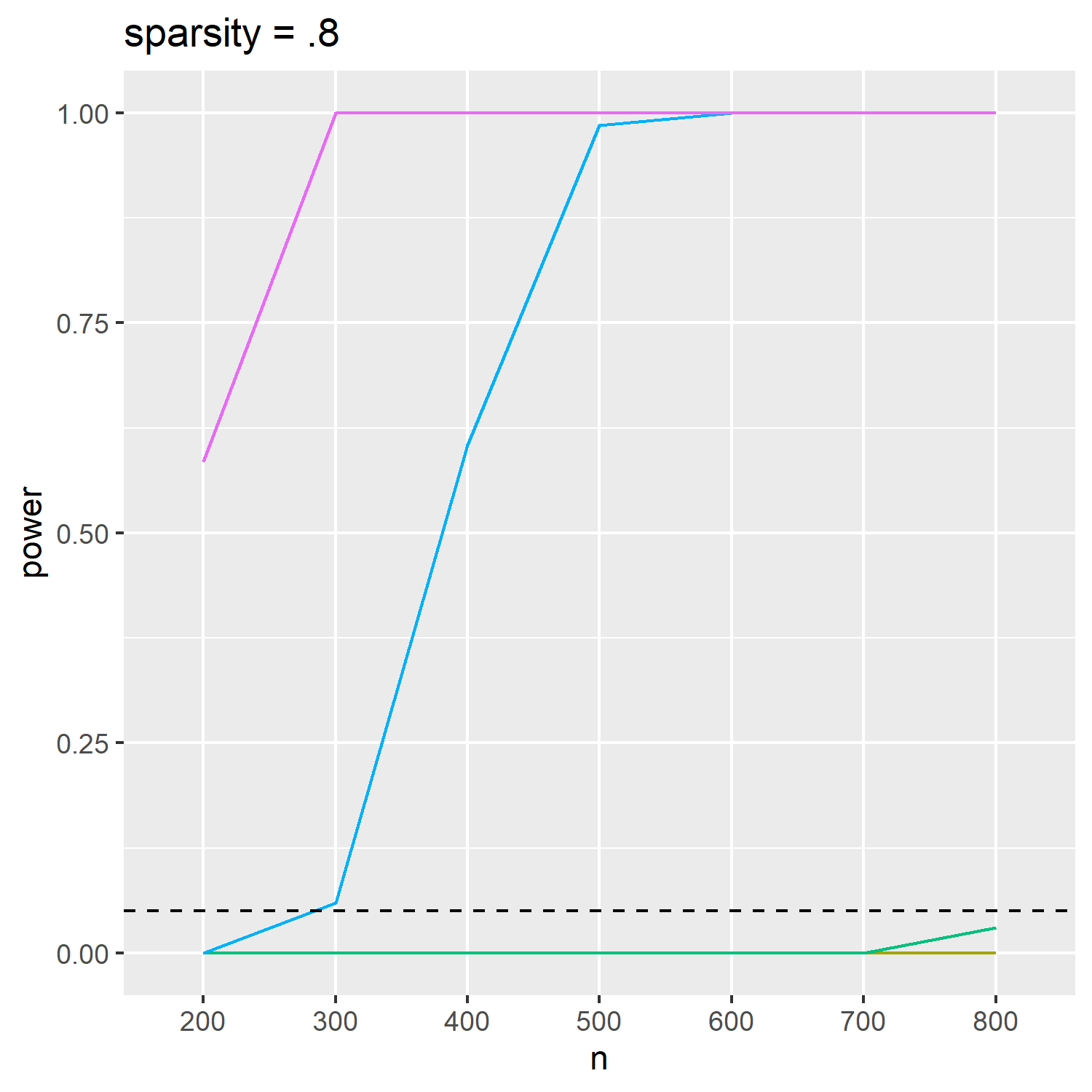}%
            \label{subfig:b}%
        }
    \subfloat{%
            \includegraphics[width=.36\linewidth,height=60mm]{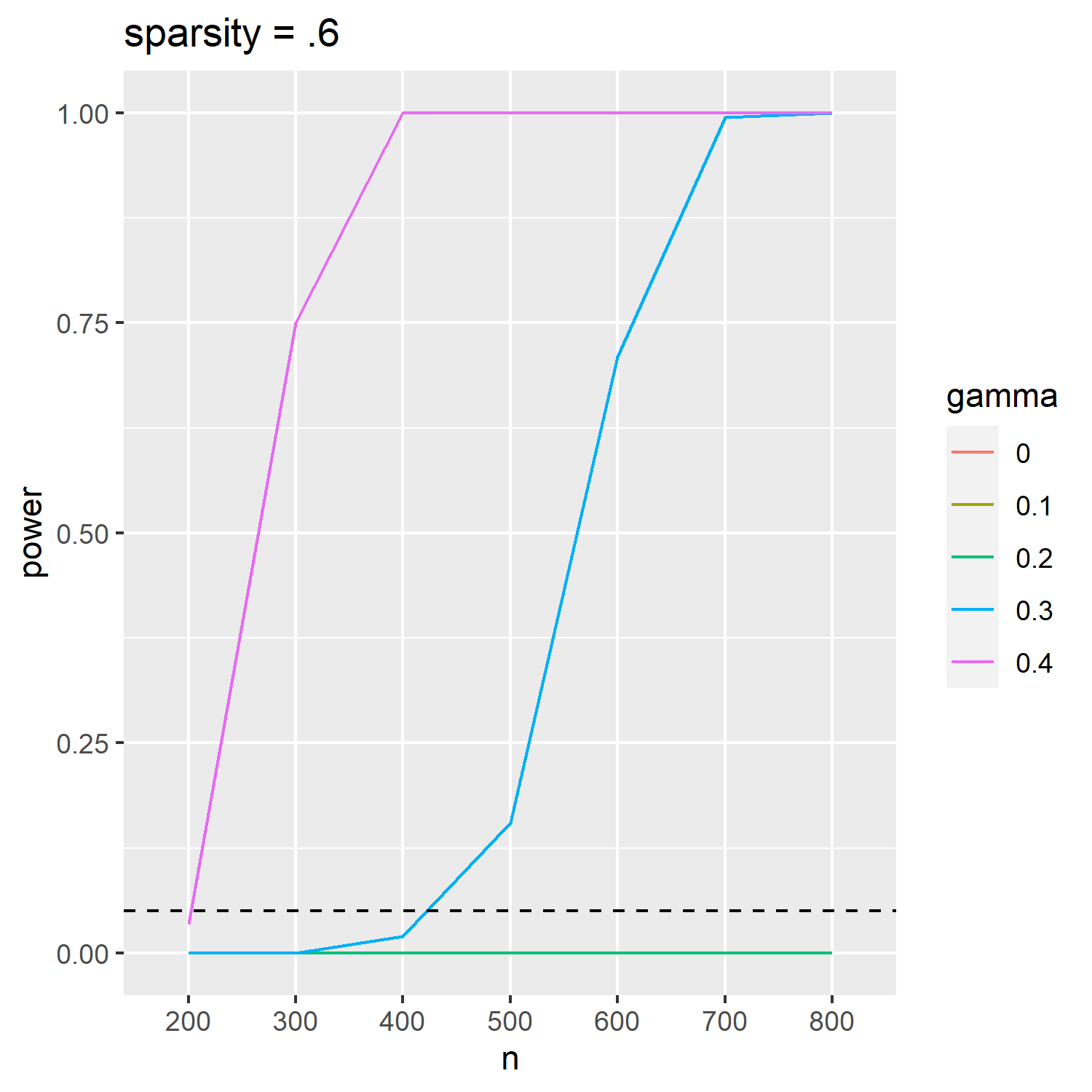}%
            \label{subfig:b}%
        }
    \caption{Hypothesis testing results for degree-corrected stochastic blockmodel departures from the null hypothesis, with $\gamma$ increasing representing more degree heterogeneity.  As $n$ increases power tends to increase, though the improvement is slower, as these networks are much sparser than the previous example.  \textcolor{black}{When curves overlap, the larger value of $\gamma$ takes precedence.} \textcolor{black}{For example, in the rightmost figure (sparsity $= .6$), the power curves associated to $\gamma \in \{0,.1,.2\}$ completely overlap at zero.}  }   \label{fig:dcsbm}
\end{figure}

In Figure \ref{fig:dcsbm} the results of this simulation are plotted, with sparsity decreasing left to right.  As in the previous example, the power tends to increase as $n$ increases, though at a slower rate for sparser networks.  Since this model is sparser than the previous model (average expected degree starting at half of the previous example), the improvements are not as stark, though all examples seem to reject consistently for large values of $\gamma$ (corresponding to high degree heterogeneity).  \textcolor{black}{However, the test struggles to detect local departures for $\gamma \in \{.1,.2\}$ except in the leftmost figure, as all of these lines completely overlap in the right two figures. }

\section{Discussion}
\label{sec:discussion}
We have shown  that a test statistic defined by using the maximum mean discrepancy applied to the rows of the adjacency spectral embedding yields a consistent test in a natural asymptotic regime as the number of vertices tends to infinity.  
The methodology we propose shows that solving the optimal transport problem estimates the orthogonal matrix stemming from the eigenvalue multiplicity of the matrices $\E(XX\t)\ipq$ and $\E(YY\t)\ipq$, which is implicit in the proofs of our results.  
While our optimization scheme alternates between points, we note that we have not proven that it yields a globally optimum solution in general, and many different initializations may be required to find the global minimizer.  

Our results show that the $U$-statistic associated to the reproducing kernel yields consistent testing under appropriate edge density; determining the exact nondegenerate limiting distribution is yet an open problem.  
The proof reveals that it will depend on the asymptotic distribution of the difference of indefinite orthogonal matrices $\sqrt{n}(\mathbf{Q_X -\tilde Q_X })$ in Section \ref{sec:proofs}, but for practical purposes, this is irrelevant, as the resulting limit will not be independent of $F_X$ and $F_Y$ in general.    
While exact derivation of the limiting distribution is complicated, we note that our procedure yields a consistent test through a simple bootstrapping procedure. Our main results have demonstrated that only repeated eigenvalues (and not negative eigenvalues) require any modification to obtain consistency for two graph hypothesis testing.

As in \citet{tang_nonparametric_2017}, one can also extend our methodology to determine whether $F_X \simeq F_Y \circ c$ for some constant $c > 0$, or for the setting $F_X \circ \pi \simeq F_Y \circ \pi$, where $\pi$ is the projection onto the sphere.  Here $F_X \simeq F_Y \circ c$ means that $F_X \simeq F_{cY}$, and $F_X \circ \pi \simeq F_Y \circ \pi$ means that $F_{\pi(X)} \simeq F_{\pi(Y)}$.  For $c$ appropriately defined so that $F_Y \circ c$ is a valid $(p,q)$-admissible distribution, one can use the estimates
\begin{align*}
    \hat s_{X} = n^{-1/2} \| \xhat \|_F; \quad  \hat s_Y = m^{-1/2}\|\yhat\|_F,
\end{align*} 
and hence, by Lemma \ref{lem:frobenius}, these are consistent estimates of the parameters 
\begin{align*}
    s_X = n^{-1/2} \| \mathbf{\tilde X}\|_F; \quad s_Y = m^{-1/2} \| \mathbf{\tilde Y} \|_F.
\end{align*}
Similarly, one can project the estimates $\xhat$ and $\yhat$ to the unit sphere to test whether $F_Y \circ \pi \simeq F_Y \circ \pi$.  
 
It remains an open question as to whether our results can be extended to graphons or other random graph models.  
In addition, while we have demonstrated that estimation of sparsity is sufficient to obtain consistency with the same scaling as in \citet{gretton_kernel_2012}, graphs with average expected degree below the $\sqrt{n}$ threshold may require further analysis of higher-order terms to obtain the same scaling.  



\section*{Acknowledgements}
The authors thank Mao Hong, Zachary Lubberts, and Joshua Cape for valuable feedback on an early draft of this paper. The first author also gratefully acknowledges funding from the Charles and Catherine Counselman Fellowship and the Johns Hopkins Mathematical Institute for Data Science (MINDS)  Fellowship. This work is partially supported by DARPA D3M under grant FA8750-17-2-0112.

\appendix
\section{Proofs of Main Results} \label{sec:proofs}

In this section we first prove our main results \cref{thm:consistency_repeated} and \cref{thm:centering}.  In \cref{sec:proof_cors} we prove Corollaries \ref{cor:scaling} and \ref{cor:scaling_estimate}.  
We then prove the results in \cref{sec:algorithm} in \cref{sec:proofsprops}.  Our main results also require a number of auxiliary technical results; these are proven in \cref{sec:proof_frob}, \cref{sec:proof_fclt}, and \cref{sec:proofs_lems} respectively.

Our proofs require careful tabulation of the various alignment matrices.  All orthogonal matrices appearing in the following proofs are written with the letter $\mathbf{W}$ and all indefinite orthogonal matrices are written with the letter $\mathbf{Q}$, and we allow the constants implicit in the notation $O(\cdot)$ to depend on $d$ in an arbitrary manner.  Finally, in our proofs, we will provide bounds with a constant $C$ that may change from line to line.

\textcolor{black}{Our results require analyzing a canonical choice of $\mathbf{X}$ (recall it is  defined only up to indefinite orthogonal matrix $\mathbf{Q}$).  }
 Throughout this section, recall that we define
\begin{align*}
    \mathbf{\tilde X} := \mathbf{U_X |\Lambda_X|}^{1/2}; \qquad \mathbf{P}\one = \mathbf{U_X \Lambda_X U_X\t} = \alpha_n \mathbf{X}\ipq \mathbf{X\t}
\end{align*}
with similar notation for $\mathbf{\tilde Y}$ and $\mathbf{P}\two$.  \textcolor{black}{We define the matrix $\mathbf{Q_X}$ as the matrix that satisfies}
\begin{align}
    \mathbf{\tilde X} = \alpha_n^{1/2}\mathbf{XQ_X\inv} \label{qx},
\end{align}
\textcolor{black}{with $\mathbf{Q_Y}$ defined similarly.}
Note that $\mathbf{\tilde X}$ and $\xhat$ depend on the sparsity $\alpha_n$, but the matrix $\mathbf{X}$ does not (since its rows are i.i.d. $F_X$), and that the matrix $\mathbf{\tilde X}$ can be thought of as the adjacency spectral embedding of the probability generating matrix $\mathbf{P} = \alpha_n \mathbf{X}\ipq \mathbf{X}\t$. See also Figure \ref{tab:diagram} below.

The following result is an adaptation of Theorem 1 from \citet{agterberg_two_2020} for the two graph setting.  The proof is straightforward and included in Section \ref{sec:proofs_lems}.  

\begin{lemma} \label{lem2}
Let $\mathbf{Q_X}$ be such that $\mathbf{U_X |\Lambda|}^{1/2} \mathbf{Q}_X = \alpha_n^{1/2}\mathbf{X}$, and $\mathbf{Q_Y}$ similarly.  Then there exist deterministic matrices $\mathbf{\tilde Q_X}$ and $\mathbf{\tilde Q_Y}$ and sequences of block-orthogonal matrices $\mathbf{W_X}$ and $\mathbf{W_Y}$ such that $\mathbf{Q_X - W_{X} \tilde Q_X} \to 0$ and similarly for $\mathbf{Q_Y}$ and $\mathbf{W_Y}$.  In particular, under the null hypothesis there exists some $\mathbf{T}$ such that $F_X \circ \mathbf{T} = F_Y$, in which case $\mathbf{\tilde Q_Y}$ can be chosen such that $\mathbf{\tilde Q_Y} =\mathbf{\tilde Q_X T}$.  
\end{lemma}

The next result concerns the approximation of the matrices $\mathbf{Q_X}$ and $\mathbf{Q_Y}$ to their population counterparts.   The proof is also in Section \ref{sec:proofs_lems}.

\begin{lemma}\label{lem:qtilde}
Let $\mathbf{Q_X}$ and $\mathbf{\tilde Q_X}$ and $\mathbf{W_X}$ be as in \cref{lem2}.  Then with probability 
at least $1 - n^{-2}$ it holds that
\begin{align}
    \|\mathbf{Q_X} - \mathbf{W_X}\mathbf{\tilde Q_X }\| &= O\left( \frac{\sqrt{\log(n)}}{\sqrt{n}} \right). \label{bound4}
\end{align}
The results continue to hold with $\mathbf{Q_X},\mathbf{\tilde Q_X},\mathbf{W_X}$ replaced with $\mathbf{Q_Y,\tilde Q_Y}$ and $\mathbf{W_Y}$ respectively.  
\end{lemma}

Finally, we need the following slight restatement of Theorem 5 of \citet{rubindelanchy_statistical_2022}.  

\begin{theorem}[Theorem 5 of \citet{rubindelanchy_statistical_2022}] \label{thm:grdpg}
Let $(\mathbf{A},\mathbf{X})\sim GRDPG(F_X,n,\alpha_n)$ for $n\alpha_n  \gg \log^{4}(n)$.  Let $\mathbf{Q_X}$ be the matrix such that $\mathbf{U_X |\Lambda|}^{1/2}\mathbf{Q_X = X}$.  Then there exists an orthogonal matrix $\mathbf{W_*^X} \in \mathbb{O}(d) \cap \mathbb{O}(p,q)$ depending on $n$ such that with probability at least $1 - n^{-2}$
\begin{align*}
    \| \xhat  \mathbf{W_*^X} - \mathbf{\tilde X } \|_{2,\infty} &= O\bigg( \frac{\log(n)}{n^{1/2}}\bigg).
\end{align*}
\end{theorem}

The next lemma is a technical result concerning the Frobenius norm concentration of $\mathbf{\hat X}$ to $\mathbf{\tilde X}$ and is needed to guarantee the existence of the specific matrices $\wx$ and $\wy$, though it is also used in the proof of Lemma \ref{lem:fclt}.  We note that similar results were proven in \citet{tang_semiparametric_2017} and \citet{tang_limit_2018} in the setting of random dot product graphs, and in \citet{athreya_numerical_2023} in the setting of numerical linear algebra for random matrices.  The asymptotic normality of a related Frobenius norm error for stochastic blockmodels was proven in \citet{li_two-sample_2018}.  However, the following lemma requires a number of additional considerations for negative eigenvalues not considered in previous works. The proof is in Section \ref{sec:proof_frob}.  

\begin{lemma} \label{lem:frobenius}
Let $(\mathbf{A, X}) \sim GRDPG(F_X, n, \alpha_n )$ for some $\alpha_n$ satisfying $n\alpha_n \geq \omega(\log^4(n))$.  Asymptotically almost surely, for the sequence of block-orthogonal matrices $\mathbf{W_*^X}$  from Theorem \ref{thm:grdpg}, the matrix $\xhat \mathbf{W_*^X}- \mathbf{\tilde X }$ admits the decomposition
\begin{align*}
    \xhat \mathbf{W_*^X} - \mathbf{\tilde X}= (\mathbf{A - P)U_X |\Lambda_X|}^{-1/2}\ipq  + \mathbf{R} 
\end{align*}
where the matrix $\mathbf{R}$ satisfies
\begin{align*}
    \|\mathbf{R}\|_F = O \left( \sqrt{\frac{\log(n)}{n\alpha_n}} \right)
\end{align*}
with high probability. Furthermore, also with high probability,
\begin{align*}
    \left| \|\xhat \mathbf{W_*^X} - \mathbf{\tilde X} 
    \|_F^2 - C^2(\mathbf{\tilde X}) \right| = O\left( \sqrt{\frac{\log(n)}{n\alpha_n}}\right),
\end{align*}
where 
\begin{align*}
C^2(\mathbf{\tilde X}) :&= \E\|(\mathbf{A- P}) \mathbf{U |S|^{-1/2}}\|_F^2.
\end{align*}
where the expectation is with respect to the randomness in $\mathbf{A}$. Finally, as $n \to \infty$, we have that
\begin{align*}
   \| \xhat \mathbf{W_*^X}- \mathbf{\tilde X } \|_F^2 \to \tr \bigg( \mathbf{\tilde Q\inv \Delta\inv \Gamma \Delta\inv \tilde Q^{-\top}} \bigg).
\end{align*}
almost surely, where $\Gamma$ is defined via
\begin{align*}
    \Gamma := \begin{cases} \E[ X X\t ( X\t \ipq \mu - X\t \ipq \Delta\ipq X)] & \alpha_n \equiv 1 \\
     \E[ X X\t ( X\t \ipq \mu )] & \alpha_n \to 0, \end{cases}
\end{align*}
with $\mu = \E X$ and $\Delta = \E XX\t$.
\end{lemma}

Finally, we present the following functional central limit theorem for the approximation of $\xhat$ to $\mathbf{\tilde X}$ under sparsity.  The proof can be found in Section \ref{sec:proof_fclt}.  The result is a generalization of  Theorem 5 in \citet{tang_nonparametric_2017} to account for both sparsity and indefinite orthogonal transformations.

\begin{lemma} \label{lem:fclt}
Let $(\mathbf{A,X})\sim GRDPG(F_X,n,\alpha_n)$ where $n\alpha_n \gg \log^4(n)$.  Suppose $\mathcal{F} \subset \{f: \R^d \to \R: f$ \textcolor{black}{is twice continuously differentiable}$\}.$   
\textcolor{black}{Suppose that $\sup_{f \in \mathcal{F}}  \| \partial f \|_{\infty} < \infty$ and $\sup_{f \in \mathcal{F}}\| \partial^2 f \| < \infty$, where $\partial f$ is understood as a vector, and $\partial^2 f$ is understood as a matrix.}
Let $\mathbf{\tilde X} := \mathbf{U_X |\Lambda_X|}^{1/2}$, and let $\mathbf{W_*^X}$ be the matrix guaranteed by Lemma \ref{lem:frobenius}. Then, as $n \to \infty$, the empirical process
\begin{align}
    f \in \mathcal{F} \mapsto  \mathbb{\hat G}_n f :=  \sqrt\frac{\alpha_n}{ n} \sum_{i=1}^{n} \left[ f\left( \frac{(\mathbf{W_*^X})\t\hat X_i}{\sqrt{\alpha_n}}\right) - f \left(  \frac{\tilde X_i}{\sqrt{\alpha_n}}\right) \right] \to 0 \label{empiricalprocess}
\end{align}
almost surely as $n \to \infty$.  Moreover, with probability at least $1- O(n^{-2})$, 
\begin{align}
   \sup_{f \in \mathcal{F}}\bigg|  \sqrt\frac{\alpha_n}{ n} \sum_{i=1}^{n} \left[ f\left( \frac{(\mathbf{W_*^X})\t\hat X_i}{\sqrt{\alpha_n}}\right) - f \left(  \frac{\tilde X_i}{\sqrt{\alpha_n}}\right) \right]\bigg| &= O\bigg(\sqrt{\frac{\log(n)}{n\alpha_n}}\bigg). \label{eq10}
\end{align}
In addition, the results hold if $(\mathbf{W_*^X})\t\hat X_i$ and $\tilde X_i$ are replaced with $\hat X_i$ and $\mathbf{W_*^X}\tilde X_i$ respectively.  Finally, if $ \sqrt{n} \alpha_n = \omega( n^{1/2} \log^{1+\eta}(n))$ for some $\eta >0$ then the result in Equation \ref{empiricalprocess} still holds under the scaling $\frac{1}{\sqrt{n}}$ instead of $\frac{\sqrt{\alpha_n}}{\sqrt{n}}$, and in Equation \eqref{eq10} the right hand side bound is of the form $O\bigg(\frac{\log^{1/2}(n)}{n^{1/2}\alpha_n}\bigg).$
\end{lemma}

\begin{table}[H]
\centering
\begin{tabular}{|c|l|} \hline &\\
 {\textbf{Notation}} & \textbf{Definition} \\ \hline &\\
$\mathbf{Q_X}, \mathbf{Q_Y} \in \mathbb{O}(p,q)$& The matrices such that $\mathbf{U_X | \Lambda_X|^{1/2}Q_X} = \alpha_n^{1/2}\mathbf{X}$  \\ & and $\mathbf{U_Y | \Lambda_Y|^{1/2}Q_Y =} \beta_m^{1/2}\mathbf{ Y}$  \\
& (Equation \ref{qx})\\
     \hline  &\\
     $\mathbf{\tilde Q_X}, \mathbf{\tilde Q_Y}\in \mathbb{O}(p,q)$ & The limiting matrices for $\mathbf{Q_X}$ and $\mathbf{Q_Y}$ from Lemma \ref{lem:qtilde}\\
     & (Equation \ref{qx}) \\
     \hline &\\
     $\wx$,$\wy\in \mathbb{O}(d) \cap \mathbb{O}(p,q)$ & The block-orthogonal matrices aligning $\mathbf{\hat X}$ and $\mathbf{\tilde X}$ \\ 
     & (Lemma \ref{lem:frobenius} and Theorem \ref{thm:grdpg}) \\
     \hline  &\\
     $\mathbf{W_X}$, $\mathbf{W_Y}\in \mathbb{O}(d) \cap \mathbb{O}(p,q)$ & The block-orthogonal matrices aligning $\mathbf{Q_X}$ and $\mathbf{\tilde Q_X}$ \\ &  from Lemma \ref{lem:qtilde}. \\
     \hline
\end{tabular}
\caption{Table of Notation}
\label{table:notation}
\end{table}

\begin{figure}[H]
    \centering
    \begin{tikzcd}
 F_X \arrow[to=3-2,"i.i.d.",sloped]\arrow[to=1-5,"\text{Lemma \ref{lem:qtilde}}"]    &                                          & & &  \mathbf{\tilde Q_X} \arrow[to=2-7,"\text{Lemma \ref{lem:qtilde}}",sloped] & \\
  &&&&&& \mathbf{W_X} \\
 & \mathbf{X} \arrow[to=4-2,"\alpha_n"] \arrow[to=3-5,"\text{Eq. \eqref{qx}}"] & & &\mathbf{Q_X} \arrow[to=2-7,"\text{Lemma \ref{lem:qtilde}}",sloped] & &\\
   &\alpha_n \mathbf{X}\ipq \mathbf{X}\t = \mathbf{P} \arrow[to=4-4,"\text{ASE}"]  \arrow[to = 6-2,"\text{Bernoulli}",sloped] & &\mathbf{\tilde X} \arrow[to=3-5,"\text{Eq. \eqref{qx}}",sloped]\arrow[to=5-6,"\text{Lemma \ref{lem:frobenius}}",sloped]  & \\
   &&&&& \mathbf{W_*^X} \\
                                         &  \mathbf{A} \arrow[to=6-4,"\text{ASE}"] & & \xhat\arrow[to=5-6,"\text{Lemma \ref{lem:frobenius}}",sloped]  & &
\end{tikzcd}
    \caption{Diagram of the alignment matrices and where they come from.  Both $\mathbf{\tilde Q_X}$  and $\mathbf{W_X}$ come from Lemma \ref{lem:qtilde}, whereas the matrix $\mathbf{W_*^X}$ comes from Lemma \ref{lem:frobenius} (or Theorem \ref{thm:grdpg}). }
    \label{tab:diagram}
\end{figure}

Armed with these technical results, we are now ready to prove Theorem \ref{thm:consistency_repeated}.  To make the proof more straightforward, we have compiled the notation for all the alignment matrices in Table \ref{table:notation}.  Although the proof mirrors that in \citet{tang_nonparametric_2017}, the steps require careful tabulation of sparsity parameters, orthogonal transformations, and indefinite orthogonal transformations, all of which require novel technical analyses.  Figure \ref{tab:diagram} also shows how to find the various alignment matrices.  Essentially, we have the approximations
\begin{align*}
    \xhat \wx \mathbf{W_X} \approx \mathbf{\tilde X} \mathbf{W_X}  = \alpha_n^{1/2} \mathbf{X Q_X\inv W_X } \approx \alpha_n^{1/2} \mathbf{ X  \tilde  Q_X\inv },
\end{align*}
using Lemmas \ref{lem:qtilde} and \ref{lem:frobenius}.  Similar approximations hold for $\yhat$ and $\mathbf{\tilde Y}$.  

\begin{proof}[Proof of Theorem  \ref{thm:consistency_repeated}]



Define the matrices $\mathbf{W_*^X}$ and $\mathbf{W_X}$ where $\mathbf{W_*^X}$ is the orthogonal matrix from Lemma \ref{lem:frobenius}, and $\mathbf{W_X}$ is the matrix from Lemma \ref{lem:qtilde}.  Define $\mathbf{W_*^Y}$ and $\mathbf{W_Y}$ similarly.   We will define
\begin{align*}
    \mathbf{\hat W}_n := \mathbf{W_*^X} \mathbf{W_X} (\mathbf{W_*^Y})\t \mathbf{W_Y}\t; \qquad \mathbf{W}_n :=  \mathbf{W_X}  \mathbf{W_Y}\t.
\end{align*}
We will implicitly use the fact that $\kappa$ is radial to reorganize terms.  In particular, it holds that
\begin{align*}
 U_{n,m}(\xhat \mathbf{\hat W}_n/\sqrt{\alpha_n}, \yhat/\sqrt{\beta_m}) = U_{n,m}( \xhat \mathbf{W_*^X} \mathbf{W_X} /\sqrt{\alpha_n}, \yhat \mathbf{W_*^Y}\mathbf{W_Y}/\sqrt{\beta_m} ).
\end{align*}
We will use this fact without further reference.


\textcolor{black}{
Let $\Phi( \cdot)$ denote the feature map of $\kappa(\cdot,\cdot)$ associated to its corresponding reproducing kernel Hilbert space.}  Define $\tilde V_{n,m} := V_{n,m}\left( \frac{\mathbf{\tilde XW_X}}{\sqrt{\alpha_n}}, \frac{\mathbf{\tilde YW_Y}}{\sqrt{\beta_m}} \right)$ via
\begin{align*}
    \tilde V_{n,m}\left( \frac{\mathbf{\tilde XW_X}}{\sqrt{\alpha_n}}, \frac{\mathbf{\tilde YW_Y}}{\sqrt{\beta_m}} \right) &= \left \| \frac{1}{n} \sum_{i=1}^{n} \Phi\left( \frac{\mathbf{W_X\t}\tilde X_i}{\sqrt{\alpha_n}}\right) - \frac{1}{m} \sum_{k=1}^{m} \Phi\left( \frac{ \mathbf{W_Y\t}\tilde Y_k}{\sqrt{\beta_m}}\right) \right\|_{\mathcal{H}}^2 \\
    &= \frac{1}{n^2} \sum_{i=1}^{n} \sum_{j=1}^{m} \kappa\left( \frac{\mathbf{W_X\t}\tilde X_i}{\sqrt{\alpha_n}},\frac{\mathbf{W_X\t}\tilde X_j}{\sqrt{\alpha_n}}\right) \\
    &\quad- \frac{2}{mn} \sum_{i=1}^{n} \sum_{k=1}^{m} \kappa\left( \frac{\mathbf{W_X\t}\tilde X_i}{\sqrt{\alpha_n}}, \frac{ \mathbf{W_Y\t}\tilde Y_k}{\sqrt{\beta_m}}\right) \\&\quad +\frac{1}{m^2} \sum_{k=1}^{m} \sum_{l=1}^{m} \kappa\left(  \frac{ \mathbf{W_Y\t}\tilde Y_k}{\sqrt{\beta_m}}, \frac{ \mathbf{W_Y\t}\tilde Y_l}{\sqrt{\beta_m}}\right)
\end{align*}
and analogously for $\hat V_{n,m}$. We have the decomposition
\begin{align*}
    (m\beta_m + n\alpha_n)&\left(V_{n,m}\left( \frac{\mathbf{\tilde X}\mathbf{W_X}}{\sqrt{\alpha_n}}, \frac{\mathbf{\tilde Y}\mathbf{W_Y}}{\sqrt{\beta_m}} \right) - V_{n,m}\left( \frac{\mathbf{\hat X}\wx\mathbf{W_X}}{\sqrt{\alpha_n}}, \frac{\mathbf{\hat Y}\wy\mathbf{W_Y}}{\sqrt{\beta_m}} \right) \right) \\
    &= (m\beta_m + n\alpha_n)\left(U_{n,m}\left( \frac{\mathbf{\tilde X}\mathbf{W_X}}{\sqrt{\alpha_n}}, \frac{\mathbf{\tilde Y}\mathbf{W_Y}}{\sqrt{\beta_m}} \right) - U_{n,m}\left( \frac{\mathbf{\hat X}\wx\mathbf{W_X}}{\sqrt{\alpha_n}}, \frac{\mathbf{\hat Y}\wy\mathbf{W_Y}}{\sqrt{\beta_m}} \right) \right) \\&\quad + r_1 + r_2,
\end{align*}
where
\begin{align*}
    r_1 &= \frac{ (m\beta_m + n\alpha_n)}{n(n-1)} \sum_{i=1}^{n} \bigg[\kappa\left( \frac{ \mathbf{W_X\t}\tilde X_i}{\sqrt{\alpha_n}},\frac{\mathbf{W_X\t}\tilde X_i}{\sqrt{\alpha_n}}\right) -  \kappa\left(  \frac{(\wx\mathbf{W_X})\t\hat X_i}{\sqrt{\alpha_n}},\frac{(\wx\mathbf{W_X})\t\hat X_i}{\sqrt{\alpha_n}}\right) \bigg] \\
    &\quad + \frac{ (m\beta_m + n\alpha_n)}{n(n-1)} \sum_{k=1}^{m} \bigg[\kappa\left( \frac{ \mathbf{W_Y\t}\tilde Y_k}{\sqrt{\beta_m}}, \frac{\mathbf{W_Y\t}\tilde Y_k}{\sqrt{\beta_m}}\right) 
   -  \kappa\left( \frac{ ( \wy\mathbf{W_Y})\t\hat Y_k}{\sqrt{\beta_m}},\frac{( \wy\mathbf{W_Y})\t\hat Y_k}{\sqrt{\beta_m}}\right) \bigg];
    \end{align*}
    \begin{align*}
    r_2 &= \frac{ (m\beta_m + n\alpha_n)}{n^2(n-1)} \sum_{i=1}^{n}\sum_{j=1}^{n} \bigg[\kappa\left( \frac{ \mathbf{W_X\t}\tilde X_i}{\sqrt{\alpha_n}},\frac{\mathbf{W_X\t}\tilde X_j}{\sqrt{\alpha_n}}\right) -  \kappa\left(  \frac{(\wx\mathbf{W_X})\t\hat X_i}{\sqrt{\alpha_n}},\frac{(\wx\mathbf{W_X})\t\hat X_j}{\sqrt{\alpha_n}}\right) \bigg] \\
    &\quad +  \frac{ (m\beta_m + n\alpha_n)}{m^2(m-1)} \sum_{k=1}^{m}\sum_{l=1}^{m}\bigg[\kappa\left( \frac{ \mathbf{W_Y\t}\tilde Y_k}{\sqrt{\beta_m}}, \frac{\mathbf{W_Y\t}\tilde Y_l}{\sqrt{\beta_m}}\right) 
   -  \kappa\left( \frac{ ( \wy\mathbf{W_Y})\t\hat Y_k}{\sqrt{\beta_m}},\frac{( \wy\mathbf{W_Y})\t\hat Y_l}{\sqrt{\beta_m}}\right) \bigg]. 
\end{align*}
By assumption $\mathbf{\bar \Omega}$ is compact, so by the fact that $\kappa$ is twice continuously differentiable, $\kappa$ is Lipschitz on  $\mathbf{Q_X\inv \bar \Omega}$ since $\mathbf{Q_X\inv \bar\Omega}$ is compact. In particular, for some positive $K$, we have that
\begin{align*}
\|r_1\| &\leq K\frac{m\beta_m + n\alpha_n}{n-1} \max_{i}\left\| \frac{\mathbf{W_X\t}\tilde X_i}{\sqrt{\alpha_n}} - \frac{(\wx \mathbf{W_X})\t \hat X_i}{\sqrt{\alpha_n}} \right\| \\
&\quad 
+ K\frac{m\beta_m + n\alpha_n}{m-1} \max_{i} \left\|\frac{ \mathbf{W_Y\t}\tilde Y_i}{\sqrt{\alpha_n}} - \frac{(\wy \mathbf{W_Y})\t \hat Y_i}{\sqrt{\alpha_n}} \right\| \\
&\leq 2K \frac{m\beta_m + n\alpha_n}{n\alpha_n } \sqrt{\alpha_n} \frac{\log(n)}{\sqrt{n\alpha_n}} + 2K\frac{m\beta_m + n\alpha_n}{m\beta_m } \sqrt{\beta_m} \frac{\log(m)}{\sqrt{m\beta_m}} \\
&\leq C \bigg(\frac{\log(n)}{\sqrt{n}} + \frac{\log(m)}{\sqrt{m}}\bigg),
\end{align*}
by the assumption that $\frac{m\beta_m}{m\beta_m + n\alpha_n} \to \rho \in (0,1)$ and the $\ell_{2,\infty}$ bound in Theorem \ref{thm:grdpg}.  By a similar argument,
\begin{align*}
    \| r_2 \| &\leq K \frac{m\beta_m + n\alpha_n}{(n-1)}   \max_{i}\left\| \frac{\mathbf{W_X\t}\tilde X_i}{\sqrt{\alpha_n}} - \frac{(\wx \mathbf{W_X})\t \hat X_i}{\sqrt{\alpha_n}} \right\| \\
    &\quad 
+ K\frac{m\beta_m + n\alpha_n}{m-1} \max_{i}\left\|\frac{ \mathbf{W_Y\t}\tilde Y_i}{\sqrt{\alpha_n}} - \frac{(\wy \mathbf{W_Y})\t \hat Y_i}{\sqrt{\alpha_n}} \right\| \\
&\leq C \bigg(\frac{\log(n)}{\sqrt{n}} + \frac{\log(m)}{\sqrt{m}}\bigg).
\end{align*}
Both of these bounds are  independent of $\alpha_n$ and $\beta_m$. 
Define
\begin{align*}
    \xi :&= \frac{ \sqrt{m \beta_m + n\alpha_n}}{n}  \sum_{i=1}^{n} \kappa(\mathbf{W_X\t}\tilde X_i/\sqrt{\alpha_n}, \cdot) - \frac{ \sqrt{m \beta_m + n\alpha_n}}{m} \sum_{k=1}^{m} \kappa(\mathbf{W_Y\t}\tilde Y_k/\sqrt{\beta_m}, \cdot) \\
    \hat \xi :&=\frac{ \sqrt{m \beta_m + n\alpha_n}}{n}  \sum_{i=1}^{n} \kappa((\wx\mathbf{W_X})\t\hat X_i/\sqrt{\alpha_n}, \cdot) - \frac{ \sqrt{m \beta_m + n\alpha_n}}{m} \sum_{k=1}^{m} \kappa((\wy\mathbf{W_Y})\t\hat Y_k/\sqrt{\beta_m}, \cdot).
\end{align*}
\textcolor{black}{By the triangle inequality for $\|\cdot\|_{\mathcal{H}}$, we now have that }
\begin{align*}
    \bigg|(m\beta_m + n\alpha_n)&\left(V_{n,m}\left( \frac{\mathbf{\tilde X}\wx\mathbf{W_X}}{\sqrt{\alpha_n}}, \frac{\mathbf{\tilde Y}\wy\mathbf{W_Y}}{\sqrt{\beta_m}} \right) - V_{n,m}\left( \frac{\mathbf{\hat X}\mathbf{W_X}}{\sqrt{\alpha_n}}, \frac{\mathbf{\hat Y}\mathbf{W_Y}}{\sqrt{\beta_m}} \right) \right)\bigg| \\
    &= \left| \| \xi \|_\mathcal{H}^2 - \| \hat \xi \|_{\mathcal{H}}^2 \right| \\
    &\leq \| \xi - \hat \xi\|_{\mathcal{H}} \left( 2 \| \xi \|_{\mathcal{H}} + \| \xi - \hat \xi \|_{\mathcal{H}} \right).
\end{align*}
We have that

\begin{align*}
    \xi - \hat \xi &= \sqrt{\frac{m\beta_m + n\alpha_n}{n}} \sum_{i=1}^{n} \frac{\kappa(\mathbf{W_X\t}(\wx)\t\hat X_i/\sqrt{\alpha_n}, \cdot) - \kappa(\mathbf{W_X\t}\tilde X_i/\sqrt{\alpha_n}, \cdot) }{\sqrt{n}} \\&\quad- \sqrt{\frac{m\beta_m + n\alpha_n}{n}} \sum_{i=1}^{n} \frac{\kappa(\mathbf{W_Y\t}(\wy)\t\hat Y_i/\sqrt{\beta_m}, \cdot) - \kappa((\mathbf{W_Y\t}\tilde Y_i/\sqrt{\beta_m}, \cdot) }{\sqrt{n}} \\
    :&= \zeta_X - \zeta_Y.
\end{align*}
We note that since $\kappa$ is radial, we can disregard $\mathbf{W_X\t}$ and $\mathbf{W_Y\t}$.  Moreover,
\begin{align*}
    \|\zeta_X\|
    &= \sqrt{\frac{m\beta_m + n\alpha_n}{n\alpha_n}} \left\| \sqrt{\alpha_n}\sum_{i=1}^{n}  \frac{\kappa((\wx)\t\hat X_i/\sqrt{\alpha_n}, \cdot) - \kappa(\tilde X_i/\sqrt{\alpha_n}, \cdot) }{\sqrt{n}} \right\|.
\end{align*}
Since $m\beta_m/(m\beta_m+n\alpha_n) \to \rho \in (0,1)$ the term outside of the norm is of order $O(1)$. \textcolor{black}{Letting $\mathcal{F}$ denote the set of feature maps associated to $\mathcal{H}$, we see that $\mathcal{F}$ satisfies the conditions of Lemma \ref{lem:fclt} since the set of feature maps has uniformly bounded first and second derivatives by Corollary 4.36 of \citet{steinwart_support_2008}.} Lemma  \ref{lem:fclt} then implies that the term inside of the norm tends to zero, and the same argument holds for $\zeta_Y$.  In particular, we see that by Lemma \ref{lem:fclt},
\begin{align}
    \| \xi - \hat \xi\|_{\mathcal{H}} &=  O\bigg( \sqrt{\frac{\log(n)}{n\alpha_n}} + \sqrt{\frac{\log(m)}{m\beta_m}}\bigg)\label{fclt:app}
\end{align}
with probability at least $1 - O(n^{-2} + m^{-2})$.

We now bound $\|\xi\|_{\mathcal{H}}$ under the null and alternative respectively.  
Recall that $\mathbf{\tilde Q_X}$ is the limiting matrix guaranteed by Lemma \ref{lem:qtilde}.  Note further that $\mathbf{\tilde X}/\sqrt{\alpha_n} = \mathbf{XQ_X\inv}$ so that $\mathbf{\tilde XW_X}/\sqrt{\alpha_n} = \mathbf{XQ_X\inv W_X}$, where $\mathbf{W_X}$ was the matrix such that $\mathbf{Q_X  - W_X \tilde Q_X}$ is of order $\sqrt{\log(n)/n}$ with probability at least $1 - n^{-2}$. 
Hence, from the equation
\begin{align*}
    \mathbf{Q_X\inv} = \ipq \mathbf{Q_X\t} \ipq,
\end{align*}
we see that with probability at least $1- n^{-2}$ that
\begin{align*}
    \| \mathbf{Q_X\inv W_X - \tilde Q_X\inv} \| 
    &= \| \mathbf{Q_X} - \mathbf{W_X}    \mathbf{\tilde Q_X} \| \\
&= O\bigg( \frac{\sqrt{\log(n)}}{\sqrt{n}} \bigg).
\end{align*}
Hence, we have that
\begin{align*}
    \xi &= \frac{ \sqrt{m \beta_m + n\alpha_n}}{n}  \sum_{i=1}^{n} \kappa(\mathbf{W_X\t}\tilde X_i/\sqrt{\alpha_n}, \cdot) - \frac{ \sqrt{m \beta_m + n\alpha_n}}{m} \sum_{k=1}^{m} \kappa(\mathbf{W_Y\t}\tilde Y_k/\sqrt{\beta_m}, \cdot) \\
    &=  \frac{ \sqrt{m \beta_m + n\alpha_n}}{\sqrt n}  \sum_{i=1}^{n} \frac{\kappa( (\mathbf{Q_X\inv W_X)}\t X_i,\cdot) - \kappa(\mathbf{\tilde Q_X^{-\top}} X_i,\cdot)}{\sqrt{n}} \\&\quad - \frac{ \sqrt{m \beta_m + n\alpha_n}}{\sqrt{m}} \sum_{k=1}^{m} \frac{\kappa((\mathbf{Q_Y\inv W_Y)}\t Y_i,\cdot)- \kappa(\mathbf{\tilde Q_Y^{-\top}} X_Y,\cdot)}{\sqrt{m}} \\
    &\quad + \frac{ \sqrt{m \beta_m + n\alpha_n}}{\sqrt n}  \sum_{i=1}^{n} \frac{  \kappa(\mathbf{\tilde Q_X^{-\top}} X_i,\cdot) - \mu[ F_X \circ \mathbf{\tilde Q_X^{-\top}}] }{\sqrt{n}} \\
    &\quad - \frac{ \sqrt{m \beta_m + n\alpha_n}}{\sqrt{m}} \sum_{k=1}^{m} \frac{ \kappa(\mathbf{\tilde Q_Y^{-\top}} Y_k,\cdot)  - \mu[ F_Y \circ \mathbf{\tilde Q_Y^{-\top}}] }{\sqrt{m}} \\
    &\quad  + \sqrt{m \beta_m + n\alpha_n}  \mu[ F_X \circ \mathbf{\tilde Q_X^{-\top}}]  -  \sqrt{m \beta_m + n\alpha_n}  \mu[ F_Y \circ \mathbf{\tilde Q_Y^{-\top}}]
\end{align*}
For the first two terms, by the Lipschitz property of $\kappa$ and the fact that $\mathbf{\bar \Omega}$ is compact, we have that
\begin{align*}
   \frac{ \sqrt{m \beta_m + n\alpha_n}}{\sqrt n}  \sum_{i=1}^{n} &\frac{\kappa( (\mathbf{Q_X\inv W_X)}\t X_i,\cdot) - \kappa(\mathbf{\tilde Q_X^{-\top}} X_i,\cdot)}{\sqrt{n}} \\&\leq   \frac{ \sqrt{m \beta_m + n\alpha_n}}{\sqrt n} \sum_{i=1}^{n} \frac{\| (\mathbf{Q_X\inv W_X)}\t X_i - \mathbf{\tilde Q_X^{-\top}} X_i\|}{\sqrt{n}}  \\
   &\leq  \sqrt{m \beta_m + n\alpha_n} \max_{i}\| X_i (\mathbf{Q_X\inv W_X) - \tilde Q_X\inv}) \| \\
   &\leq  \| \mathbf{X}\|_{2,\infty} \sqrt{m \beta_m + n\alpha_n} \|\mathbf{Q_X\inv W_X - \tilde Q_X\inv} \| \\
   &\leq C \sqrt{m \beta_m + n\alpha_n} \frac{\sqrt{\log(n)}}{\sqrt{n}} \\
   &= O\bigg( \sqrt{\alpha_n \log(n)}\bigg)
\end{align*}
with probability at least $1 - O(n^{-2})$ by Lemma  \ref{lem:qtilde}.  Hence the first term is of order
\begin{align*}
  O\bigg(  \sqrt{\alpha_n \log(n)} + \sqrt{\beta_m \log(m)} \bigg)
\end{align*}
with probability at least $1 - O(n^{-2} + m^{-2})$.

Now, define
\begin{align*}
    \psi_X :&=\sum_{i=1}^{n} \frac{\kappa(\mathbf{\tilde Q_X\inv} X_i,\cdot)  - \mu( F_X \circ \mathbf{\tilde Q_X\inv})}{n} \\
    \psi_Y :&= \sum_{i=1}^{n} \frac{\kappa(\mathbf{\tilde Q_Y\inv} Y_i,\cdot) - \mu( F_Y \circ  \mathbf{\tilde Q_Y\inv})}{m}
\end{align*}
By Remark 1 in \citet{schneider_probability_2016}, we see that 
\begin{align*}
    \p  ( \| \psi_X \|^2 > \eps^2 ) &\leq 2\exp\bigg[ \frac{-n\eps^2}{64}\bigg]
\end{align*}
which in particular shows that $\|\psi_X\| \leq \frac{C \sqrt{\log(n)}}{\sqrt{n}}$ with probability at least $1 - O(n^{-2})$, and similarly for $\psi_Y$.  Thus far, we have shown with probability at least $1 - O(n^{-2} + m^{-2})$ that
\begin{align*}
    \| \xi \|_{\mathcal{H}} &= O\bigg( \sqrt{\alpha_n \log(n)} + \sqrt{\beta_m \log(m)} \bigg) + \sqrt{m \beta_m + n\alpha_n}\bigg( \| \mu[ F_X \circ \mathbf{\tilde Q_X^{-\top}}] - \mu[F_Y \circ \mathbf{\tilde Q_Y^{-\top}}] \|_{\mathcal{H}}\bigg).
\end{align*}
Under the null hypothesis, the term in the parentheses is zero as $\mathbf{\tilde Q_Y}$ can be chosen to be $\mathbf{T\inv \tilde Q_X\inv}$ by Lemma \ref{lem:qtilde} and the fact that $\mathbf{\tilde Q_Y}$ and $\mathbf{\tilde Q_X}$ do not depend on the sparsity factors by Lemma \ref{lem:qtilde}.  Hence, we see that
\begin{align*}
 \|\xi - \hat \xi \| \bigg(2 \| \xi \| + \| \xi - \hat \xi\|\bigg) &= O\bigg(\sqrt{\frac{\log(n)}{n\alpha_n}} + \sqrt{\frac{\log(m)}{m\beta_m}}\bigg) \bigg( \sqrt{\alpha_n \log(n) + \beta_m \log(m)} \bigg).
\end{align*}
Under the alternative the term $\| \mu[ F_X \circ \mathbf{\tilde Q_X^{-\top}}] - \mu[F_Y \circ \mathbf{\tilde Q_Y^{-\top}}] \|_{\mathcal{H}}$ is not necessarily zero, so we have that
\begin{align*}
     \|\xi - \hat \xi \| \bigg(2 &\| \xi \| + \| \xi - \hat \xi\|\bigg) \\&= O\bigg(\sqrt{\frac{\log(n)}{n\alpha_n}} +\sqrt{\frac{\log(m)}{m\beta_m}}\bigg) \bigg( \sqrt{\alpha_n \log(n) + \beta_m \log(m)} + \sqrt{n\alpha_n + m\beta_m} \bigg) \\
     &= O\bigg( \sqrt{ \log(n)} + \sqrt{\log(m)}\bigg).
\end{align*}
Hence, dividing by $\log(n)$ yields the result.
\end{proof}

\noindent We now prove \cref{thm:centering}.

\begin{proof}[Proof of \cref{thm:centering}]
Recall that we define $\mathbf{W}_n := \mathbf{W_X W_Y}\t$, where $\mathbf{W_X}$ and $\mathbf{W_Y}$ are as in the proof of \cref{thm:consistency_repeated}.  By the Lipschitz property of $\kappa$, it holds that
\begin{align*}
    | U_{n,m} (\xtilde \mathbf{W}_n/ \alpha_n^{1/2},& \ytilde/\beta_m^{1/2} ) - U_{n,m}( \mathbf{X} \mathbf{\tilde Q_X}\inv, \mathbf{Y} \mathbf{\tilde Q_Y}\inv) | \\&\leq C \big( \| \mathbf{Q_X\inv W_X} - \mathbf{\tilde Q_X}\inv \| + \| \mathbf{Q_Y\inv W_Y} - \mathbf{\tilde Q_Y}\inv \| \big) \\
    &= O\bigg( \sqrt{\frac{\log(n)}{n}} + \sqrt{\frac{\log(m)}{m}} \bigg),
\end{align*}
where the final bound holds with probability at least $1 - n^{-2} - m^{-2}$ by \cref{lem:qtilde} and similar manipulations as in the proof of \cref{thm:consistency_repeated} using the identity $\mathbf{Q_X}\inv = \ipq \mathbf{Q_X}\t \ipq$. 

Furthermore, by the concentration inequality for $U$-statistics \citep[Section~5.6]{serfling_approximation_1980}, it holds that
\begin{align*}
    \big| U_{n,m} (\mathbf{X} \mathbf{\tilde Q_X}\inv, \mathbf{Y} \mathbf{\tilde Q_Y}\inv ) - \| \mu[ F_X \circ \mathbf{\tilde Q_X}\inv - \mu[ F_Y \circ \mathbf{\tilde Q_Y}\inv]] \|_{\mathcal{H}}^2 \big| = O\bigg( \sqrt{\frac{\log(n+m)}{n}} + \sqrt{\frac{\log(n+m)}{m}}\bigg)
\end{align*}
with probability at least $1- O(n^{-2} + m^{-2})$.  

Finally, the fact that $\| \mu[ F_X \circ \mathbf{\tilde Q_X}\inv - \mu[ F_Y \circ \mathbf{\tilde Q_Y}\inv]] \|_{\mathcal{H}}^2 = 0$ under the null hypothesis holds by \cref{lem2} since $\mathbf{\tilde Q_Y}$ and $\mathbf{\tilde Q_X}$ can be chosen to be the same matrix.  Under the alternative, the result is immediate, since if not then it holds that $F_X\simeq F_Y$ by the fact that $\kappa$ is universal.
\end{proof}

\subsection{Proofs of Corollaries} \label{sec:proof_cors}

The proof of Corollary \ref{cor:scaling} is now straightforward given the previous two proofs.

\begin{proof}[Proof of Corollary \ref{cor:scaling}]
We will highlight where the previous proof changes.  Examining the proof of Theorem \ref{thm:consistency_repeated}, we see that we have (under the new scaling $(m+n)$) the residual bounds
\begin{align*}
    r_1 = O\bigg( \frac{\log(n)}{\sqrt{n\alpha_n}} +  \frac{\log(m)}{\sqrt{m\beta_m}}\bigg).
\end{align*}
and similarly for $r_2$.  Furthermore, by the final statement in Lemma \ref{lem:fclt}, we have that 
\begin{align*}
    \| \xi - \hat \xi \|_{\mathcal{H}} &\leq  O\bigg( \frac{\sqrt{\log(n)}}{\sqrt{n} \alpha_n } +\frac{\sqrt{\log(m)}}{\sqrt{m} \beta_m }\bigg).
\end{align*} 
Under the null hypothesis, through a similar analysis, we have that
\begin{align*}
    \|\xi \|_{\mathcal{H}} &= O\bigg( \sqrt{\log(n) + \log(m)}\bigg).
\end{align*}
Therefore, with probability $1 - O(n^{-2} + m^{-2})$, we have the bound
\begin{align*}
    \| \xi - \hat \xi \| \bigg( 2 \| \xi \| + \| \xi - \hat \xi \| \bigg) &= O\bigg( \frac{\sqrt{\log(n)}}{\sqrt{n} \alpha_n } +\frac{\sqrt{\log(m)}}{\sqrt{m} \beta_m }\bigg) \bigg( \sqrt{\log(n)}  + \sqrt{\log(m)}\bigg) \\
    &= O\bigg( \frac{\log(n)}{\sqrt{n} \alpha_n} + \frac{\log(m)}{\sqrt{m}\beta_m} \bigg)
\end{align*}
since $m/(n+m) \to \rho \in (0,1)$.  Therefore, since $\min(\alpha_n,\beta_m) \gg n^{-1/2} \log(n)$, the right hand side tends to zero.  
\end{proof}

\noindent In order to prove Corollary \ref{cor:scaling_estimate}, we will need the following additional lemmas.  The first is straightforward and included in Section \ref{sec:proofs_lems} for completeness.

\begin{lemma} \label{sparsity_lemma}
When $\E(X_1\t \ipq X_2) = \frac{1}{2}$, we have with probability at least $1 - O(n^{-2})$ that
\begin{align*}
    \frac{1}{ \sqrt{\hat \alpha_n}} - \frac{1}{\sqrt{\alpha_n}} =      O\bigg(\frac{\sqrt{\log(n)}}{n^{1/2}\alpha_n} \bigg).
\end{align*}
\end{lemma}

\noindent 
The next lemma shows we can replace $\alpha_n$ with $\hat \alpha_n$ in the appropriate place and still maintain convergence in probability.  The proof is in Section \ref{sec:proof_frob} and is simply a modification of the proof of Lemma \ref{lem:fclt}.

\begin{lemma} \label{cor:fclt_estimate}
 Under the setting of Lemma  \ref{lem:fclt}, the limiting result holds with $\hat X_i/\alpha_n^{1/2}$ replaced with $\hat X_i/ \hat \alpha_n^{1/2}$ with the almost sure convergence replaced with convergence in probability.  
 \end{lemma} 

\begin{proof}[Proof of Corollary \ref{cor:scaling_estimate}]
Now, the result follows by simply noting that the functional CLT still holds in probability, and hence the result holds with $\xhat/\alpha_n^{1/2}$ and $\yhat/\beta_m^{1/2}$ replaced with $\xhat/\hat \alpha_n^{1/2}$ and $\yhat/\hat \beta_m^{1/2}$.
\end{proof}

\subsection{Proofs in \cref{sec:algorithm}} \label{sec:proofsprops}
In this section we prove \cref{thm:iterativeconvergence}, \cref{prop3}, and \cref{prop2}.  
\subsubsection{Proof of \cref{thm:iterativeconvergence}}
\begin{proof}[Proof of \cref{thm:iterativeconvergence}]
First, note that
\begin{align*}
    \| \mathbf{\hat W}_{\mathrm{init}}\t \mathbf{\hat X}\t \mathbf{\hat \Pi}^{(0)} \mathbf{\hat Y} - (\mathbf{\hat W}^{(\infty)})\t \mathbf{\hat X} \mathbf{\hat \Pi}^{(\infty)} \mathbf{\hat Y} \|_F \leq  C\| \mathbf{\hat W}_{\mathrm{init}}\t \xhat\t \mathbf{\hat \Pi}^{(0)} \mathbf{\hat Y} - (\mathbf{\hat W}^{(\infty)})\t \mathbf{\hat X}\t \mathbf{\hat \Pi}^{(\infty)} \mathbf{\hat Y} \|.
\end{align*}
Now consider the term on the right hand side.  Note that
\begin{align*}
    \| \mathbf{\hat W}_{\mathrm{init}}\t \xhat\t \mathbf{\hat \Pi}^{(0)} &\mathbf{\hat Y} - (\mathbf{\hat W}^{(\infty)})\t \mathbf{\hat X}\t \mathbf{\hat \Pi}^{(\infty)} \mathbf{\hat Y} \| \\
    &= \sup_{\|u\|=\|v\| = 1} \bigg| \sum_{l,l'} u_l v_{l'} \bigg( \sum_{i,j} \big(  \xhat\mathbf{\hat W}_{\mathrm{init}}\big)_{il} \big(\mathbf{\hat \Pi}^{(0)}\big)_{ij} \big(\mathbf{\hat Y}\big)_{jl'} - (\mathbf{\hat X}\mathbf{\hat W}^{(\infty)})_{il} \mathbf{\hat \Pi}^{(\infty)}_{ij} \mathbf{\hat Y}_{il'} \bigg) \bigg|.
\end{align*}
\textcolor{black}{Let $W_1(\cdot,\cdot)$ denote the 1-Wasserstein distance between two probability measures}.    Above, if $u$ and $v$ are fixed, let $f(x,y)$ denote the constant function $f(x,y) = (u,v)$.  Then $f$ is clearly Lipschitz with Lipschitz parameter at most 1, and hence by the well-known duality of the 1-Wasserstein norm with supremums over 1-Lipschitz functions, it holds that
\begin{align*}
    \bigg| \sum_{l,l'} u_l v_{l'} \bigg( \sum_{i,j} \big(  \xhat\mathbf{\hat W}_{\mathrm{init}}\big)_{il} \big(\mathbf{\hat \Pi}^{(0)}\big)_{ij} \big(\mathbf{\hat Y}\big)_{jl'} - (\mathbf{\hat X}\mathbf{\hat W}^{(\infty)})_{il} \mathbf{\hat \Pi}^{(\infty)}_{ij} \mathbf{\hat Y}_{il'} \bigg) \bigg| &\leq W_1\big( \mathbf{\hat \Pi}^{(0)}, \mathbf{\hat \Pi}^{(\infty)} \big),
\end{align*}
where $\mathbf{\hat \Pi}^{(0)}$ is understood as the joint probability measure placing mass on $\mathbf{\hat X} \mathbf{\hat W}_{\mathrm{init}}$ and $\mathbf{\hat Y}$, with $\mathbf{\hat \Pi}^{(\infty)}$ defined similarly. 
Since the right hand side is independent of $u$ and $v$, it holds that
\begin{align*}
    \| \mathbf{\hat W}_{\mathrm{init}}\t \xhat\t \mathbf{\hat \Pi}^{(0)} \mathbf{\hat Y} - (\mathbf{\hat W}^{(\infty)})\t \mathbf{\hat X}\t \mathbf{\hat \Pi}^{(\infty)} \mathbf{\hat Y} \|_F &\leq C W_1\big( \mathbf{\hat \Pi}^{(0)}, \mathbf{\hat \Pi}^{(\infty)} \big) \\
    &\leq C \big( W_1( \hat F_{\mathbf{\hat X}\mathbf{\hat W}_{\mathrm{init}}}, \hat F_{\mathbf{\hat X}\mathbf{\hat W}^{(\infty)}}) + W_1(\hat F_{\mathbf{\hat Y}},\hat F_{\mathbf{\hat Y}}) \big) \\
    &\leq C W_2 ( \hat F_{\mathbf{\hat X}\mathbf{\hat W}_{\mathrm{init}}}, \hat F_{\mathbf{\hat X}\mathbf{\hat W}^{(\infty)}}) \\
    &\leq C' \| \mathbf{\hat W}_{\mathrm{init}} - \mathbf{\hat W}^{(\infty)} \|_F,
\end{align*}
where the second inequality is due to Theorem 3 of \citet{deligiannidis_quantitative_2024}, and we implicitly used the fact that $\mathbf{\hat W}^{(\infty)}, \mathbf{\hat \Pi}^{(\infty)}$ is a fixed point; \textcolor{black}{namely, that $\mathbf{\hat \Pi}^{(\infty)}$ is the optimal coupling matrix aligning $\mathbf{\hat X}\mathbf{\hat W}^{(\infty)}$ and $\mathbf{\hat Y}$}.   \textcolor{black}{By the assumption of the theorem,} $\sigma_d\bigg( \mathbf{\hat X}\t \mathbf{\hat \Pi}^{(\infty)} \mathbf{\hat Y} \bigg) \geq C_0$, \textcolor{black}{where $\sigma_d(\cdot)$ denotes the $d$'th largest singular value of a matrix.}
Therefore, by Theorem 2.1 of \citet{chun-hui_perturbation_1989}, it holds that
\begin{align*}
    \| \mathbf{\hat W}_1 - \mathbf{\hat W}^{(\infty)} \|_F &\leq \frac{2}{C_0} C'  \| \mathbf{\hat W}_{\mathrm{init}} - \mathbf{\hat W}^{(\infty)} \|_F,
\end{align*}
and hence as long as $C_0 \geq 4 C'$, we obtain the upper bound
\begin{align*}
       \| \mathbf{\hat W}_1 - \mathbf{\hat W}^{(\infty)} \|_F &\leq \frac{1}{2}\| \mathbf{\hat W}_{\mathrm{init}} - \mathbf{\hat W}^{(\infty)} \|_F.  
\end{align*}
The result is completed by repeating the argument by induction.  
\end{proof}

\subsubsection{Proof of \cref{prop3}} \label{sec:proofs_props_34}
\begin{proof}[Proof of \cref{prop3}]
Suppose first that the sparsity factors $\alpha_n$ and $\beta_m$ are known. By Lemmas \ref{lem2} and \ref{lem:qtilde}, we have that under the null hypothesis there exist sequences of orthogonal matrices $\mathbf{W_X}$ and $\mathbf{W_Y}$ such that with probability at least $1 - n^{-2} - m^{-2}$
\begin{align}
\|\mathbf{Q_X-W_X\Tilde{Q}}\| = O\left( \frac{\sqrt{\log(n)}}{\sqrt{n}}\right); &\qquad \|\mathbf{Q_Y T^{-1}-W_Y\Tilde{Q}}\| = O\left( \frac{\sqrt{\log(m)}}{\sqrt{m}}\right). \label{qtilde}
\end{align}
Furthermore, by Theorem \ref{thm:grdpg}, there exists block-orthogonal $\mathbf{W_*^X}$ and $\mathbf{W_*^Y}$ (depending on $n$ and $m$) such that with probability $1 - n^{-2} - m^{-2}$
\begin{align}
    \|\mathbf{\hat{X}}-\alpha_n^{1/2} \mathbf{X Q_X\inv W_*^X}\|_{2,\infty} = O\bigg( \frac{\log(n)}{\sqrt{n}} \bigg);  \qquad \|\mathbf{\hat{Y}}-\beta_m^{1/2} \mathbf{Y Q_Y\inv W_*^Y}\|_{2,\infty} = O\bigg( \frac{\log(m)}{\sqrt{m}} \bigg) \label{twoinfinity}
\end{align}
(note that $\mathbf{W_*^X}$ is taken to be the transpose of the block-orthogonal matrix from \cref{thm:grdpg}).
Define the event $\mathcal{A} := \{$\eqref{qtilde}   and \eqref{twoinfinity} hold.$\}$  Note that $\p(\mathcal{A}) \geq 1 - 2n^{-2} - 2m^{-2}$.  Note that on $\mathcal{A}$ we have that
\begin{align*}
    \| \xhat - \alpha_n^{1/2} \mathbf{X\tilde Q\inv W_X\t W_*^X}\|_{2,\infty} &= O\bigg( \frac{\log(n)}{\sqrt{n}} \bigg);  \\\|\mathbf{\hat{Y}}-\beta_m^{1/2} \mathbf{Y T\inv \tilde Q\inv W_Y\t W_*^Y}\|_{2,\infty} &= O\bigg( \frac{\log(m)}{\sqrt{m}} \bigg).
\end{align*}

Define $\Gamma_{\hat X,\hat Y}$ to be the set of couplings of $\hat F_{\hat  X/\alpha_n^{1/2}}$ and $\hat F_{\hat  Y/\beta_m^{1/2}}$.  
We have that for any block-orthogonal $\mathbf{W}_n$, the minimizer $\mathbf{\hat W}_n$ satisfies
\begin{align*}
        d_2(\hat{F}_{\hat{X}/\alpha_n^{1/2}},\hat{F}_{\hat{Y}/\beta_m^{1/2}}\circ \mathbf{\hat{W}}_n)&\leq d_2(\hat{F}_{\hat{X}/\alpha_n^{1/2}},\hat{F}_{\hat{Y}/\beta_m^{1/2}}\circ \mathbf{W}_n) 
\end{align*}
To show this tends to zero, we choose an appropriate block-orthogonal matrix $\mathbf{W}_n$.  Define
\begin{align*}
    \mathbf{W}_n := \mathbf{(W_*^X W_X)\t ( W_Y W_*^Y)}.
\end{align*}
Note that under the null hypothesis $F_X \simeq F_Y$ we have that $F_X \circ  \mathbf{\tilde Q\inv} = F_Y \circ \mathbf{T\inv} \circ \mathbf{\tilde Q}\inv$ for some $\mathbf{T} \in \mathbb{O}(p,q)$.  Then, by the rotational invariance of the Euclidean norm, we have that,
\begin{align*}
d_2(\hat{F}_{\hat{X}/\alpha_n^{1/2}} \circ \mathbf{W}_n,\hat{F}_{\hat{Y}/\beta_m^{1/2}}) 
&= d_2( \hat F_{\hat  X/\alpha_n^{1/2}} \circ \mathbf{(W_X\t W_*^X})\t, \hat F_{\hat  Y/\beta_m^{1/2}} \circ \mathbf{( W_Y W_*^Y)\t} ) \\
&\leq d_2( \hat F_{\hat  X/\alpha_n^{1/2}} \circ \mathbf{(W_X\t W_*^X})\t, \hat F_{X} \circ \mathbf{\tilde  Q}\inv) \\
&\quad+ d_2( \hat F_{\hat  Y/\beta_m^{1/2}} \circ \mathbf{( W_Y W_*^Y)\t}, \hat F_{Y} \circ \mathbf{T\inv \tilde Q}\inv ) \\
&\quad + d_2 ( \hat F_{X} \circ \mathbf{\tilde Q}\inv, F_{X} \circ \mathbf{\tilde Q} \inv ) +  d_2 ( \hat F_{Y} \circ \mathbf{T\inv\tilde Q}\inv, F_{Y} \circ \mathbf{T\inv \tilde Q} \inv ) 
\end{align*}
We show each term is small.  For the first two terms, note that these are the empirical CDFs of the points $\hat X_i$ and $X_i$, where the $X_i$ aare suitably transformed but fixed.  Consider the coupling $\gamma$ which places mass of $\frac{1}{n}$ at the joint observation $(\hat X_i,  X_i)$.  Then on the event $\mathcal{A}$,
\begin{align*}
    d_2( \hat F_{\hat  X/\alpha_n^{1/2}} \circ \mathbf{(W_X\t W_*^X})\t), \hat F_{X} \circ \mathbf{\tilde  Q}\inv) &\leq 
         \left(\frac{1}{n}\|\xhat\mathbf{(W_X\t W_*^X})\t /\sqrt{\alpha_n}-\mathbf{X}\mathbf{\tilde Q\inv}\|_F^2\right)^{\frac{1}{2}}\\
         &=\frac{1}{\sqrt{n}}\|\xhat /\sqrt{\alpha_n}- \mathbf{X}\mathbf{\tilde Q\inv}\mathbf{W_X\t W_*^X}\|_F\\
         &\leq \|\xhat/\sqrt{\alpha_n}- \mathbf{X}\mathbf{\tilde Q\inv}\mathbf{W_X\t W_*^X}\|_{2,\infty}\\
         &= O\left( \frac{\log(n)}{(n\alpha_n)^{1/2}}\right)
        \end{align*}
Similarly,
\begin{align*}
    d_2( \hat F_{\hat  Y/\beta_m^{1/2}} \circ \mathbf{(W_*^Y W_Y)}, \hat F_{Y} \circ \mathbf{T\inv \tilde Q}\inv ) &= O\left( \frac{\log(m)}{(m\beta_m)^{1/2}}\right).
\end{align*}
Finally, we bound the final two terms.  By assumption $\supp(F)$ is bounded and hence so is any fixed invertible linear transformation of $\supp(F)$.  Hence, since $\|X\|_{\infty} \leq M$ almost surely, we can apply Theorem 2 of \citet{fournier_rate_2015} to see that with probability $1- n^{-2}$ that
\begin{align*}
   d_2 ( \hat F_{X} \circ \mathbf{\tilde Q}\inv, F_{X} \circ \mathbf{\tilde Q} \inv ) = O\bigg( \frac{\log^{1/d}(n)}{n^{1/d}}\bigg).
\end{align*}
Similarly, 
\begin{align*}
    d_2 ( \hat F_{Y} \circ \mathbf{T\inv\tilde Q}\inv, F_{Y} \circ \mathbf{T\inv \tilde Q} \inv )  = O\bigg( \frac{\log^{1/d}(m)}{m^{1/d}}\bigg)
\end{align*}
with probability $1 - m^{-2}$.  Therefore, putting it all together, we see that with probability $1 - O(n^{-2} + m^{-2})$,
\begin{align*}
    d_2(\hat F_{\hat  X/\alpha_n^{1/2}}, \hat F_{\hat  Y/\beta_m^{1/2}} \circ \mathbf{\hat W}_n ) = O\left( \frac{\log^{1/d}(n)}{n^{1/d}} +\frac{\log^{1/d}(m)}{m^{1/d}} + \frac{\log(n)}{(n\alpha_n)^{1/2}} + \frac{\log(m)}{(m\beta_m)^{1/2}}\right).
\end{align*}

Finally, if the sparsity is not known, we have that  by Lemma \ref{sparsity_lemma},  $\frac{1}{\sqrt{\hat\alpha_n}} = \frac{1}{\sqrt{\alpha_n}}\bigg( 1 +O\bigg(\frac{\sqrt{\log(n)}}{n^{1/2}\alpha_n} \bigg)\bigg)$ with probability at least $1 - O(n^{-2})$. Hence, we see that
\begin{align*}
    \| \xhat (\hat \alpha_n^{-1/2} - \alpha_n^{-1/2}) \|_{2,\infty} &\leq   \| \xhat  \|_{2,\infty} O\bigg(\frac{\sqrt{\log(n)}}{n^{1/2}\alpha_n} \bigg) \\
    &= O\bigg(\frac{\sqrt{\log(n)}}{n^{1/2}\alpha_n} \bigg)\bigg(\| \xhat - \sqrt{\alpha_n}\mathbf{XQ_X\inv W_*^X} \|_{2,\infty} + \| \sqrt{\alpha_n}\mathbf{XQ_X\inv W_*^X} \|_{2,\infty} \bigg) \\
    &=  O\bigg(\frac{\sqrt{\log(n)}}{n^{1/2}\alpha_n} \bigg) \bigg( \frac{\log(n)}{\sqrt{n}} + O(\sqrt{\alpha_n}) \bigg) \\
    &= O\bigg( \frac{\log(n)^{3/2}}{ n\alpha_n} \bigg) + O\bigg( \sqrt{\frac{\log(n)}{ n\alpha_n}}\bigg) \\
    &= O\bigg(\sqrt{ \frac{\log(n)}{n\alpha_n}} \bigg).
\end{align*}
Therefore, by analogous arguments as in the setting with the sparsity factors known, replacing $\hat \alpha_n^{-1/2}$ and $\hat \beta_m^{-1/2}$ with $\alpha_n^{-1/2}$ and $\beta_m^{-1/2}$ is negligible compared to the $\ell_{2,\infty}$ bound.  Hence, we have that with probability at least $1 - O(n^{-2} + m^{-2})$
\begin{align*}
    d_2(\hat F_{\hat X/\hat\alpha_n^{1/2}}, \hat F_{\hat Y/\hat \beta_m^{1/2}} \circ \mathbf{\hat W}_n ) &= O\left( \frac{\log^{1/d}(n)}{n^{1/d}} +\frac{\log^{1/d}(m)}{m^{1/d}} + \frac{\log(n)}{(n\alpha_n)^{1/2}} + \frac{\log(m)}{(m\beta_m)^{1/2}}\right).
\end{align*}

\end{proof}

\subsubsection{Proof of Proposition \ref{prop2}}
\label{sec:proofs_props_12}

\begin{proof}[Proof of Proposition \ref{prop2}]
We will show the result holds for any two sequences of block-orthogonal matrices $\mathbf{W}_n^1$ and $\mathbf{W}_n^2$, since the result follows by taking $\mathbf{W}_n := \mathbf{W}_n^1 (\mathbf{W}_n^2)\t$.  
We note that it suffices to prove the result with the replacement $\mathbf{XQ_X\inv}$ and $\mathbf{YQ_Y\inv}$ instead of $\xhat/\alpha_n^{1/2}$ and $\yhat/\beta_m^{1/2}$, since the $\ell_{2,\infty}$ result in Theorem \ref{thm:grdpg}  and the Lipschitz property of $\kappa$ shows that 
\begin{align*}
  |  U_{n,m}(\xhat \mathbf{W_*^X} \mathbf{W}_n^1/\alpha_n^{1/2}, \yhat \wy \mathbf{W}_n^2/\beta_m^{1/2}) - U_{n,m}( \mathbf{X  Q_X\inv} \mathbf{W}_n^1, \mathbf{YQ_Y\inv}\mathbf{W}_n^2) | \to 0
\end{align*}
for any sequence of block-orthogonal matrices $\mathbf{W}_n^1$ and $\mathbf{W}_n^2$.

Suppose $F_X \not \simeq F_Y$, but suppose for contradiction that $$\liminf U_{n,m}( \mathbf{X Q_X\inv} \mathbf{W}_n^1, \mathbf{Y Q_Y\inv} \mathbf{W}_n^2) = 0$$ for some sequence of block-orthogonal matrices $\mathbf{W}_n^1, \mathbf{W}_n^2$.  By passing to a convergent subsequence (by compactness of the orthogonal group and the fact that $\mathbb{O}(p,q) \cap \mathbb{O}(d)$ is a subgroup), we may assume the limit exists. Let $\mathbf{\tilde Q_X\inv}$ and $\mathbf{\tilde Q_Y\inv}$ be the limiting matrices given by Lemma \ref{lem2}, and similarly for the sequences of block orthogonal matrices $\mathbf{W^{X}}$ and $\mathbf{W^{Y}}$.  
We have that
\begin{align}
    |U_{n,m}( \mathbf{X \tilde Q_X\inv} \mathbf{W}_n^1, \mathbf{Y \tilde Q_Y\inv}  \mathbf{W}_n^2) | &= |U_{n,m}( \mathbf{X \tilde Q_X\inv W_{X} (W_{X})\t}\mathbf{ W}_n^1, \mathbf{Y \tilde Q_Y\inv W_{Y} (W_{Y})\t} \mathbf{ W}_n^2) | \nonumber
    \\
    &\leq \bigg|U_{n,m}( \mathbf{X \tilde Q_X\inv W_{X} (W_{X})\t}\mathbf{ W}_n^1 ,\mathbf{Y \tilde Q_Y\inv W_{Y} (W_{Y})\t } \mathbf{ W}_n^2)  \nonumber \\
    &\qquad  \qquad \qquad    -U_{n,m}( \mathbf{X  Q_X\inv (W_{X})\t}\mathbf{ W}_n^1, \mathbf{Y Q_Y\inv (W_{Y})\t } \mathbf{ W}_n^2)  \bigg|\nonumber \\
    &\qquad
    + |U_{n,m}( \mathbf{X  Q_X\inv (W_{X})\t }\mathbf{ W}_n^1, \mathbf{Y Q_Y\inv (W_{Y})\t } \mathbf{ W}_n^2)  |. \label{ineq1}
\end{align}
We note that
\begin{align*}
    \bigg|U_{n,m}( \mathbf{X \tilde Q_X\inv W_{X} (W_{X})\t} \mathbf{W}_n^1, \mathbf{Y \tilde Q_Y\inv W_{Y} (W_{Y})\t } \mathbf{W}_n^2) -U_{n,m}( \mathbf{X  Q_X\inv (W_{X})\t } \mathbf{W}_n^1, \mathbf{ Y Q_Y\inv (W_{Y})\t } \mathbf{W}_n^2) \bigg|
\end{align*}
tends to zero by Lemma \ref{lem:qtilde}. 
Therefore, we see that 
\begin{align*}
    \limsup | U_{n,m}( \mathbf{X \tilde Q_X\inv } \mathbf{W}_n^1,  \mathbf{Y\tilde Q_Y\inv } \mathbf{W}_n^2)| &\leq \liminf |U_{n,m}( \mathbf{XQ_X\inv (W_X)\t } \mathbf{ W}_n^1,\mathbf{ YQ_Y\inv (W_Y)\t } \mathbf{W}_n^2) | 
\end{align*}
where by assumption the right hand side is presumed to exist. Note that the only terms on the left hand side that are random are the matrices $\mathbf{X}$ and $\mathbf{Y}$ whose rows are drawn i.i.d. $F_X$ and $F_Y$ respectively.  If the term on the right hand side tends to zero (which it does by assumption), then that implies that
\begin{align*} 
\limsup | U_{n,m}( \mathbf{X \tilde Q_X\inv } \mathbf{W}_n^1, \mathbf{ Y\tilde Q_Y\inv} \mathbf{ W}_n^2)|  \to 0.
\end{align*}
This implies that
\begin{align*}
    \| \mu[ F_X \circ \mathbf{\tilde Q_X\inv } \mathbf{W}_n^1 ]- \mu[F_Y \circ \mathbf{\tilde Q_Y\inv } \mathbf{W}_n^2 ] \|_{\mathcal{H}}^2 \to 0,
\end{align*}
by, e.g. Remark 1 in \citet{schneider_probability_2016} (as in the proof of Theorem \ref{thm:consistency_repeated}). The only quantities that are changing in $n$ and $m$ are $\mathbf{W}_n^1$ and $\mathbf{W}_n^2$.  Define the map
\begin{align}
    \mathbf{W} \mapsto \| \mu[ F_X \circ  \mathbf{\tilde Q_X\inv} ]- \mu[F_Y \circ \mathbf{\tilde Q_Y\inv W} ] \|_{\mathcal{H}}^2. \label{map}
\end{align}
By Lemma 6 in \citet{gretton_kernel_2012}, we have that for any two distributions $X\sim F$ and $Y \sim G$ that
\begin{align*}
    \| \mu[ F]- \mu[G \circ \mathbf{W}_1 ] \|_{\mathcal{H}}^2 - \| \mu[ F]- \mu[G \circ \mathbf{W}_2 ] \|_{\mathcal{H}}^2  &= 2\E_{F,G}\left[ \kappa(X,\mathbf{W}_1\t Y)- \kappa(X, \mathbf{W_2}\t Y) \right]\\
    &\quad + \E_{G}\left[ \kappa(\mathbf{W}_1\t Y, \mathbf{W}_1\t Y') - \kappa(\mathbf{W}_2\t Y, \mathbf{W_2\t}Y') \right]
\end{align*}
by the definition that $Y \sim G \circ \mathbf{W}$ if $\mathbf{W\t}Y \sim G$.  Hence, continuity of the map  in \eqref{map} follows from continuity of $\kappa$.  Therefore, by the assumption that $\kappa$ is radial, we see that since
\begin{align*}
      \| \mu[ F_X \circ \mathbf{\tilde Q_X\inv } ]- \mu[F_Y \circ \mathbf{\tilde Q_Y\inv W}_n^2(\mathbf{W}_n^1)\t ] \|_{\mathcal{H}}^2 \to 0,
\end{align*}
we must have that some subsequence of $\mathbf{W}_n^2(\mathbf{W}_n^1)\t$ is converging (since the set $\mathbb{O}(p,q)\cap \mathbb{O}(d)$ is compact).  Let $\mathbf{\tilde W}$ be this subsequential limit.  Then this implies that 
\begin{align*}
     \| \mu[ F_X \circ \mathbf{\tilde Q_X\inv } ]- \mu[F_Y \circ \mathbf{\tilde Q_Y\inv \tilde W} ] \|_{\mathcal{H}}^2 = 0.
\end{align*}
However, under the alternative, since $\kappa$ is characteristic, we have that $\mu[F_X] \neq \mu[F_Y \circ \mathbf{T}]$ for any $\mathbf{T} \in \mathbb{O}(p,q)$.  But then the above equation is a contradiction.   Furthermore, working backwards, we have the chain of inequalities
\begin{align*}
     \inf_{\mathbf{W} \in \mathbb{O}(d)\cap\mathbb{O}(p,q)} \| \mu[F_X \circ \mathbf{\tilde Q_X\inv}]& - \mu[F_Y \circ \mathbf{\tilde Q_Y\inv W}] \|_{\mathcal{H}}^2 \\&\leq \liminf | U_{n,m}( \mathbf{X \tilde Q_X\inv } \mathbf{W}_n^1, \mathbf{ Y\tilde Q_Y\inv} \mathbf{ W}_n^2)| \\&\leq \limsup | U_{n,m}( \mathbf{X \tilde Q_X\inv } \mathbf{W}_n^1, \mathbf{ Y\tilde Q_Y\inv} \mathbf{ W}_n^2)| \\
     &\leq \liminf |U_{n,m}( \mathbf{XQ_X\inv (W_X)\t } \mathbf{ W}_n^1,\mathbf{ YQ_Y\inv (W_Y)\t } \mathbf{W}_n^2) |,
\end{align*}
which shows that $$C := \inf_{\mathbf{W} \in \mathbb{O}(d)\cap\mathbb{O}(p,q)} \| \mu[F_X \circ \mathbf{\tilde Q_X\inv}] - \mu[F_Y \circ \mathbf{\tilde Q_Y\inv W}] \|_{\mathcal{H}}^2 $$ is a lower bound independent of the particular sequence $\mathbf{W}_n^1, \mathbf{W}_n^2$, \textcolor{black}{which is strictly larger than zero under the alternative since the infimum is over all $\mathbf{W} \in \mathbb{O}(p,q)$}. This proves the result.
\end{proof}

\subsection{Proof of the Frobenius Concentration (Lemma \ref{lem:frobenius})} \label{sec:proof_frob}
In this section we present the proof of Lemma \ref{lem:frobenius}.  
We will need the following lemma, adapted from Lemma A.4 \citet{athreya_numerical_2023}.  

\begin{lemma}\label{lem:numtol}
Let $\mathbf{A}$ be a matrix whose entries are generated via $\mathbf{A}_{ij} \sim Bernoulli( \mathbf{P}_{ij})$ for $i \leq j$, and let $\mathbf{V} = \mathbf{U_X |\Lambda_X|^{-1/2}}$ .  Then with probability at least $1 - O(n^{-2})$,
\begin{align*}
  \bigg|  \|(\mathbf{A} - \mathbf{P}) \mathbf{V} \|_F^2 - \E (  \|(\mathbf{A} - \mathbf{P}) \mathbf{V} \|_F^2 ) \bigg|= O \bigg(\sqrt{\frac{\log(n)}{n\alpha_n }}\bigg).
\end{align*}
\end{lemma}

\begin{proof}[Proof of Lemma \ref{lem:numtol}]
We follow the proof in \citet{athreya_numerical_2023}, though we have a slightly different argument for the inclusion of the sparsity factor. Let $\mathbf{A'} \sim \mathbf{P}$ be independent from $\mathbf{A}$. 
For $1\leq r\leq s\leq n$ define the term
\begin{align*}
    Z_{rs} := \| (\mathbf{A}^{(r,s)} - \mathbf{P)V}\|_F^2,
\end{align*}
where the matrix $\mathbf{A}^{(r,s)}$ agrees with $\mathbf{A}$ in every entry except the $(r,s)$ and $(s,r)$ones, where it equals $\mathbf{A'}$. Defining $Z := \mathbf{\| (A -P) V\|}_F^2$, we see that for $r \neq s$
\begin{align*}
    Z - Z_{rs} &= 2\left(\mathbf{A-A}^{\prime}\right)_{r s}\left[\left[\mathbf{(A-P) V V}\t\right]_{r s}+\left[\mathbf{(A-P) V V}\t\right]_{s r}+\left(\mathbf{A}^{\prime}-\mathbf{P}\right)_{r s}\left(\left(\mathbf{V V}\t\right)_{s s}+\left(\mathbf{V V}\t\right)_{r r}\right)\right]
\end{align*}
Furthermore, we have that
\begin{align*}
    (Z - Z_{rs})^2 \leq \begin{cases} 16 \bigg( [\mathbf{(A - P)VV\t}]^2_{rs} + [\mathbf{(A - P)VV\t}]^2_{sr} +\left(\mathbf{A}^{\prime}-\mathbf{P}\right)_{r s}^2 \big[ (\mathbf{VV\t})_{ss}^2 + \mathbf{VV\t})_{rr}^2\big] \bigg) & r\neq s\\
    8 \bigg( [\mathbf{(A - P)VV}\t]_{rr}^2 +\left(\mathbf{A}^{\prime}-\mathbf{P}\right)_{r s}^2  [\mathbf{VV}\t]_{rr}^2\bigg) & r= s\end{cases}
\end{align*}
Hence, 
\begin{align*}
    \sum_{r\leq s} (Z - Z_{rs})^2 &\leq 16 Z \| \mathbf{V} \|_2^2 + 8 \sum_{r} (\mathbf{VV\t})_{rr}^2 + 16 \sum_{r < s} (\mathbf{A' - P})_{rs}^2 [(\mathbf{VV\t}_{ss}^2 + (\mathbf{VV\t})_{rr}^2] \\
    &= 16Z \| \mathbf{V} \|_2^2 + 8\| \text{diag}(\mathbf{VV\t}) \|_F^2 + 16\sum_{s=1}^{n} \sum_{r=1}^{s-1}( \mathbf{A' - P})_{rs}^2[ (\mathbf{VV\t})_{ss}^2 + (\mathbf{VV\t})_{rr}^2].
\end{align*}
For the final term above, we have that
\begin{align*}
\E_{\mathbf{A'}} \bigg[ 
\sum_{s=1}^{n} \sum_{r=1}^{s-1}( \mathbf{A' - P})_{rs}^2[ (\mathbf{VV\t})_{ss}^2 + (\mathbf{VV\t})_{rr}^2]\bigg] &= \sum_{s=1}^{n} \sum_{r=1}^{s-1} \E_{\mathbf{A'}}( \mathbf{A' - P})_{rs}^2 [ (\mathbf{VV\t})_{ss}^2 + (\mathbf{VV\t})_{rr}^2] \\
&\leq 2 n\alpha_n \| \text{diag}(\mathbf{VV\t}) \|_F^2.
\end{align*}
Moreover, from the definitions of $\mathbf{V}$, we have that $\|\text{diag}( \mathbf{VV\t})\|^2_F \leq d \lambda_d^{-2} \leq Cd(n\alpha_n)^{-2}$, and $\|\mathbf{V}\|_2^2 = |\lambda_d|\inv \leq C(n\alpha_n)\inv$.  Hence, we see that
\begin{align*}
     \E_{\mathbf{A'}} \sum_{r\leq s} (Z - Z_{rs})^2 &\leq \frac{C_1}{n\alpha_n}Z + \frac{C_2d}{(n\alpha_n)^2} + \frac{C_3d}{n\alpha_n},
\end{align*}
Define $a :=\frac{C_1}{n\alpha_n}$ and $b := \frac{C_2d}{(n\alpha_n)^2} + \frac{C_3d}{n\alpha_n}$. By Theorems 5 and 6 in \citet{boucheron_concentration_2003}, we have that
\begin{align*}
    \p( | Z - \E Z| > t ) &\leq 2 \exp\bigg( \frac{-t^2}{4a \E(Z) + 4b + 2at} \bigg) \\
    &\leq 2 \exp\bigg( \frac{ -t^2 n\alpha_n}{ 4C_1 \E(Z) + 4C_2 d(n\alpha_n)\inv + 4C_3d + 2C_1 t} \bigg)\\
    &\leq 2\exp \bigg( \frac{ -t^2 n\alpha_n}{ \tilde C_1 + \tilde C_2 t}\bigg)
\end{align*}
for some constants $\tilde C_1$ and $\tilde C_2$ (depending on $d$) and for $n\alpha_n$ sufficiently large.  Note that we implicitly used that $\E Z = O(1)$, which can be seen from the fact that
\begin{align*}
    \E(Z) &= \E \| (\mathbf{A} - \mathbf{P)V}\|_F^2 \\
    &= \sum_{i=1}^{n} \sum_{k=1}^{d} \E \bigg( \sum_{j} ( \mathbf{A}_{ij} - \mathbf{P}_{ij}) \mathbf{V}_{jk}\bigg)^2 \\
    &= \sum_{i=1}^{n} \sum_{k=1}^{d} \sum_{j=1}^{n} \mathbf{V}_{jk}^2 \E( ( \mathbf{A}_{ij} - \mathbf{P}_{ij}) )^2  +  \sum_{j\neq l}^{n} \mathbf{V}_{jk}\mathbf{V}_{lk}\E( ( \mathbf{A}_{ij} - \mathbf{P}_{ij})( \mathbf{A}_{il} - \mathbf{P}_{il}) ) \\
    &= \sum_{i=1}^{n} \sum_{k=1}^{d} \sum_{j=1}^{n} \mathbf{V}_{jk}^2 \mathbf{P}_{ij}(1- \mathbf{P}_{ij}) \\
    &\leq n \alpha_n  \sum_{k=1}^{d} \sum_{j=1}^{n}\mathbf{V}_{jk}^2 \\
    &\leq  n \alpha_n \| \mathbf{U_X |\Lambda_X|}^{-1/2} \|_F^2 \\
    &\leq C
\end{align*}
for some constant $C$ depending on $d$ and $\lambda_d$.  Hence, with the choice $t = \tilde C \sqrt{\frac{\log(n)}{n\alpha_n}}$ for some constant $\tilde C$ depending on $d$, this is bounded above by $2n^{-2}$.  
\end{proof}

\noindent We are now ready to prove Lemma \ref{lem:frobenius}.



\begin{proof}[Proof of Lemma \ref{lem:frobenius}]
First, by the proof of Theorem 5 in \citet{rubindelanchy_statistical_2022}, we note that there exists an orthogonal matrix $\mathbf{W_*} \in \mathbb{O}(d) \cap \mathbb{O}(p,q)$ (see equations 5 and 6 in \citet{rubindelanchy_statistical_2022}) such that
\begin{align*}
   \mathbf{ \hat U |\hat \Lambda|^{1/2} - U |\Lambda|^{1/2} W_*\t } &= \mathbf{(A -}  \mathbf{P)U|\Lambda|^{-1/2}W_*\t}\ipq + \mathbf{R},
\end{align*}
where the matrix $\mathbf{R}$ satisfies $$\| \mathbf{R} \|_{2,\infty} = O\bigg( \frac{d^{1/2}\log^{1/2}(n)}{\sqrt{n} (n\alpha_n)^{1/2}}\bigg).$$  Passing to the Frobenius norm, we see that $$\|\mathbf{R}\|_F = O\bigg( \frac{d^{1/2}\log^{1/2}(n)}{ (n\alpha_n)^{1/2}}\bigg).$$  
This proves the first claim.  Hence, 
\begin{align*}
    \|\xhat - \mathbf{\tilde X W_*\t} \|_F &= \| \mathbf{(A- P)U|\Lambda|^{-1/2}}\|_F + O\bigg(  \sqrt{\frac{d\log (n)}{ n\alpha_n}}\bigg).
\end{align*}
We then can apply Lemma \ref{lem:numtol} to see that
 \begin{align*}
     \p\left( | \| \mathbf{(A - P)V}\|_F^2 - C(\mathbf{P})^2| > C  \sqrt{\log(n)/n\alpha_n }\right) = O(n^{-2}),
 \end{align*}
 where $\mathbf{V = U|\Lambda|^{-1/2}}$ and $C^2(\mathbf{P}) =\E \|(\mathbf{A - P)V}\|_F^2$. 
The rest of the proof is similar to \citet{tang_limit_2018}.  By similar manipulations as those leading to Equation 18 in \citet{rubindelanchy_statistical_2022}, we have that
\begin{align*}
    \mathbf{(A-  P) U|\Lambda|^{-1/2} W_*\t} \ipq = \alpha_n^{-1/2}\mathbf{(A - P) X (X\t X)\inv} \ipq \mathbf{Q_X}\inv. 
\end{align*}
By Lemma \ref{lem:qtilde}, we have that there exists a sequence of block-orthogonal matrices such that $\mathbf{W_n\t Q_X\inv} \to \mathbf{\tilde Q\inv}$ almost surely. Hence, we have that
\begin{align*}
  \mathbb{E}  \| \mathbf{(A-  P)U|\Lambda|^{-1/2}}\|_F^2 
    &= \frac{1}{\alpha_n}\tr \bigg( \mathbf{W_n\t Q_X\inv (X\t X)\inv X\t \E(A - P)^2 X (X\t X)\inv Q_X^{-\top} W_n}\bigg) \\
    &= \tr \bigg( \mathbf{W_n\t Q_X\inv} (n \mathbf{(X\t X})\inv ) \bigg[ \frac{ \mathbf{X\t \E(A - P)^2 X}}{n^2 \alpha_n} \bigg]  (n \mathbf{(X\t X)}\inv) \mathbf{ Q_X^{-\top} W_n}\bigg).
\end{align*}
By the strong law of large numbers the term $\mathbf{X\t X}/n \to \Delta$ almost surely, so $n \mathbf{(X\t X)}\inv \to \Delta\inv$ almost surely by the continuous mapping theorem.  In addition, we have that
\begin{align*}
    \frac{ \mathbf{X\t \E(A - P)^2 X}}{n^2 \alpha_n} &= \frac{1}{n^2 \alpha_n}\sum_{i=1}^{n} \sum_{k} X_i X_i\t( p_{ik}( 1- p_{ik})) \\
    &=  \frac{1}{n^2 }\sum_{i=1}^{n} \sum_{k} X_i X_i\t( X_i\t \ipq X_k - \alpha_n X_i\t \ipq X_k X_k\t \ipq X_i).
\end{align*}
As $n \to \infty$, this is tending to the matrix $\Gamma$, where $\Gamma$ is defined via
\begin{align*}
    \Gamma := \begin{cases} \E( XX\t (X\t \ipq \mu - X\t \ipq \Delta\ipq X )) & \alpha \equiv 1 \\
    \E( XX\t (X\t \ipq \mu)) &\alpha \to 0,
\end{cases}
\end{align*}
where $\mu = \E(X)$ and $\Delta= \E XX\t$.  Hence, putting it together, we have almost surely,
\begin{align*}
    \| \xhat - \mathbf{\tilde X W_*\t} \|_F^2 \to \tr \bigg( \mathbf{\tilde Q\inv \Delta\inv \Gamma \Delta\inv \tilde Q^{-\top}} \bigg).
\end{align*}
\end{proof}

\subsection{Proof of the Functional CLT (Lemma \ref{lem:fclt}) and Lemma \ref{cor:fclt_estimate}} \label{sec:proof_fclt}
In this section, we prove Lemma \ref{lem:fclt} and Lemma \ref{cor:fclt_estimate}.

\begin{proof}[Proof of Lemma \ref{lem:fclt}]
We follow the proof by analogy to the proof Lemma 3 of \citet{tang_nonparametric_2017}, though we use the decomposition from Lemma \ref{lem:frobenius}.  As every $f \in \mathcal{F}$ \textcolor{black}{is twice continuously differentiable}, for $f \in \mathcal{F}$ we Taylor expand to note that
\begin{align*}
    \frac{\sqrt{\alpha_n}}{\sqrt n}  &\sum_{i=1}^{n} \big( f( (\mathbf{W_*^X} )\t \hat X_i/\sqrt{\alpha_n}) - f( \tilde X_i/\sqrt{\alpha_n}) \big) \\&= \frac{\sqrt{\alpha_n}}{\sqrt n} \sum_{i=1}^{n} (\partial f)(\tilde X_i) \frac{((\mathbf{W_*^X} )\t \hat X_i - \tilde X_i )}{\sqrt \alpha_n} \\
    &\quad + \frac{\sqrt{\alpha_n}}{2\sqrt n} \sum_{i}  \frac{( (\mathbf{W_*^X} )\t\hat X_i - \tilde X_i )\t (\partial^2 f)(X_i^*) ((\mathbf{W_*^X} )\t \hat X_i - \tilde X_i) }{\alpha_n},
\end{align*}
for some $X_i^*$.  \textcolor{black}{Here and throughout our proofs the quantity $(\partial f)(\tilde X_i)$ is to be understood as a row vector of partial derivatives of $f$ evaluated at $\tilde X_i$, and $(\partial^2 f)(X_i^*)$ is to be understood as a matrix of all cross-derivatives.}  The second order term is straightforwardly bounded by noting that by assumption $\bar\Omega$ is compact, and hence we can apply Lemma \ref{lem:frobenius} to see that there exists some constant $C$ such that
\begin{align*}
    \sup_{f\in \mathcal{F}} \sum_{i=1}^{n} \frac{ ((\mathbf{W_*^X} )\t\hat X_i - \tilde X_i )\t (\partial^2 f)(X_i^*) ( (\mathbf{W_*^X} )\t\hat X_i - \tilde X_i) }{\sqrt{n\alpha_n}} &\leq \sup_{f \in \mathcal{F},X \in \bar \Omega} \frac{\| \partial^2(f)(X) \| \| (\mathbf{W_*^X} )\t\hat X_i -\tilde X_i \|_F^2}{\sqrt{n\alpha_n}} \\
    &\leq \frac{C}{\sqrt{n\alpha_n}},
\end{align*}
which converges to zero almost surely.  

We now bound the linear terms.  Let $\mathbf{M}(\partial f)  = \mathbf{M}(\partial f; \tilde X_1, \dots., \tilde X_n) \in \R^{n\times d}$ be the matrix whose rows are the vectors $\partial(f)(\tilde X_1)$.  Then
\begin{align*}
    \zeta(f) :&= \frac{1}{\sqrt{n}} \sum_{i=1}^{n} (\partial f)(\tilde X_i)\t ((\mathbf{W_*^X} )\t\hat X_i - \tilde X_i) \\
    &= \frac{1}{\sqrt{n}} \tr\big( (\xhat \mathbf{W_*^X}  - \mathbf{\tilde X)[M}(\partial f)]\t \big) \\
    &= \frac{1}{\sqrt{n}} \tr\big( \mathbf{(A - P)U |\Lambda|^{-1/2}}\ipq)[\mathbf{M}(\partial f)]\t \big) + \frac{1}{\sqrt{n}} \tr\big( \mathbf{R}[\mathbf{M}(\partial f)]\t \big),
\end{align*}
where $\mathbf{R}$ is the residual matrix in \ref{lem:frobenius}.  Recall the second term satisfies
\begin{align*}
   \frac{1}{\sqrt{n}} \tr\big( \mathbf{R}[\mathbf{M}(\partial f)]\t \big) &=  \frac{1}{\sqrt{n}} \langle \mathbf{R},\mathbf{M}(\partial f) \rangle \\
   &\leq \sup_{f\in\mathcal{F}, X\in \bar \Omega} \frac{ \sqrt{n} \| \partial f(X) \| }{\sqrt{n}} \| \mathbf{R}\|_F \\
   &\leq \frac{C \sqrt{n}}{\sqrt{n}}\| \mathbf{R}\|_F \\
   &\leq \frac{C \sqrt{\log(n)}}{\sqrt{n\alpha_n}},
\end{align*}
where the penultimate inequality comes from the fact that $\bar \Omega$ can be taken to be compact by assumption, and since $\mathcal{F}$ is twice-continuously differentiable, the gradient is Lipschitz on any fixed transformation of the support.  

Now, we show that the final term converges to zero.  
The rest of the proof is largely the same as in \citet{tang_nonparametric_2017}. 
Define the set of derivatives of $\partial \mathcal{F}:= \{\partial f: f\in\mathcal{F}\}$.  Let $\| \partial f\|_{\infty}$ be the maximum Euclidean norm attained by $f$ on $\bar \Omega$.  Note that by enlarging it if necessary, $\bar \Omega$ can be taken to be compact and contain the $\hat X_i$'s by the fact that $\kappa$ is twice continuously differentiable on $\R^d$ and by virtue of Theorem 
 \ref{thm:grdpg} and Lemma \ref{lem:qtilde}.  Therefore the set of derivatives is totally bounded; define $M := \sup_{\partial \mathcal{F}} \| \partial f\|_{\infty}$.  

Then for any $j$ there exists a finite subset $S_j$ of covering functions such that for any $g \in \partial \mathcal{F}$, we have that $\| g - f_j\|_{\infty} \leq 2^{-j} M$.   Define the mapping $\mathcal{P}_j$ as the mapping that assigns the function $g \in \partial \mathcal{F}$ to its closest function $f_j \in S_j$.  Then we have that
\begin{align*}
  \sup_{f} \bigg|  \frac{1}{\sqrt{n}} \tr&\big( (\mathbf{A - P)U|\Lambda^{-1/2}}\ipq ) [\mathbf{M}(\partial f)]\t\big) \bigg| \\&= \sup_{f} \bigg| \frac{1}{\sqrt{n}} \sum_{i=1}^{n} \sum_{j=0}^{\infty} (\mathcal{P}_{j+1}\partial f - \mathcal{P}_j \partial f)(\tilde X_i)\t \big[ ( \mathbf{A- P)U|\Lambda|^{-1/2}}\ipq \big]_i \bigg|\\
  &\leq \sum_{j=0}^{\infty}  \sup_{f} \bigg| \sum_{i=1}^{n} (\mathcal{P}_{j+1}\partial f - \mathcal{P}_j \partial f)(\tilde X_i)\t \big[ ( \mathbf{A- P)U|\Lambda|^{-1/2}}\ipq \big]_i \bigg| \\
  &+ \frac{1}{\sqrt{n}}\bigg|\sum_{i=1}^{n} (\mathcal{P}_0 \partial f)(\tilde X_i)\t \big[ ( \mathbf{A- P)U|\Lambda|^{-1/2}}\ipq \big]_i \bigg|.
\end{align*} 
We note that for fixed $j$, defining the term $\mathfrak{P}_j^f$ as the $n\times d$ matrix whose rows are $(\mathcal{P}_{j+1}\partial f - \mathcal{P}_j \partial f)(\tilde X_i)\t$, we have that
\begin{align*}
\frac{1}{\sqrt{n}} \sum_{i=1}^n \bigg|(\mathcal{P}_{j+1}\partial f - \mathcal{P}_j \partial f)(\tilde X_i)\t \big[ ( \mathbf{A- P)U|\Lambda|^{-1/2}} \big]_i \bigg| &= \frac{1}{\sqrt{n}}\bigg| \sum_{s=1}^d ((\mathfrak{P}_j^f) \t \big( (\mathbf{A - P}) \mathbf{U}_s \frac{- \mathbb{I}_{s > p} +  \mathbb{I}_{s \leq p}}{\sqrt{|\lambda_s|^{1/2}}}\big) \bigg|.
\end{align*}
We note that 
\begin{align*}
    \| (\mathfrak{P}_f^j)_s \| \leq \frac{3}{2} 2^{-j} M \sqrt{n}.
\end{align*}
Hence, for fixed $s \in \{1 \dots, d\}$ this is a linear combination of mean-zero random variables.  Therefore we have by Hoeffding's inequality that
\begin{align*}
    \p\bigg( ((\mathfrak{P}_j^f)_s \t \big( (\mathbf{A - P}) \mathbf{U}_s \frac{- \mathbb{I}_{s > p} +  \mathbb{I}_{s \leq p}}{\sqrt{|\lambda_s|^{1/2}}}\big) > t \bigg) \leq 2 \exp\bigg( -\frac{t^2}{C 2^{-2j} \lambda_d\inv} \bigg),
\end{align*}
for some constant $C$ depending on $M$.  Hence, by the union bound, we have that
\begin{align*}
    \p\bigg[ \frac{1}{\sqrt{n}} \bigg| \sum_{s=1}^{d}  ((\mathfrak{P}_j^f)_s \t \big( (\mathbf{A - P}) \mathbf{U}_s \frac{- \mathbb{I}_{s > p} +  \mathbb{I}_{s \leq p}}{\sqrt{|\lambda_s|^{1/2}}}\big) \bigg| > dt \bigg] &\leq 2d \exp \bigg( - \frac{t^2}{C 2^{-2j} |\lambda_d|\inv} \bigg),
\end{align*}
provided that $t$ is chosen appropriately. By another union bound over the set $S_j$, we have that
\begin{align*}
     \p\bigg[\sup_f \frac{1}{\sqrt{n}} \bigg| \sum_{i=1}^{n}  ((\mathfrak{P}_j^f) \t \big( (\mathbf{A - P}) \mathbf{U} |\Lambda|^{-1/2} \ipq\big)_i \bigg| > dt \bigg] &\leq 2d |S_j| \exp \bigg( - \frac{t^2}{C 2^{-2j} |\lambda_d|\inv} \bigg).
\end{align*}
We note that $|S_j| \leq (C2^{j})^d$ by using the bound on the covering number (e.g. Lemma 2.5 in \citet{van_de_geer_empirical_2009}).  Following the steps from equation A.5 to A.6 in \citet{tang_nonparametric_2017}, by rearranging the equation above, we have that for any $t_j >0$
\begin{align*}
    \p\bigg[ \sup_{f} \bigg| \frac{1}{\sqrt{n}}  \bigg| \sum_{i=1}^{n}  ((\mathfrak{P}_j^f) \t \big( (\mathbf{A - P}) \mathbf{U} |\Lambda|^{-1/2} \ipq\big)_i \bigg| > \eta_j \bigg] &\leq 2d\exp(-t^2_j),
\end{align*}
where $\eta_j = d\sqrt{C 2^{-2j} \lambda_d\inv(t_j^2 + \log|S_{j+1}|^2)}$.  Summing over $j$ and bounding the zeroth order term similarly, we have
\begin{align*}
      \p\bigg[ \sup_{f} | \frac{1}{\sqrt{n}} \tr\big(
    \big((\mathbf{A - P)U|\Lambda^{-1/2}}\ipq)[\mathbf{M}(\partial f)]\t \big) \geq \sum_{j=0}^{\infty} \tilde C \eta_{j}\bigg] \leq 2 d \sum_{j=0}^{\infty} \exp \left(-t_{j}^{2}\right).
\end{align*}
Taking $t_j^2 = 2 (\log(j) + \log(n))$, we have that
\begin{align*}
    \p\bigg[ \sup_{f} | \frac{1}{\sqrt{n}} \tr\big(
    \big((\mathbf{A - P)U|\Lambda^{-1/2}}\ipq)[\mathbf{M}(\partial f)]\t \big) \geq d \lambda_d^{-1/2} (C_1 \sqrt{\log (n)} + C_2)\bigg] \leq \frac{2d C_3}{n^2},
\end{align*}
for some constants $C_1$, $C_2$, and $C_3$.  Hence, combining all this, we see the linear term satisfies
\begin{align*}
    \sup_{f}| \zeta(f) | &\leq C\frac{ \sqrt{\log(n)}}{\sqrt{n\alpha_n}}
\end{align*}
with probability at least $1 - O(n^{-2})$.  We note that the hidden constants and the constants in the bound depend on the diameter of the set $\partial \mathcal{F}$ and the dimension $d$.  

If $(\mathbf{W_*^X})\t \hat X_i$ and $\tilde X_i$ are replaced with $\mathbf{W_*^X} \tilde X_i$ and $\hat X_i$ and respectively, the results continue to hold by noting that $\mathbf{W}_*^X$ has entries bounded by 1 and that appropriate quantities are invariant under orthogonal transformation.

Finally, if $\sqrt{n} \alpha_n \geq \log^{1+\eta}(n)$, then the right hand side still tends to zero with an additional factor of $\sqrt{\alpha_n}$ in the denominator, which is the final assertion of Lemma \ref{lem:fclt}.  
\end{proof}


\begin{proof}[Proof of Lemma \ref{cor:fclt_estimate}]

First, by Lemma \ref{sparsity_lemma}, we have that $\sqrt{\alpha_n}/\sqrt{\hat \alpha_n} \to 1$ in probability (and almost surely).  In the proof of Lemma \ref{lem:fclt}, we have shown that

\begin{align*}
    \sup_{f} \frac{\sqrt{\alpha_n}}{\sqrt{n}} \sum_{i=1}^{n}(\partial f)(\tilde X_i) \frac{(\hat X_i - \mathbf{W_*^X}\tilde X_i)}{\sqrt{\alpha_n}} &\to 0; \\
      \sup_{f} \frac{\sqrt{\alpha_n}}{2\sqrt n} \sum_{i}  \frac{(\hat X_i - \mathbf{W_*^X}\tilde X_i )\t (\partial^2 f)(X_i^*) (\hat X_i -\mathbf{W_*^X} \tilde X_i) }{\alpha_n} &\to 0
\end{align*}
almost surely.  Therefore, we can replace $\sqrt{\alpha_n}$ by $\sqrt{\hat \alpha_n}$ and apply Slutsky's Theorem to conclude that
\begin{align*}
     \sup_{f} \frac{\sqrt{\alpha_n}}{\sqrt{n}} \sum_{i=1}^{n}(\partial f)(\tilde X_i) \frac{(\hat X_i}{\sqrt{\hat\alpha_n}} - \frac{\mathbf{W_*^X}\tilde X_i)}{\sqrt{\alpha_n}} &\to 0;\\
      \sup_{f} \frac{\sqrt{\alpha_n}}{2\sqrt n} \sum_{i}  \frac{(\hat X_i}{\sqrt{\hat\alpha_n}} - \frac{\mathbf{W_*^X}\tilde X_i )\t}{\sqrt{\alpha_n}} (\partial^2 f)(X_i^*) \frac{(\hat X_i}{\sqrt{\hat\alpha_n}} - \frac{\mathbf{W_*^X}\tilde X_i )}{\sqrt{\alpha_n}} &\to 0
\end{align*}
in probability. 
\end{proof}




\subsection{Proofs of Auxiliary Lemmas} \label{sec:proofs_lems}
In this section we prove the additional technical lemmas; namely Lemmas \ref{lem2}, \ref{lem:qtilde}, and \ref{sparsity_lemma}.

\subsubsection{Proofs of Lemmas \ref{lem2} and \ref{lem:qtilde}}
This section contains various results associated to the approximation of the matrix $\mathbf{Q_X}$ to its limiting value. 

\begin{proof}[Proof of Lemma \ref{lem2}]
By Corollary 2 of \citet{agterberg_two_2020}, there exists a sequence of block-orthogonal matrices $\mathbf{W_X}$ such that $\|\mathbf{Q_X - W_X \tilde Q}\|\to 0$.  When $F_X = F_Y \circ\mathbf{T}$, it holds that only replacing $\mathbf{Q_X}$ with $\mathbf{Q_Y T\inv}$, to see that 
\begin{align*}
    \|\mathbf{Q_YT\inv - W_Y \tilde Q }\| \to 0.
\end{align*}
Note that the (block)-orthogonal matrix above need not necessarily be the same due to nonidentifiability of the eigenvectors.  Hence, 
\begin{align*}
    \mathbf{W_Y\t Q_Y T\inv - W_X\t Q_X} \to 0,
\end{align*}
which in particular implies that both terms are tending towards $\mathbf{\tilde Q}$ almost surely.
\end{proof}

\begin{proof}[Proof of Lemma \ref{lem:qtilde}]
Define $\Delta = \E( XX\t)$, and let $\mathbf{\tilde V}$  and $\mathbf{V}$ be the orthogonal matrices in the eigendecomposition of $\bigg(\frac{\mathbf{X\t X}}{n} \bigg)^{1/2} \ipq \bigg(\frac{\mathbf{X\t X}}{n} \bigg)^{1/2}$ and $\Delta^{1/2} \ipq \Delta^{1/2}$ respectively.  Let $\mathbf{\Lambda}$ be the eigenvalues of $\mathbf{X}\ipq \mathbf{X}\t$ and let $\mathbf{\tilde \Lambda}$ be the eigenvalues of $\Delta^{1/2} \ipq \Delta^{1/2}$.  We will first show that with probability at least $1-n^{-2}$, that the following hold simultaneously: 
\begin{align}
    \|\Delta^{1/2} - \left(\frac{\mathbf{X\t X}}{n}\right)^{1/2}\| & = O\left( \frac{\sqrt{\log(n)}}{\sqrt{n}} \right); \label{bound1}\\
    \|\left(\frac{\mathbf{|\Lambda|}}{n}\right)^{-1/2} - \mathbf{|\tilde \Lambda|}^{-1/2} \| & = O\left( \frac{\sqrt{\log(n)}}{\sqrt{n}} \right); \label{bound2} \\
    \|\mathbf{V - \tilde V W}_n\t \|  &= O\left( \frac{\sqrt{\log(n)}}{\sqrt{n}} \right); \label{bound3},
\end{align}
where $\mathbf{W}_n \in \mathbb{O}(p,q)\cap\mathbb{O}(d)$ will be defined later.

First, by Theorem 6.2 in \citet{higham_functions_2008}, for $\mathbf{A}$ and $\mathbf{B}$ positive-definite matrices, we have that
\begin{align}
    \| \mathbf{A}^{1/2} - \mathbf{B}^{1/2} \| \leq \frac{1}{\lambda_{\min}(\mathbf{A})^{1/2} + \lambda_{\min}(\mathbf{B})^{1/2}} \|\mathbf{A} - \mathbf{B}\|. \label{higham}
\end{align}
The result above is stated in \citet{higham_functions_2008} for matrices with distinct eigenvalues.  However, if $\mathbf{A}$ and $\mathbf{B}$ do not have distinct eigenvalues, the result holds by adding small values of $\eps$ to each of the repeated eigenvalues, applying the result for the new slightly perturbed matrices, and then taking the limit as $\eps \to 0$.

By the Law of Large Numbers, $\frac{\mathbf{X\t X}}{n}$ is almost surely positive definite whenever $\Delta$ is.  Applying the above inequality to $\Delta$ and $\frac{\mathbf{X\t X}}{n}$, we see that
\begin{align*}
    \| \Delta^{1/2} - (\frac{\mathbf{X\t X}}{n})^{1/2} \| &\leq \frac{1}{\lambda_{\min}(\Delta)^{1/2} + \lambda_{\min}(\mathbf{\frac{\mathbf{X\t X}}{n}})^{1/2}} \|\Delta - \frac{\mathbf{X\t X}}{n}\| \\
    &\leq \frac{1}{\lambda_{\min}(\Delta)^{1/2}} \|\Delta - \frac{\mathbf{X\t X}}{n}\|.
\end{align*}
By Theorem 6.5 in \citet{wainwright_high-dimensional_2019} (which applies to second moment matrices by shifting by the mean), we have that for some constants $c_1$, $c_2$ and $c_3$, 
\begin{align*}
    \|\Delta - \frac{\mathbf{X\t X}}{n}\| \leq \| \Delta\| c_1 \left\{ \sqrt{\frac{d}{n}} + \frac{d}{n} \right\} + \delta 
\end{align*}
with probability at least $1 - c_3\exp( -c_2 n \min(\delta,\delta^2)).$  Taking $\delta = \sqrt{\frac{2\log(n)}{c_3 n}}$, we see that 
\begin{align*}
    \|\Delta - \frac{\mathbf{X\t X}}{n}\| = O\left(\| \Delta\| \sqrt{\frac{d}{n}} \sqrt{\log(n)}\right)
\end{align*}
with probability at least $1 - O(n^{-2})$.  Putting it all together and noting $\mathbf{\|\Delta\|}$ and $d$ are constants in $n$, we arrive at
\begin{align*}
    \|\Delta^{1/2} - \frac{\mathbf{X\t X}}{n}^{1/2} \| &\leq  O\left( \frac{\sqrt{\log(n)}}{\sqrt{n}} \right).  
\end{align*}
which proves \eqref{bound1}.  

For \eqref{bound2}, we note that $\mathbf{\Lambda}$ and $\mathbf{\tilde \Lambda}/n$ are the eigenvalues of the matrix $\Delta^{1/2} \ipq \Delta^{1/2}$ and $\bigg(\frac{\mathbf{X\t X}}{n}\bigg)^{1/2} \ipq \bigg( \frac{\mathbf{X\t X}}{n}\bigg)^{1/2} $ respectively.  To see the latter, we note that $\mathbf{X}\ipq \mathbf{X}\t$ has the same nonzero eigenvalues as $(\mathbf{X\t X})^{1/2} \ipq (\mathbf{X\t X})^{1/2}$ by similarity.  By Weyl's inequality, we have that
\begin{align*}
    |\lambda_i - \tilde \lambda_i| &\leq \|\frac{\mathbf{X\t X}}{n}^{1/2} \ipq \frac{\mathbf{X\t X}}{n}^{1/2}  - \Delta^{1/2} \ipq \Delta^{1/2} \| \\
    &\leq \|\frac{\mathbf{X\t X}}{n}^{1/2} \ipq \Delta^{1/2} - \Delta^{1/2} \ipq \Delta^{1/2} \| +  \|\frac{\mathbf{X\t X}}{n}^{1/2} \ipq \Delta^{1/2} - \frac{\mathbf{X\t X}}{n}^{1/2} \ipq \frac{\mathbf{X\t X}}{n}^{1/2} \| \\
    &\leq \|\frac{\mathbf{X\t X}}{n}^{1/2} - \Delta^{1/2} \| \| \Delta^{1/2}   \| + \|\frac{\mathbf{X\t X}}{n}^{1/2}  \|  \| \Delta^{1/2} - \frac{\mathbf{X\t X}}{n}^{1/2} \| \\
    &= O\left( \frac{\sqrt{\log(n)}}{\sqrt{n}} \right)
\end{align*}
by \eqref{bound1} and the fact that $\frac{\mathbf{X\t X}}{n}$ can be bounded above by a constant.  Now, define the function $g(\lambda) := \frac{1}{|\lambda|^{1/2}}$.  Since $\Delta$ is full-rank and $\ipq$ is orthogonal, so is $\Delta \ipq$ and hence $\Delta^{1/2} \ipq \Delta^{1/2}$ by similarity.  Therefore, there exists an $\eps > 0$ such that all eigenvalues of $\Delta^{1/2} \ipq \Delta^{1/2}$ are outside the range $(-\eps, \eps)$, and hence so are those of $(\frac{\mathbf{X\t X}}{n})^{1/2} \ipq (\frac{\mathbf{X\t X}}{n})^{1/2}$ for $n$ sufficiently large with probability at least $1- n^{-2}$. The function $g(\lambda)$ is differentiable outside of $(-\eps,\eps)$, and hence by the Delta method applied to each eigenvalue individually, $g(\lambda_i) - g(\tilde \lambda_i) = O\left( \frac{\sqrt{\log(n)}}{\sqrt{n}} \right),$ for $n$ sufficiently large with probability $1- O(n^{-2})$, where the hidden constant depends on $g'(\tilde \lambda_i)$.  This proves \eqref{bound2}.

For \eqref{bound3}, we simply apply the Davis-Kahan Theorem to the eigenvectors associated to each eigenvalue.  First, consider $i$ such that $\tilde\lambda_i$ is unique.  By \eqref{bound2}, for $n$ sufficiently large, the eigenvalues outside of $i$ are separated from each other and we can apply the Davis-Kahan Theorem.  We see that
\begin{align*}
    \|\mathbf{v}_i - \mathbf{\tilde v}_i\| &\leq C\frac{\|\frac{\mathbf{X\t X}}{n}^{1/2} \ipq \frac{\mathbf{X\t X}}{n}^{1/2}  - \Delta^{1/2} \ipq \Delta^{1/2} \| }{\delta_{gap}^{(i)}} \\
    &= O\left( \frac{\sqrt{\log(n)}}{\sqrt{n}} \right),
\end{align*}
where the hidden constant depends on $\delta_{gap}^{(i)} := \min(\lambda_{i+1}-\lambda_i, \lambda_i - \lambda_{i-1})$, though this is a deterministic value depending only on $\Delta$.  For non-unique eigenvalues, we apply the same argument, only with an orthogonal matrix attached to the $\mathbf{V}_i$.  We note that the orthogonal matrix is chosen only for each group of repeated eigenvalues, and hence the combined matrix is block-orthogonal.  

We are now ready to prove the result in Equation \ref{bound4}.  By \citet{agterberg_two_2020}, we can write
\begin{align*}
    \mathbf{Q_X} &= \left(\frac{|\mathbf{\Lambda}|}{n}\right)^{-1/2} \mathbf{V}\t \left( \frac{\mathbf{X\t X}}{n}\right)^{1/2} \\
    \mathbf{\tilde Q} &= \left(|\mathbf{\tilde \Lambda}|\right)^{-1/2} \mathbf{\tilde V}\t \left( \Delta\right)^{1/2}.
\end{align*}
The result follows by using \eqref{bound1}, \eqref{bound2}, and \eqref{bound3} together and adding and subtracting terms and noting that by construction the orthogonal matrix from \eqref{bound3} commutes with $\left(|\mathbf{\tilde \Lambda}|\right)^{-1/2}$.  

We note that the matrix $\mathbf{Q_X}$ is invariant to the sparsity factor, since the eigenvectors of $\alpha_n\mathbf{P} =\alpha_n\mathbf{X}\ipq \mathbf{X\t}$ are the same as those of $\mathbf{P}$ and the eigenvalues are scaled by $\alpha_n$ so that if $\mathbf{D}$ are the eigenvalues of $\alpha_n\mathbf{ P}$ then
\begin{align*}
\mathbf{U_X |D|^{1/2}} &= \sqrt{\alpha_n} \mathbf{U_X |\Lambda_X|^{1/2}} \\
&= \sqrt{\alpha_n} \mathbf{X Q_X\inv},
\end{align*}
showing that the matrix $\mathbf{Q_X}$ depends only on the matrix $\mathbf{P}$ and not the sparsity component $\alpha_n$.  
\end{proof}

\subsubsection{Proof of Lemma \ref{sparsity_lemma}}\label{sec:proofs_aux}
\begin{proof}[Proof of Lemma \ref{sparsity_lemma}]
First, for any fixed matrix $\mathbf{X}$, we have that the $\mathbf{A}_{ij}$'s are independent random variables, and recall that
\begin{align*}
    \hat \alpha_n = \frac{1}{{{n}\choose 2}} \sum_{i < j} A_{ij}.
\end{align*}
Define $\theta_n := \frac{\alpha_n}{{{n}\choose 2}}\sum_{i< j} \mathbf{P}_{ij}$.  Then $\E( \hat \alpha_n | \mathbf{X}) = \theta_n$.  
Therefore, we have that
\begin{align*}
      \p\bigg( | \hat \alpha_n - \alpha_n | > 2t \bigg) \leq \p ( | \hat \alpha_n - \theta_n | > t) + \p( | \theta_n/\alpha_n - 1 | > t /\alpha_n ).
\end{align*}
For the first, term, we note that by applying Hoeffding's inequality, we see that
\begin{align*}
    \p \bigg( |\hat\alpha_n - \theta_n | \geq t \bigg) &\leq 2 \exp\bigg( -2 {{n}\choose 2} t^2 \bigg) \\
    &= 2 \exp \bigg( -( n^2 - n)t^2 \bigg)\\
    &\leq 2 \exp \bigg( -\frac{n^2 t^2}{2} \bigg).
\end{align*}
For the second term, note that 
$\frac{1}{{{n}\choose 2}} \sum_{i< j} X_i\t \ipq X_j $ is a $U$-statistic with expected value 1/2.  Hoeffding's inequality for $U$-statistics (e.g. Example 2.23 in \citet{wainwright_high-dimensional_2019}) shows that
\begin{align*}
    \p\bigg( | \frac{1}{{{n}\choose 2}} \sum_{i<j} X_i\t \ipq X_j  - 1 | \geq t \bigg) &\leq 2 \exp( -nt^2 / 8).
\end{align*}
Hence, we see that 
\begin{align*}
    \p\bigg( | \hat \alpha_n - \alpha_n | > 2t \bigg) &\leq 2 \exp \bigg( -\frac{n^2 t^2}{2} \bigg) + 2 \exp( -\frac{nt^2}{8\alpha_n^2})
\end{align*}
Set $t = 4  \sqrt{\frac{\alpha_n\log(n)}{n}}$.  Then recalling that for some $C > 0$ $1\geq \alpha_n \geq C \log^4(n)/n$, we have
\begin{align*}
    \p\bigg( | \hat \alpha_n - \alpha_n | > 8  \sqrt{ \frac{\alpha_n\log(n)}{n}} \bigg) &\leq  2 \exp \bigg( -8 n \alpha_n \log(n) \bigg) + 2\exp\bigg( -2\log(n) \alpha_n\inv\bigg) \\
    &\leq 2 \exp \bigg( -8C \log^5(n) \bigg) + 2n^{-2} \\
    &\leq 4n^{-2}.
\end{align*}

Now, define the event $\mathcal{A} := \{|\hat \alpha_n -\alpha_n| \leq 4  \sqrt{\frac{\alpha_n\log(n)}{n}} \}$.  On $\mathcal{A}$, since $\alpha_n \geq C \frac{\log^4(n)}{n}$,
\begin{align*}
    \frac{|\hat \alpha_n - \alpha_n|}{\alpha_n} \leq  \frac{4}{\log^{1.5}(n)},
\end{align*}
which is small for $n$ sufficiently large.  Hence, by Taylor expansion, we have that
\begin{align*}
    \frac{1}{\sqrt{\hat \alpha_n}} &= \frac{1}{\sqrt{ \alpha_n + (\hat \alpha_n - \alpha_n)}} \\
    &= \frac{1}{\sqrt{\alpha_n}} \bigg( \sqrt{ 1 + \frac{\hat\alpha_n - \alpha_n}{\alpha_n}} \bigg)\inv \\
    &= \frac{1}{\sqrt{\alpha_n}} \bigg( 1 + \frac{1}{2}\frac{\hat\alpha_n - \alpha_n}{\alpha_n} + O\bigg( \frac{\hat\alpha_n - \alpha_n}{\alpha_n}\bigg)^2 \bigg).
\end{align*}
By the previous observations, we have that this is equal to
\begin{align*}
    \frac{1}{\sqrt{\alpha_n}} \bigg( 1 + O\bigg[ \sqrt{\frac{\log(n)}{n\alpha_n}}\bigg]\bigg).
\end{align*}
for $n$ sufficiently large.  
\end{proof}

\section{More on the Discussion in Remark \ref{sec:close_eigs_discussion}}\label{sec:close_eigs}
In this section, for $f$ and $g$ two functions of $n$, we write $f(n) \ll g(n)$ if $f(n)/g(n) \to 0$ as $n$ tends to infinity.

The eigenvalues of $\alpha_n \mathbf{X}\mathbf{X}\t$ are the same as those of $\alpha_n \mathbf{X}\t  \mathbf{X}$, and as $n \to \infty$, the matrix $\frac{1}{n}\mathbf{X\t X}$ is converging almost surely to $\E(XX\t)$.  Therefore, as $n \to \infty$,
\begin{align*}
    \frac{1}{n\alpha_n}\big(\lambda_{i}(\mathbf{P}) - \lambda_{i+1}(\mathbf{P}) \big) \to \delta_{i},
\end{align*}
where $\delta_i = \lambda_{i}(\E(XX\t)) - \lambda_{i+1}(\E(XX\t))$.   We have
\begin{align*}
    \| \mathbf{W}_* - \mathbf{I}\|_F &\leq  \| \mathbf{W}_* - \mathbf{U_A\t U_P}\|_F + \| \mathbf{U_A\t U_P}- \mathbf{I}\|_F.
\end{align*}
The first term can be bounded directly using the Davis-Kahan Theorem as
\begin{align*}
     \| \mathbf{W}_* - \mathbf{U_A\t U_P}\|_F = O\bigg( \frac{ \|\mathbf{A - P}\|}{\lambda_d(\mathbf{P})} \bigg)^2 &= O( (n\alpha_n)\inv).
\end{align*}
For the second term, following the analysis on Page 24 of \citet{rubindelanchy_statistical_2022}, we have that for $i \neq j$,
\begin{align*}
     (\mathbf{U_A\t U_P} )_{ij} &=  -\frac{(\mathbf{U_A})_i\t (\mathbf{A - P}) (\mathbf{U_P})_j}{\lambda_j(\mathbf{P})-\lambda_i(\mathbf{A})}.
\end{align*}
By previous results on eigenvalue concentration (e.g. \citet{eldridge_unperturbed_2018,orourke_random_2018,cape_kato--temple_2017}), we have that
\begin{align*}
|\lambda_i(\mathbf{A}) - \lambda_i(\mathbf{P})| \leq C \log(n)
\end{align*}
with high probability.  Hence, the above bound can be written as
\begin{align*}
\frac{(\mathbf{U_A})_i\t (\mathbf{A - P}) (\mathbf{U_P})_j}{\lambda_j(\mathbf{P})-\lambda_i(\mathbf{A})} &= \frac{(\mathbf{U_A})_i\t (\mathbf{A - P}) (\mathbf{U_P})_j}{\lambda_j(\mathbf{P})-\lambda_i(\mathbf{P}) \pm C\log(n)}
\end{align*}
Moreover, by \citet{koltchinskii_random_2000}, we have that $\frac{\lambda_i(\mathbf{P})}{n\alpha_n} - \lambda_i(\E(XX\t)) = O_{\p}(n^{-1/2})$.  Hence, we can further simplify the bound to when $i = j+1$
\begin{align*}
    \frac{(\mathbf{U_A})_i\t (\mathbf{A - P}) (\mathbf{U_P})_j}{\lambda_j(\mathbf{P})-\lambda_i(\mathbf{P}) \pm C\log(n)} &= \frac{(\mathbf{U_A})_i\t (\mathbf{A - P}) (\mathbf{U_P})_j}{n\alpha_n \lambda_j(\E(XX\t)) - n\alpha_n\lambda_i(\E(XX\t)) \pm O_{\p}(\sqrt{n}\alpha_n)} \\
    &= \frac{(\mathbf{U_A})_i\t (\mathbf{A - P}) (\mathbf{U_P})_j}{n\alpha_n \delta   \pm O_{\p}(\sqrt{n}\alpha_n) \pm O(\log(n))}.
\end{align*}
Expanding the numerator, we see that we can write this via
\begin{align*}
    \bigg( n\alpha_n \delta &  \pm O_{\p}(\sqrt{n}\alpha_n) \pm O(\log(n)) \bigg)\inv \bigg((\mathbf{U_A})_i\t (\mathbf{A - P}) (\mathbf{U_P})_j\bigg) \\
    &=  \bigg( n\alpha_n \delta   \pm O_{\p}(\sqrt{n}\alpha_n) \pm O(\log(n)) \bigg)\inv\bigg[ (\mathbf{U_A})_i\t (\mathbf{U_P U_P\t})  (\mathbf{A - P}) (\mathbf{U_P})_j \bigg] \\&\quad +\bigg( n\alpha_n \delta   \pm O_{\p}(\sqrt{n}\alpha_n) \pm O(\log(n)) \bigg)\inv\bigg[ (\mathbf{U_A})_i\t (\mathbf{I} - \mathbf{U_P U_P\t})  (\mathbf{A - P}) (\mathbf{U_P})_j \bigg].
\end{align*}
The first term $\mathbf{U_P\t (A-P)(U_P})_j$ is a sum of independent random variables, so Hoeffding's inequality reveals that it is of order at most $\log(n)$ with high probability. The second term can be bounded via the Davis-Kahan theorem as
\begin{align*}
  \bigg[ (\mathbf{U_A})_i\t (\mathbf{I} - \mathbf{U_P U_P\t})  (\mathbf{A - P}) (\mathbf{U_P})_j \bigg] &\leq   \| (\mathbf{U_A})_i\t (\mathbf{I} - \mathbf{U_P U_P\t})\| \|  (\mathbf{A - P})  \| \\
  &\leq \frac{C \|\mathbf{A - P}\|}{\lambda_d(\mathbf{P})} \| (\mathbf{A - P})  \| \\
  &= O(1).
\end{align*}
Putting it together, we arrive at
\begin{align*}
     \| \mathbf{U_A\t U_P}- \mathbf{I}\|_F &= O\bigg( \frac{\log(n)}{n\alpha_n\delta} \bigg),
\end{align*}
where $\delta = \min_{i} \delta_i$, provided $\sqrt{n}\alpha_n \ll n\alpha_n \delta$ and $\log(n)\ll n\alpha_n \delta$.  

From an asymptotic standpoint, the term $\delta$ in the denominator makes little difference as it is a constant, and $n\alpha_n \to \infty$.  However, for finite $n$, depending on $\alpha_n$ and $n$, even if $n\alpha_n \gg \log^4(n)$, the constant may depend on $\delta$.  Therefore, for any fixed model, though the rate of convergence depends on $n\alpha_n$, for all practical purposes it also depends on the eigengap $\delta$.

\bibliographystyle{plainnat_JA} 
\bibliography{nonpar_grdrpg.bib,gretton.bib}

\end{document}